\documentclass[11pt]{article}
\usepackage{amsmath,amssymb}
\usepackage{amsthm}
\usepackage{pifont}
\usepackage{mathrsfs}
\usepackage{bm}
\usepackage{ae}
\usepackage{tikz}
\usetikzlibrary{decorations.pathreplacing, calc}
\tikzstyle{braket}=[decorate,decoration={brace,amplitude=7pt},xshift=0pt,yshift=-10pt, black]

\usepackage{subcaption}
\usepackage{xcolor}

\usepackage{geometry}
\usepackage[colorlinks=true,
linkcolor=blue!80,citecolor=blue,
urlcolor=blue]{hyperref}

\makeatletter
\def\@seccntDot{.}
\def\@seccntformat#1{\csname the#1\endcsname\@seccntDot\hskip 0.5em}
\renewcommand\section{\@startsection{section}{1}{\z@}%
{18\p@ \@plus 6\p@ \@minus 3\p@}%
{9\p@ \@plus 6\p@ \@minus 3\p@}%
{\large\bfseries\boldmath}}
\renewcommand\subsection{\@startsection{subsection}{2}{\z@}%
{12\p@ \@plus 6\p@ \@minus 3\p@}%
{3\p@ \@plus 6\p@ \@minus 3\p@}%
{\it}}
\renewcommand\subsubsection{\@startsection{subsubsection}{3}{\z@}%
{12\p@ \@plus 6\p@ \@minus 3\p@}%
{\p@}%
{\boldmath}}
\makeatother

\usepackage{microtype}

\theoremstyle{plain}
\newtheorem{theorem}{Theorem}[section]
\newtheorem{lemma}{Lemma}[section]
\newtheorem{corollary}{Corollary}[section]
\newtheorem{proposition}{Proposition}[section]
\newtheorem{conjecture}{Conjecture}[section]

\theoremstyle{definition}

\newtheorem{remark}{Remark}[section]
\newtheorem{example}{Example}[section]
\newtheorem{claim}{Claim}[section]

\numberwithin{equation}{section}
\allowdisplaybreaks
\newcommand{\imgi}{{\bf i}}
\newcommand{\e}{{\bf e}}
\newcommand{\ls}{{\bf LS}}

\title{Extremal spectral radius of nonregular graphs with prescribed maximum degree}
\author{
Lele Liu\thanks{College of Science, University of Shanghai for Science and Technology, Shanghai 200093, China
(\texttt{ahhylau@outlook.com})}
}

\date{}

\begin{document}
\maketitle

\begin{abstract}
Let $G$ be a graph attaining the maximum spectral radius among all connected nonregular graphs of 
order $n$ with maximum degree $\Delta$. Let $\lambda_1(G)$ be the spectral radius of $G$. 
A nice conjecture due to Liu, Shen and Wang [On the largest eigenvalue of non-regular graphs,
J. Combin. Theory Ser. B, 97 (2007) 1010--1018] asserts that
\[
\lim_{n\to\infty} \frac{n^2(\Delta-\lambda_1(G))}{\Delta-1} = \pi^2
\]
for each fixed $\Delta$. Concerning an important structural property of the extremal graphs $G$, Liu and 
Li present another conjecture which states that $G$ has degree sequence $\Delta,\ldots,\Delta,\delta$. 
Here, $\delta=\Delta-1$ or $\delta=\Delta-2$ depending on the parity of $n\Delta$. In this paper, we 
make progress on the two conjectures. To be precise, we disprove the first conjecture for all $\Delta\geq 3$ 
by showing that the limit superior is at most $\pi^2/2$. For small $\Delta$, we determine the precise 
asymptotic behavior of $\Delta-\lambda_1(G)$. In particular, we show that 
$\lim\limits_{n\to\infty} n^2 (\Delta - \lambda_1(G)) /(\Delta - 1) = \pi^2/4$ if $\Delta=3$; and 
$\lim\limits_{n\to\infty} n^2 (\Delta - \lambda_1(G)) /(\Delta - 2) = \pi^2/2$ if $\Delta = 4$. We also 
confirm the second conjecture for $\Delta = 3$ and $\Delta = 4$ by determining the precise structure of 
extremal graphs. Particularly, we show that the extremal graphs for $\Delta\in\{3,4\}$ must have a 
path-like structure built from specific blocks.
\par\vspace{2mm}

\noindent{\bfseries Keywords:} Spectral radius; Nonregular graph; Maximum degree.
\par\vspace{1mm}

\noindent{\bfseries AMS Classification:} 05C35; 05C50; 15A18.
\end{abstract}

\section{Introduction}
\label{sec:introduction}

% One of the best known classes of graphs is regular graphs, which have been studied extensively 
% in a variety of contexts. For an irregular graph in which not all vertices have equal degrees,
% it can be viewed as somehow deviating from regularity. In this talk, we report some recent 
% results in extremal spectral graph theory concerning the irregularity of graphs.

One of the best known classes of graphs is regular graphs, which have been studied extensively 
in a variety of contexts. For a nonregular graph in which not all vertices have equal degrees,
it can be viewed as somehow deviating from regularity. There are several measures on graphs 
which are used to determine how close a given graph is to being regular. One such measure is 
the difference between the maximum degree and the largest eigenvalue of a graph. To account 
for this, we denote by $\lambda_1(G)$ the largest eigenvalue of the adjacency matrix $A(G)$ 
of a graph $G$, which is also called the spectral radius of $G$. We also denote by $\Delta(G)$ 
and $\delta(G)$ the maximum degree and minimum degree of $G$, respectively. Given a connected 
graph $G$, it is well-known that $\lambda_1(G)\leq\Delta(G)$, with equality if and only if $G$ 
is regular. Hence, this fact allows us to consider the difference $\Delta(G) - \lambda_1(G)$
as a relevant measure of irregularity of a graph $G$ (see \cite[p.\,242]{CvetkovicRowlinson2009}).
Also, it is natural to ask how small $\Delta(G) - \lambda_1(G)$ can be when $G$ is nonregular. 
In the past decades the study of this parameter of nonregular graphs has attracted the interests 
of a large number of scholars.

Let $G$ be a connected nonregular graph on $n$ vertices with maximum degree $\Delta$ and minimum 
degree $\delta$. In 2004, Stevanovi\'c \cite{Stevanovic2004} first proved that
\[
\Delta-\lambda_1(G)>\frac{1}{2n(n\Delta-1)\Delta^2}.
\]
He also asked whether or not the power of $\Delta$ appearing in the fraction could be improved.
Using the ratio of components of the Perron vector, Zhang \cite{Zhang2005} obtained a finer
bound as follows
\[
\Delta-\lambda_1(G)>\frac{(\sqrt{\Delta}-\sqrt{\delta})^2}{nD\Delta},
\]
where $D$ is the diameter of $G$. Correlatively, Alon and Sudakov \cite{AlonSudakov2000} proved 
that for a connected nonbipartite graph $G$,
\[
\Delta + \lambda_n(G) > \frac{1}{n(D+1)},
\]
where $\lambda_n(G)$ is the least eigenvalue of $A(G)$. By taking the number of edges $m$ of 
$G$ into account, Cioab\u a, Gregory and Nikiforov \cite{CioabaGregoryNikiforov2007} further 
improved Zhang's bound \cite{Zhang2005}, which stated that 
\[
\Delta-\lambda_1(G) > \frac{n\Delta-2m}{n(D(n\Delta-2m)+1)} \geq\frac{1}{n(D+1)}.
\]
Let us remark that the above inequalities also refine Alon-Sudakov's bound 
\cite{AlonSudakov2000} as $\lambda_1(G) \geq - \lambda_n(G)$. Moreover, the 
authors \cite{CioabaGregoryNikiforov2007} also conjectured that $\Delta-\lambda_1(G)>1/(nD)$
for all connected nonregular graphs with maximum degree $\Delta$. This was subsequently 
proved by Cioab\u a \cite{Cioaba2007}. In 2007, Liu, Shen and Wang \cite{LiuShenWang2007}
proved a lower bound of $\Delta(G) - \lambda_1(G)$ in terms of $\Delta$ and $n$ as follows
\[
\Delta-\lambda_1(G) \geq \frac{\Delta + 1}{n(3n + 2\Delta - 4)},
\]
which is asymptotically best possible up to a constant factor. Later, Liu, Huang and 
You \cite{LiuHuangYou2009} presented a slight improvement for the aforementioned bound.
Shi \cite{Shi2009} established another strong inequality by introducing more parameters as follows 
\[
\Delta-\lambda_1(G) > \Big[(n-\delta) D + \frac{1}{\Delta-2m/n} - \binom{D}{2}\Big]^{-1},
\]
which improved Cioab\u a's bound \cite{Cioaba2007} in some cases. Recently, Zhang \cite{Zhang2021} 
and Feng-Zhang \cite{FengZhang2021} furthermore improve Cioab\u a's result \cite{Cioaba2007} 
in different forms. For more results on this topic, we refer to 
\cite{ChenHou2014,NingLiLu2013,NingLuWangJiang2018,ShiuHuangSun2017} and references therein for details.
It is worth mentioning that any lower bound on $\Delta-\lambda_1(G)$ also gives an upper bound 
on $\lambda_1(G)$.

Let $\mathcal{G}(n,\Delta)$ denote the set of graphs attaining the maximum spectral radius 
among all connected nonregular graphs with $n$ vertices and maximum degree $\Delta$, and let
$\lambda_1(n,\Delta)$ denote the maximum spectral radius. For a graph $G\in\mathcal{G}(n,\Delta)$, 
Liu, Shen and Wang \cite{LiuShenWang2007} investigated the order of magnitude of 
$\Delta-\lambda_1(G)$, and posed the following conjecture.

\begin{conjecture}[\cite{LiuShenWang2007}]\label{conj:limit}
Let $G\in\mathcal{G}(n,\Delta)$. For each fixed $\Delta$, the limit of 
$n^2(\Delta-\lambda_1(G))/(\Delta-1)$ exists. Furthermore, 
\[
\lim_{n\to\infty} \frac{n^2(\Delta-\lambda_1(G))}{\Delta-1} = \pi^2.
\]
\end{conjecture}

\begin{remark}
This conjecture is trivially true for $\Delta = 2$. Indeed, if $\Delta = 2$, then $G$ is 
necessarily a path and $\lambda_1(G) = 2\cos\pi/(n+1)$. Obviously, $2-2\cos\pi/(n+1)$ is 
asymptotic to $\pi^2/n^2$. Therefore, \autoref{conj:limit} holding for $\Delta=2$.
\end{remark}

The following example shows that \autoref{conj:limit} no longer hold if the condition on 
$\Delta$ was dropped. 

\begin{example}
Let $\Delta=n-1$. Then $G$ must be the graph obtained by removing an edge from the complete 
graph $K_n$. Evidently, 
\[ 
\lambda_1(G) = \frac{n-3+\sqrt{n^2+2n-7}}{2}.
\]
By direct computation we obtain
\[
\lim_{n\to\infty} \frac{n^2(\Delta-\lambda_1(G))}{\Delta-1}=2.
\]
Therefore the condition that $\Delta$ is fixed is crucial in \autoref{conj:limit}. 
\end{example}

Intuitively, the graphs attaining the maximum spectral radius among all connected nonregular 
graphs with prescribed maximum degree must be close to regular graphs. In particular, Liu 
and Li \cite{LiuLi2008} posed the following conjecture.

\begin{conjecture}[\cite{LiuLi2008}]\label{conj:degree-sequence}
Let $3\leq\Delta\leq n-2$ and $G\in\mathcal{G}(n,\Delta)$. Then $G$ has degree sequence 
$(\Delta,\ldots,\Delta,\delta)$, where
\[
\delta =
\begin{cases}
\Delta-1, & n\Delta\ \text{is odd}, \\
\Delta-2, & n\Delta\ \text{is even}.
\end{cases}
\]
\end{conjecture}

Although the statement of \autoref{conj:degree-sequence} is intuitive, it appears difficult 
to prove or disprove even for small $\Delta$. An indication to the difficulty of this 
conjecture, as well as \autoref{conj:limit}, is that graphs having bounded degree are sparse 
graphs, whose spectral radius are bounded by a constant. Hence, numerous tools from spectral 
graph theory are ineffective. In this paper we make the first progress on these two conjectures. 
More precisely, we disprove \autoref{conj:limit} for all $\Delta\geq 3$, and confirm 
\autoref{conj:degree-sequence} for $\Delta=3$ and $\Delta=4$. In addition, we determine the 
exact structure of extremal graphs in $\mathcal{G}(n,\Delta)$, as well as the leading term of 
$\Delta-\lambda_1(n,\Delta)$ for $\Delta\in\{3,4\}$. Particularly, we show that the extremal 
graphs must have a path-like structure for $\Delta\in\{3,4\}$. To prove our results, we use the 
local switching operations (see in Section \ref{sec:preliminaries}) and forbidden induced subgraphs 
to deduce structural properties of the extremal graphs. The original ideas come from 
\cite{AbdiGhorbaniImrich2021, AbdiGhorbani2020, BrandGuiduliImrich2007, Guiduli1997},
which study the minimum algebraic connectivity of regular graphs.

The present paper is built up as follows. Preliminary definitions and notation are collected 
in Section \ref{sec:preliminaries}. In Section \ref{sec:general-results}, we present some 
general results on the extremal graphs that will be used frequently in the sequel. We 
prove \autoref{conj:degree-sequence} and determine the structure of extremal graphs for 
$\Delta=3$ and $\Delta=4$ in Section \ref{sec:structure-Delta=3} and Section \ref{sec:structure-Delta=4}, 
respectively. In Section \ref{sec:disprove-conj-1-1}, we give an upper bound of $\Delta-\lambda_1(n,\Delta)$, 
which allows us to disprove \autoref{conj:limit}. In addition, we determine the precise asymptotic 
behavior of $\Delta-\lambda_1(n,\Delta)$ for $\Delta\in\{3,4\}$. We conclude this paper in 
Section \ref{sec:concluding-remarks} with some remarks and open problems.

\section{Preliminaries}
\label{sec:preliminaries}

In this section we introduce notation and preliminary lemmas that will be used in the sequel.

Throughout this paper we consider only simple graphs, i.e, undirected graphs without multiple 
edges or loops. Given a subset $X$ of the vertex set $V(G)$ of a graph $G$, the subgraph of 
$G$ induced by $X$ is denoted by $G[X]$, and the graph obtained from $G$ by deleting $X$ is 
denoted by $G\setminus X$. As usual, for a vertex $v$ of $G$ we write $d_G(v)$ and $N_G(v)$
for the degree of $v$ and the set of neighbors of $v$ in $G$, respectively. If the underlying 
graph $G$ is clear from the context, simply $d(v)$ and $N(v)$. Let $N_X(v)$ denote the set of 
vertices in $X$ that adjacent to $v$, i.e., $N_X(v)=N_G(v)\cap X$. Given two vertices 
$u$, $v$, we use $u\sim v$ (resp. $u\nsim v$) to indicate that vertices $u$ and $v$ are adjacent 
(resp. nonadjacent). For a positive integer $n$, let $[n]$ denote the set $\{1,2,\ldots,n\}$.
Let $\bm{x}\in\mathbb{R}^n$ be a vector, we use $x_{\max}$ and $x_{\min}$ to denote the maximum
and minimum components of $\bm{x}$, respectively. 

The Perron--Frobenius theorem implies that the adjacency matrix $A(G)$ of a connected graph $G$ 
has a positive eigenvector corresponding to $\lambda_1(G)$, and this is called the {\em Perron vector} 
of $G$. Let $\bm{x}$ and $\Delta$ be the Perron vector and maximum degree of $G$, respectively. 
A short argument shows that
\begin{equation}\label{eq:Delta-rho-sum-form}
(\Delta-\lambda_1(G))\cdot\|\bm{x}\|_2^2 
= \sum_{v\in V(G)} (\Delta - d(v)) x_v^2 + \sum_{uv\in E(G)} (x_u-x_v)^2.
\end{equation}
Recall that the Laplacian matrix of $G$ is defined as $L(G):= D(G)-A(G)$, where $D(G)$ is 
the diagonal matrix whose diagonal entries are the vertex degrees of $G$. Hence, by Rayleigh 
principle, for any vector $\bm{y}\in\mathbb{R}^{|V(G)|}$ we have
\begin{equation}\label{eq:Delta-rho}
(\Delta - \lambda_1(G)) \cdot \|\bm{y}\|_2^2 \leq
\sum_{v\in V(G)} (\Delta - d(v)) y_v^2 + \bm{y}^{\mathrm{T}} L(G) \bm{y}.
\end{equation}
On the other hand, noting that $A(G) \bm{x} = \lambda_1(G) \bm{x}$, we have 
$\bm{1}^{\mathrm{T}}A(G) \bm{x} = \lambda_1(G) \bm{1}^{\mathrm{T}} \bm{x}$, 
where $\bm{1}$ is the all-ones vector. That is,
\[
\sum_{v\in V(G)} d(u) x_v = \lambda_1(G)\sum_{v\in V(G)} x_v.
\]
We immediately obtain that
\begin{equation}\label{eq:identity-sum-xi}
\sum_{v\in V(G)} (\Delta - d(v)) x_v = (\Delta - \lambda_1(G)) \sum_{v\in V(G)} x_v.
\end{equation}

The following two local operations on edges of graphs are well-known, from which one can obtain
perturbation results in spectral radius under edge operations.

\begin{lemma}[\cite{CvetkovicRowlinson2009}, Theorem 8.1.3]\label{lem:rotation}
Let $G$ be a connected graph with $uv\in E(G)$ and $uw\notin E(G)$. Let 
$\widetilde{G}:=G+uw-uv$ and $\bm{x}$ be the Perron vector of $G$. If 
$x_w\geq x_v$, then $\lambda_1(\widetilde{G})>\lambda_1(G)$.
\end{lemma}

Given a graph $G$, the {\em local switching} is the replace of a pair of edges $uv$ and $st$  
in $G$ by the edges $sv$ and $tu$, which is denoted by $\ls (s,t,v,u)$. In addition, we use 
$\ls (G; s,t,v,u)$ to denote the resulting graph obtained by the local switching operation 
$\ls (s,t,v,u)$. Note that local switching preserves degrees.

\begin{lemma}[\cite{CvetkovicRowlinson2009}, Theorem 8.1.10]\label{lem:local-switching}
Let $G$ be a connected graph, and $\widetilde{G} = \ls (G; s,t, v,u)$. If $(x_s-x_u)(x_v-x_t)\geq 0$, 
then $\lambda_1(\widetilde{G})\geq\lambda_1(G)$, with equality if and only if $x_s=x_u$ and $x_v=x_t$.  
\end{lemma}

\section{General results for graphs in $\mathcal{G}(n,\Delta)$}
\label{sec:general-results}

In this section, we present some general results on the extremal graphs in $\mathcal{G}(n,\Delta)$ 
that will be useful at various points in this paper. In order to state our results, we should 
first introduce more symbol. Let $G\in\mathcal{G}(n,\Delta)$. We denote 
\[
S = \{v\in V(G): d(v)<\Delta\},~~
T = \{v\in V(G): d(v)=\Delta\}.
\]
Throughout this section and the next two sections, we always assume that $G\in\mathcal{G}(n,\Delta)$
and $\bm{x}$ is the Perron vector of $G$ with $\|\bm{x}\|_2=1$. Denote $V(G):=\{v_1,v_2,\ldots,v_n\}$, 
and write $x_i := x_{v_i}$ for brevity. For convenience, we assume $\bm{x}$ is decreasing, i.e., 
$x_{1}\geq x_2\geq\cdots\geq x_n$.

\subsection{The size of $S$ is small}

The proof of the following lemma can be found in \cite{LiuLi2008} and \cite{LiuLiuYou2009}, we include 
it here in order to keep this paper complete and self-contained.

\begin{lemma}[\cite{LiuLi2008,LiuLiuYou2009}]\label{lem:S-clique}
The induced subgraph $G[S]$ is a complete graph.
\end{lemma}

\begin{proof}
The assertion is clear for $|S|\geq 3$ by the monotonicity of $\lambda_1(G)$ with respect to edge 
addition. So in the following we assume $|S|=2$. 

Suppose for the sake of contradiction that $S=\{u,v\}$ and $u\nsim v$. Then $d(u)=d(v)=\Delta-1$ 
due to the maximality of $\lambda_1(G)$. Without loss of generality, we assume $x_u\geq x_v$. 
We first show that $N(u)=N(v)$ by contradiction. If $N(u)\cap N(v)\neq\emptyset$, we choose 
a vertex $w\in N(v)\setminus N(u)$. Let $G':=G+uw-vw$. Obviously, $G'$ is still a connected 
nonregular graph, and $\lambda_1(G')>\lambda_1(G)$ by \autoref{lem:rotation}, a contradiction.
If $N(u)\cap N(v)=\emptyset$, let $P$ be a shortest path from $v$ to $u$. Since
$d(v)=\Delta-1\geq 2$, there exists a vertex $w\in N(v)$ such that $w\notin V(P)$.
Let $G'':=G+uw-vw$. Then $G''$ is connected and $\lambda_1(G'')>\lambda_1(G)$, a contradiction.
It follows that $N(u)=N(v)$. 

Since $G$ is connected and $d(u)=d(v)=\Delta-1$, there exits two vertices $s,t\in N(u)=N(v)$ 
such that $st\notin E(G)$. In what follows we shall prove $x_s > x_u$. If $|N_T(s)\cap N(u)|=\Delta-2$, 
by eigenvalue equations we have
\[
(\lambda_1(G)+1) (x_s-x_u) = x_v,
\]
which yields $x_s > x_u$. If $|N_T(s)\cap N(u)| < \Delta-2$, there must be a vertex $w\in N_T(s)$
such that $wu\notin E(G)$. If $x_s\leq x_u$, then the graph $G+uw-ws$ has larger spectral radius 
than $G$, a contradiction. Therefore, we have $x_s > x_u$. Similarly, we have $x_t > x_v$. Finally, 
we deduce that $\lambda_1(\ls (G; u,t, v,s)) > \lambda_1(G)$ by \autoref{lem:local-switching}, 
a contradiction completing the proof.
\end{proof}

We immediately obtain an upper bound on $|S|$ as follows.

\begin{corollary}\label{coro:S-less-Delta-1}
$|S|\leq \Delta-1$.
\end{corollary}

\subsection{The components of Perron vector on $S$ are small} 

Roughly speaking, the next lemma states that for any $u$, $v\in S$, the neighbors of $u$ and $v$ in $T$ 
are nested, i.e., one of sets is contained in the other one.   

\begin{lemma}\label{lem:subset-neighbours}
Let $G\in\mathcal{G}(n,\Delta)$ and $u,v\in S$. Then $x_u\leq x_v$ if and only if $N_T(u)\subset N_T(v)$. 
\end{lemma}

\begin{proof}
We first show the necessity. Suppose contrary that there exists $w\in N_T(u)$ and $w\notin N_T(v)$. 
Let $G':=G-uw+vw$. Then $\lambda_1(G')>\lambda_1(G)$ by \autoref{lem:rotation}, a contradiction. 
Hence, we have $N_T(u)\subset N_T(v)$.

For the other direction, by eigenvalues equations for $u$, $v$ and \autoref{lem:S-clique}, we have
\[
\lambda_1(G) x_u = \sum_{w\in S\setminus\{u\}} x_w + \sum_{w\in N_T(u)} x_w,~~
\lambda_1(G) x_v = \sum_{w\in S\setminus\{v\}} x_w + \sum_{w\in N_T(v)} x_w.
\]
Therefore, we derive that
\[
(\lambda_1(G) + 1)(x_v-x_u) = \sum_{w\in N_T(v)\setminus N_T(u)} x_w \geq 0,
\]
which implies that $x_v\geq x_u$, as desired.
\end{proof}

A vertex $v\in V(G)$ is called {\em local minimum} if $x_w\geq x_v$ for any $w \in N(v)$.
Obviously, any vertex in $T$ is not local minimum, for otherwise, $G$ has largest eigenvalue 
at least $\Delta$, a contradiction. 

\begin{lemma}\label{lem:T-min-S-max}
Let $u\in S$ and $v\in T$ be two vertices such that $x_u=\max\{x_w: w\in S\}$ and 
$x_v=\min\{x_w: w\in T\}$. Then $x_u<x_v$.
\end{lemma}

\begin{proof}
Assume by contradiction that $x_u\geq x_v$. Since $v$ is not local minimum, we have 
$N_S(v)\neq\emptyset$. Using \autoref{lem:subset-neighbours}, we find that $uv\in E(G)$.
If there exists $w\in N_T(v)$ such that $uw\notin E(G)$, we let $G':=G+uw-vw$. Then 
$G'$ is connected and $\lambda_1(G') > \lambda_1(G)$, a contradiction. So we assume 
$N_T(v)\subset N_T(u)$ below. Now we consider the degree of $u$. In light of \autoref{lem:S-clique} 
we deduce that 
\[ 
d(u) = (|S|-1) + |N_T(u)| \geq (|S| + |N_T(v)|) - 1 \geq d(v) - 1 = \Delta - 1.
\] 
On the other hand, $d(u)\leq\Delta-1$ due to $u\in S$. Hence, all inequalities above 
must be equalities, we immediately obtain $N_T(v) = N_T(u)$. By the eigenvalue equations 
for $u$ and $v$, we have $(\lambda_1(G) +1) x_u = \lambda_1(G) x_v$, contrary to the 
assumption $x_u\geq x_v$. This completes the proof of this lemma.
\end{proof}

% \begin{lemma}
% Let $u\in S$ and $N_T(u)\neq\emptyset$. For any $v\notin N(u)$, we have $x_v>x_w$ for 
% each $w\in N(u)$.
% \end{lemma}

% \begin{proof}
% The assertion is clearly for $w\in N(u)\cap S$ by \autoref{lem:T-min-S-max}. Therefore 
% we assume $w\in N_T(u)$. We first show that there exists a vertex $w'\in N_T(v)$ such 
% that $ww'\notin E(G)$. If not, $w$ is adjacent to each vertex in $N(v)\setminus (S\cup\{w\})$. 
% By \autoref{lem:subset-neighbours}, $w$ is also adjacent to each vertex in $N(v)\cap S$. 
% Hence, we have $d(w)\geq \Delta+1$, a contradiction. 

% Let $G'=G+uv+ww'-uw-vw'$. If $x_v\leq x_w$, then 
% \[ 
% \lambda_1(G') - \lambda_1(G)\geq 2(x_v-x_w)(x_u-x_{w'})\geq 0,
% \] 
% a contradiction.
% \end{proof}

\subsection{Local adjacency relations of vertices}

Given an integer $k$, we denote by $\mathcal{M}_k$ the collection of the first $k$ vertices 
(with respect to the components of Perron vector $\bm{x}$) of $G$, i.e., 
$\mathcal{M}_k := \{v_1,v_2,\ldots,v_k\}$.

The following lemma will be used frequently in the subsequent proofs.

\begin{lemma}[Principle of Proximity]\label{lem:common-use}
Let $v\in\mathcal{M}_k$ and $d=d_{G[\mathcal{M}_k]} (v) < \Delta$. If $v$ is not adjacent 
to $v_{k+1},\ldots,v_{k+s}$ and $N(v_i)\cap\mathcal{M}_k = \emptyset$ for 
$i\in\{k+s+1,\ldots,k+s+\Delta-d\}\setminus\{j: v_j\in N(v)\}$, then we can connect 
$v$ to $v_{k+s+1},\ldots,v_{k+s+\Delta-d}$, not decreasing $\lambda_1(G)$.
\end{lemma}

\begin{proof}
Assume that $v$ is not adjacent to some $u\in\{v_{k+s+1},\ldots,v_{k+s+\Delta-d}\}$, and
$v_j\notin \mathcal{M}_k\cup\{v_{k+s+1},\ldots,v_{k+s+\Delta-d}\}$ is a neighbor of $v$ 
with $j > k+s+\Delta-d$. Since $N(u)\cap\mathcal{M}_k = \emptyset$, we deduce that there 
exists a neighbor $w$ of $u$ such that $w\nsim v_j$. By \autoref{lem:local-switching}, 
we may apply the operation $\ls (v,v_j, u,w)$ to $G$, not decreasing the largest eigenvalue 
and leaving $v$ adjacent to $u$, as desired. 
\end{proof}

If $s=0$, we immediately obtain the following corollary.

\begin{corollary}\label{coro:common-use-1}
Let $v\in\mathcal{M}_k$ and $d=d_{G[\mathcal{M}_k]} (v) < \Delta$. 
If $N(v_i)\cap\mathcal{M}_k = \emptyset$ for 
$i\in\{k+1,\ldots,k+\Delta-d\}\setminus\{j: v_j\in N(v)\}$,
then we can connect $v$ to $v_{k+1},\ldots,v_{k+\Delta-d}$, 
not decreasing $\lambda_1(G)$.
\end{corollary}

The next conclusion then follows immediately by setting $k=1$ in \autoref{coro:common-use-1}.

\begin{corollary}\label{coro:common-use-2}
Let $G\in\mathcal{G}(n,\Delta)$. Then we can connect $v_1$ to both $v_2$, $v_3,\ldots,v_{\Delta+1}$,
not decreasing the largest eigenvalue of $G$.
\end{corollary}

We end this section with two auxiliary results.

\begin{proposition}\label{prop:preserving-connected}
For any local switching to $G$ that preserving non-decreasing of $\lambda_1(G)$, the 
resulting graph is still a connected graph.
\end{proposition}

\begin{lemma}\label{lem:G-delete-vertices-connected}
For any positive integer $k$, we have $G\setminus\mathcal{M}_k$ is connected.
\end{lemma}

\begin{proof}
We prove this lemma by contradiction. Recall that $S$ induced a clique in $G$. Hence, 
this assertion is clear for $k\geq n - |S|$ by \autoref{lem:T-min-S-max}. In what 
follows, we assume $k < n - |S|$. Thus, $S$ must be contained in a unique connected 
components of $G\setminus\mathcal{M}_k$. Let $H$ be another connected components of 
$G\setminus\mathcal{M}_k$ not containing $S$, and let $v$ be a vertex in $H$ such that 
$x_v = \min\{x_u: u\in V(H_2)\}$. Then $v\in T$ is a local minimum vertex in $G$, a 
contradiction completing the proof of this lemma.
\end{proof}

\section{Structure of extremal graphs in $\mathcal{G}(n,3)$}
\label{sec:structure-Delta=3}

The aim of this section is to determine the structure of graphs attaining the maximum
spectral radius among all connected nonregular graphs with $n$ vertices and maximum 
degree $3$, as well as give a proof of \autoref{conj:degree-sequence} for $\Delta = 3$.

\subsection{Transferring the extremal graphs into the path-like structure}

Our first goal is to prove that we can reconnect the first few vertices of $G$ to obtain 
one of the two graphs $G_1$ and $G_2$ as illustrated in Fig.\,\ref{fig:G1-G2-G3}\,\subref{subfig:G1} 
and \subref{subfig:G2}. A local switching is said to be {\em proper} if it satisfies 
the condition of \autoref{lem:local-switching}.

\begin{lemma}\label{lem:transfer-first-few-vertices-Delta-3}
Let $n\geq 8$. The induced subgraph on the first few vertices in $G$ can be transferred 
into $G_1$ or $G_2$.
\end{lemma}

\begin{proof}
We prove this lemma by the following four steps.

{\bfseries A.} Connecting $v_1$ to $v_i$, $i=2,3,4$. This can be done by \autoref{coro:common-use-2}.

{\bfseries B.} Connecting $v_2$ to $v_3$. By \autoref{lem:G-delete-vertices-connected}, 
$G\setminus\{v_1\}$ is connected. If $v_2\nsim v_3$, then either $v_2$, $v_3$ share two 
neighbors which are adjacent in $G\setminus\{v_1\}$, or $v_2$, $v_3$ have disjoint 
neighbors and each neighbor of $v_2$ is adjacent to each neighbor of $v_3$ in $G\setminus\{v_1\}$.
This contradicts the fact that $G\setminus\{v_1\}$ is connected.    

{\bfseries C.} Connecting $v_2$ to $v_4$. Let $u$ be the third neighbor of $v_2$ other 
than $v_1$, $v_3$; and $v$, $w$ be the two neighbors of $v_4$ other than $v_1$. We may 
assume that $v_3$ is not adjacent to $v_4$, for otherwise, the proper local switching 
$\ls(v_2,u,v_4,v_3)$ connects $v_2$ to $v_4$. We can also assume $v_3\nsim u$. Otherwise, 
at least one of $v$ and $w$, say $v$, is not adjacent to $u$. Then $\ls (v_2,u,v_4,v)$ 
connects $v_2$ to $v_4$.

Let $\ell$ be the third neighbor of $v_3$ other than $v_1$ and $v_2$. We consider the
following two cases.

\noindent {\bfseries Case 1.} $u\nsim v_4$. We may assume $u\sim v$ and $u\sim w$.
Otherwise, without loss of generality we assume $u\nsim w$, then $\ls(v_2,u, v_4,w)$ 
makes $v_2$ adjacent to $v_4$. Since $G\setminus\{v_1,v_2,v_3\}$ is connected, we see 
$v\nsim w$. We further assume $v_3\nsim v$ and $v_3\nsim w$. Otherwise, if $v_3\sim w$, 
then $\ls(v_3,w,v_4,v)$ makes $v_3\sim v_4$; if $v_3\sim v$, then $\ls (v_3,v,v_4,w)$ 
makes $v_3\sim v_4$, contrary to our assumption. Finally, we can transfer $G$ to the 
one such that $\ell\sim v$ and $\ell\sim w$, contradicting to \autoref{prop:preserving-connected}.

\noindent {\bfseries Case 2.} $u\sim v_4$. Since we can assume $u$ is adjacent to the 
third neighbor $v$ of $v_4$ other than $v_1$, $u$, we have $v_3\nsim v$. Finally, 
$\ls (v_3,\ell, v_4,u)$ makes $v_3\sim v_4$, a contradiction to our assumption.

Based on previous construction, we shall reconnect the subsequent few vertices.

{\bfseries D.} By \autoref{coro:common-use-1}, we may assume $v_5\sim v_3$. If $v_5\sim v_4$, 
we get the graph $G_1$. If $v_5\nsim v_4$, we may assume $v_4\sim v_6$ by \autoref{lem:common-use}. 
Furthermore, we can connect $v_5$ to $v_6$. Indeed, if $v_5$ is not adjacent to $v_6$, there 
exist $u\in N(v_5)\setminus\{v_3\}$ and $v\in N(v_6)\setminus\{v_4\}$ such that $u\nsim v$. 
Then $\ls (v_5,u,v_6,v)$ connects $v_5$ to $v_6$.

Using \autoref{coro:common-use-1} again, we assume $v_5\sim v_7$. If $v_6\sim v_7$, we 
obtain $G_2$. If $v_6\nsim v_7$, then $\ls (v_4,v_6,v_5,v_7)$ transfers the first fewer
vertices of $G$ into $G_1$. This completes the proof of the lemma.
\end{proof}

\begin{figure}[htbp]
\centering
\subcaptionbox{$G_1$ \label{subfig:G1}}[.3\textwidth]{%
\begin{tikzpicture}[scale=0.9]
\coordinate (u1) at (1.7,0);
\coordinate (v1) at (0,0.5);
\coordinate (w1) at (0,-0.5);
\coordinate (v2) at (1,0.5);
\coordinate (w2) at (1,-0.5);

\draw (v1) -- (v2) -- (u1) -- (w2) -- (w1) -- (v2);
\draw (w2) -- (v1) -- (w1);
\draw (u1) -- (1.9,0);

\node[above=2pt, scale=0.9] at (u1) {$v_5$};
\node[above=2pt, scale=0.9] at (v1) {$v_1$};
\node[above=2pt, scale=0.9] at (v2) {$v_3$};
\node[below=2pt, scale=0.9] at (w1) {$v_2$};
\node[below=2pt, scale=0.9] at (w2) {$v_4$};

\foreach \i in {u1,v1,v2,w1,w2}
\filldraw (\i) circle (0.07);
\end{tikzpicture}
}
\subcaptionbox{$G_2$ \label{subfig:G2}}[.3\textwidth]{%
\begin{tikzpicture}[scale=0.9]
\coordinate (u1) at (2.7,0);
\coordinate (v1) at (0,0.5);
\coordinate (w1) at (0,-0.5);
\coordinate (v2) at (1,0.5);
\coordinate (w2) at (1,-0.5);
\coordinate (v3) at (2,0.5);
\coordinate (w3) at (2,-0.5);

\draw (v1) -- (v2) -- (v3) -- (u1) -- (w3) -- (w2) -- (w1) -- (v1) -- (w2);
\draw (v2) -- (w1);
\draw (v3) -- (w3);
\draw (u1) -- (2.9,0);

\node[above=2pt, scale=0.9] at (u1) {$v_7$};
\node[above=2pt, scale=0.9] at (v1) {$v_1$};
\node[above=2pt, scale=0.9] at (v2) {$v_3$};
\node[above=2pt, scale=0.9] at (v3) {$v_5$};
\node[below=2pt, scale=0.9] at (w1) {$v_2$};
\node[below=2pt, scale=0.9] at (w2) {$v_4$};
\node[below=2pt, scale=0.9] at (w3) {$v_6$};

\foreach \i in {u1,v1,v2,v3,w1,w2,w3}
\filldraw (\i) circle (0.07);
\end{tikzpicture}
}
\subcaptionbox{$G_3$ \label{subfig:G3}}[.3\textwidth]{%
\begin{tikzpicture}[scale=0.9]
\coordinate (o) at (0,0);
\coordinate (u1) at (1,0);
\coordinate (u2) at (2.4,0);

\coordinate (v1) at (1.7,0.5);
\coordinate (w1) at (1.7,-0.5);

\draw (o) -- (u1) -- (v1) -- (u2) -- (w1) -- (v1);
\draw (u1) -- (w1);
\draw (-0.15,0.15) -- (o) -- (-0.15,-0.15);
\draw (u2) -- (2.6,0);

\node[above=2pt, scale=0.9] at (o) {$v_{k+1}$};
\node[above=2pt, scale=0.9] at (0.9,0) {$v_{k+2}$};
\node[above=2pt, scale=0.9] at (v1) {$v_{k+3}$};
\node[below=2pt, scale=0.9] at (w1) {$v_{k+4}$};
\node[above=2pt, scale=0.9] at (2.6,0) {$v_{k+5}$};
\foreach \i in {o,u1,u2,v1,w1}
\filldraw (\i) circle (0.07);
\end{tikzpicture}
}
\caption{Graphs $G_1$, $G_2$ and $G_3$}
\label{fig:G1-G2-G3}
\end{figure}
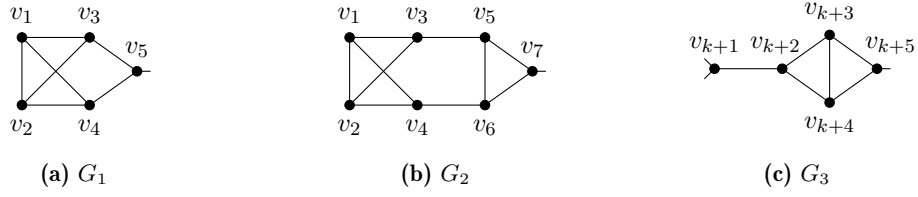

In what follows, we shall continue to reconnect the remaining vertices. In the process 
of reconnecting, assume we have already built some specific subgraph on $\mathcal{M}_k$, 
then we continue to build a desired subgraph on $V(G)\setminus\mathcal{M}_k$ in a way 
not to alter the subgraph already constructed on $\mathcal{M}_k$. 

Before continuing, we need to rule out a special structure.

\begin{lemma}\label{lem:forbidden-D-1}
The graph $G$ does not contain the graph $D_1$ shown in Fig.\,\ref{fig:D1-D2} as an induced subgraph. 
\end{lemma}

\begin{proof}
Assume to the contrary that $G$ contains such a structure. We first show that
$x_{k+1} > x_{k+4}$. If not, we have $x_{k+1}=\cdots=x_{k+4}$. By eigenvalue equation 
for $v_{k+2}$, we see 
\[
x_k = (\lambda_1(G) - 2) x_{k+2} < x_{k+2},
\]
a contradiction yielding $x_{k+1} > x_{k+4}$. By \autoref{lem:local-switching}, 
$\lambda_1(\ls (G; v_{k+1},v_{k+3}, v_{k+2},v_{k+4})) > \lambda_1(G)$, a contradiction
finishing the proof of the lemma.
\end{proof}

\begin{figure}[htbp]
\centering
\subcaptionbox{$D_1$ \label{subfig:D1}}[.25\textwidth]{%
\begin{tikzpicture}[scale=0.9]
\coordinate (v1) at (0,0.5);
\coordinate (w1) at (0,-0.5);
\coordinate (v2) at (1,0.5);
\coordinate (w2) at (1,-0.5);
\coordinate (v3) at (2,0.5);
\coordinate (w3) at (2,-0.5);

\draw (v1) -- (v2) -- (v3);
\draw (w1) -- (w2) -- (w3);
\draw (v2) -- (w2);
\draw (-0.2,0.5) -- (v1) -- (0,0.3); 
\draw (-0.2,-0.5) -- (w1) -- (0,-0.3);
\draw (2.2,0.5) -- (v3) -- (2,0.3);
\draw (2.2,-0.5) -- (w3) -- (2,-0.3);

\node[above=2pt, scale=0.9] at (v1) {$v_{k}$};
\node[above=2pt, scale=0.9] at (v2) {$v_{k+2}$};
\node[above=2pt, scale=0.9] at (v3) {$v_{k+4}$};
\node[below=2pt, scale=0.9] at (w1) {$v_{k+1}$};
\node[below=2pt, scale=0.9] at (w2) {$v_{k+3}$};
\node[below=2pt, scale=0.9] at (w3) {$v_{k+5}$};

\foreach \i in {v1,v2,v3,w1,w2,w3}
\filldraw (\i) circle (0.07);
\end{tikzpicture}
}
\subcaptionbox{$D_2$ \label{subfig:D2}}[.32\textwidth]{%
\begin{tikzpicture}[scale=0.9]
\coordinate (o) at (0,0);
\coordinate (u1) at (1,0);
\coordinate (u2) at (3.5,0);

\coordinate (v0) at (-0.8,0.5);
\coordinate (w0) at (-0.8,-0.5);
\coordinate (v1) at (1.8,0.5);
\coordinate (w1) at (1.8,-0.5);
\coordinate (v2) at (2.8,0.5);
\coordinate (w2) at (2.8,-0.5);

\draw (o) -- (u1) -- (v1) -- (v2);
\draw (u1) -- (w1) -- (w2);
\draw (v1) -- (w1);
\draw (v0) -- (o) -- (w0);
\draw (v2) -- (w2) -- (u2) -- (v2);
\draw (u2) -- (3.7,0);
\draw (-1,0.5) -- (v0) -- (-0.8,0.3);
\draw (-1,-0.5) -- (w0) -- (-0.8,-0.3);

\node[above=2pt, scale=0.9] at (o) {$v_{k+1}$};
\node[above=2pt, scale=0.9] at (u1) {$v_{k+2}$};
\node[above=2pt, scale=0.9] at (v0) {$u$};
\node[below=2pt, scale=0.9] at (w0) {$v$};
\node[above=2pt, scale=0.9] at (v1) {$v_{k+3}$};
\node[below=2pt, scale=0.9] at (w1) {$v_{k+4}$};
\node[above=2pt, scale=0.9] at (v2) {$v_{k+5}$};
\node[below=2pt, scale=0.9] at (w2) {$v_{k+6}$};

\foreach \i in {o,u1,u2,v0,v1,v2,w0,w1,w2}
\filldraw (\i) circle (0.07);
\end{tikzpicture}
}
\subcaptionbox{$\widetilde{D}_2$ \label{subfig:tilde-D2}}[.32\textwidth]{%
\begin{tikzpicture}[scale=0.9]
\coordinate (o) at (0,0);
\coordinate (u1) at (1,0);
\coordinate (u2) at (2.6,0);

\coordinate (v-1) at (-1.8,0.5);
\coordinate (w-1) at (-1.8,-0.5);
\coordinate (v0) at (-0.8,0.5);
\coordinate (w0) at (-0.8,-0.5);
\coordinate (v1) at (1.8,0.5);
\coordinate (w1) at (1.8,-0.5);

\draw (o) -- (u1) -- (v1);
\draw (u1) -- (w1);
\draw (o) -- (v0) -- (v-1);
\draw (o) -- (w0) -- (w-1);
\draw (v0) -- (w0);
\draw (-2,0.5) -- (v-1) -- (-1.8,0.3);
\draw (-2,-0.5) -- (w-1) -- (-1.8,-0.3);
\draw (u2) -- (v1) -- (w1) -- (u2) -- (2.8,0);

\node[above=2pt, scale=0.9] at (o) {$v_{k+3}$};
\node[above=2pt, scale=0.9] at (u1) {$v_{k+4}$};
\node[above=2pt, scale=0.9] at (v-1) {$u$};
\node[below=2pt, scale=0.9] at (w-1) {$v$};
\node[above=2pt, scale=0.9] at (v0) {$v_{k+1}$};
\node[below=2pt, scale=0.9] at (w0) {$v_{k+2}$};
\node[above=2pt, scale=0.9] at (v1) {$v_{k+5}$};
\node[below=2pt, scale=0.9] at (w1) {$v_{k+6}$};

\foreach \i in {o,u1,u2,v-1,v0,v1,w-1,w0,w1}
\filldraw (\i) circle (0.07);
\end{tikzpicture}
}
\caption{Graphs $D_1$, $D_2$ and $\widetilde{D}_2$}
\label{fig:D1-D2}
\end{figure}
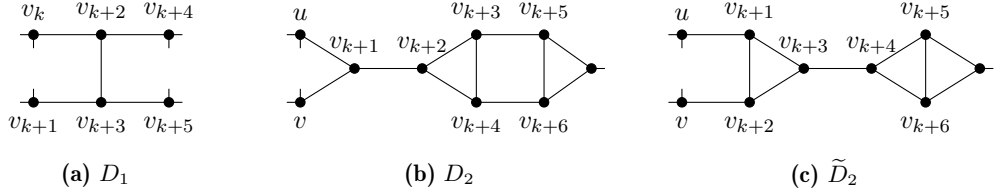

\begin{lemma}\label{lem:induced-middle-vertices-Delta=3}
Let $H$ be an induced subgraph of $G$ on vertices $v_{k+1}$,$\ldots$,$v_{k+5}$, where 
the first vertex $v_{k+1}$ has degree one in $H$. Let $N_G(v_i)\cap\mathcal{M}_k=\emptyset$, 
$i\geq k+2$. Then $H$ can be transferred into $G_3$ as an induced subgraph of $G$, not 
decreasing the largest eigenvalue of $G$.
\end{lemma}

\begin{proof}
By repeatedly using \autoref{coro:common-use-1}, we have $v_{k+1}\sim v_{k+2}$, 
$v_{k+2}\sim v_{k+3}$ and $v_{k+2}\sim v_{k+4}$. By considering the neighbors of 
$v_{k+3}$ and $v_{k+4}$, we may assume $v_{k+3}\sim v_{k+4}$. Furthermore, we 
can connect $v_{k+3}$ to $v_{k+5}$ by \autoref{coro:common-use-1}. 

Next, we shall connect $v_{k+4}$ to $v_{k+5}$, and therefore finishing the proof.
If $v_{k+4}\nsim v_{k+5}$, we obtain the graph $D_2$ in Fig.\,\ref{fig:D1-D2}\,\subref{subfig:D2} 
by \autoref{lem:common-use} and \autoref{lem:forbidden-D-1}. By eigenvalue 
equations we deduce that $x_{k+3}=x_{k+4}$ and $x_{k+5}=x_{k+6}$.

To finish the proof we will construct a graph $\widetilde{G}$ by replacing $D_2$ 
with $\widetilde{D}_2$ in $G$ such that $\lambda_1(\widetilde{G}) > \lambda_1(G)$, 
which yields a contradiction. To compare $\lambda_1(\widetilde{G})$ with $\lambda_1(G)$, 
we define a vector $\bm{y}$ for $\widetilde{G}$ as follows:
\[
y_u =
\begin{cases}
    x_{k+1}, & u = v_{k+2}, \\
    x_{k+1} + x_{k+3} - x_{k+2}, & u = v_{k+3}, \\
    x_u, & \text{otherwise}.
\end{cases}
\]
Recall that $x_{k+3}=x_{k+4}$ and $x_{k+5}=x_{k+6}$. A straightforward calculation shows that 
\begin{equation}\label{eq:equal-sum-1}
\sum_{uv\in E(G)} (x_u-x_v)^2 = \sum_{uv\in E(\widetilde{G})} (y_u-y_v)^2.
\end{equation}
Therefore, in light of \eqref{eq:Delta-rho} and \eqref{eq:equal-sum-1}, we have
\begin{align*}
(3-\lambda_1(\widetilde{G}))\cdot\|\bm{y}\|_2^2 
& \leq \sum_{v\in V(\widetilde{G})} (3-d_{\widetilde{G}}(v)) y_v^2 + \sum_{uv\in E(\widetilde{G})} (y_u-y_v)^2 \\
& = 3 - \lambda_1(G).
\end{align*}
In the following we shall prove that $\|\bm{y}\|_2>1$, and therefore $3-\lambda_1(\widetilde{G}) < 3-\lambda_1(G)$,
which results in a contradiction. Indeed, 
\begin{align*}
\|\bm{y}\|_2^2 
& = \|\bm{x}\|_2^2 + (y_{k+2}^2+y_{k+3}^2) - (x_{k+2}^2 + x_{k+3}^2) \\
& = 1+ 2(x_{k+1}-x_{k+2}) (x_{k+1}+x_{k+3}) \\
& > 1,
\end{align*}
a contradiction completing the proof of \autoref{lem:induced-middle-vertices-Delta=3}.
\end{proof}

By \autoref{lem:transfer-first-few-vertices-Delta-3} we can transfer the first few 
vertices in $G\in\mathcal{G}(n,3)$ into one of the graphs as shown in 
Fig.\,\ref{fig:G1-G2-G3}\,\subref{subfig:G1}\,--\,\subref{subfig:G2}. Moreover, 
whatever we obtained, we end up with a cut vertex of degree two. Next, we 
can employ \autoref{lem:induced-middle-vertices-Delta=3} to reconnect the 
remaining vertices. Doing so, we again obtain a cut vertices with degree two.
By repeatedly using \autoref{lem:induced-middle-vertices-Delta=3} we can 
reconnect almost all vertices in $G$ except the last few vertices.

\subsection{Proof of \autoref{conj:degree-sequence} for $\Delta=3$}

Now we are ready to confirm the \autoref{conj:degree-sequence} for $\Delta=3$, 
which is stated as follows.

\begin{theorem}\label{thm:proof-conj-degree-sequ-for-degree3}
Let $G\in\mathcal{G}(n,3)$. Then $G$ has degree sequence $(3,\ldots,3,\delta)$, 
where
\[
\delta =
\begin{cases}
2, & n\ \text{is odd}, \\
1, & n\ \text{is even}.
\end{cases}
\]
\end{theorem}

\begin{figure}[htbp]
\centering
\subcaptionbox{$D_3$ \label{subfig:D3}}[.3\textwidth]{%
\begin{tikzpicture}[scale=0.9]
\coordinate (v1) at (0,0);
\coordinate (v2) at (1,0);
\coordinate (v3) at (2.4,0);

\coordinate (u1) at (-0.7, 0.5);
\coordinate (w1) at (-0.7, -0.5);
\coordinate (u2) at (1.7, 0.5);
\coordinate (w2) at (1.7, -0.5);

\draw (v1) -- (u1) -- (w1) -- (v1) -- (v2) -- (u2) -- (w2) -- (v3) -- (2.6,0);
\draw (u2) -- (v3);
\draw (v2) -- (w2);

\node[above=2pt, scale=0.9] at (u1) {$v_n$};
\node[below=2pt, scale=0.9] at (w1) {$v_{n-1}$};
\node[above=2pt, scale=0.9] at (v1) {$v_{n-2}$};
\node[below=2pt, scale=0.9] at (0.9,0) {$v_{n-3}$};
\node[above=2pt, scale=0.9] at (u2) {$v_{n-4}$};
\node[below=2pt, scale=0.9] at (w2) {$v_{n-5}$};

\foreach \i in {u1,u2,v1,v2,v3,w1,w2}
\filldraw (\i) circle (0.07);
\end{tikzpicture} 
}
\subcaptionbox{$D_4$ \label{subfig:D4}}[.3\textwidth]{%
\begin{tikzpicture}[scale=0.9]
\coordinate (u1) at (0.7,0);
\coordinate (u2) at (2.1,0);
\coordinate (v1) at (0,0.5);
\coordinate (v2) at (1.4,0.5);
\coordinate (w1) at (0,-0.5);
\coordinate (w2) at (1.4,-0.5);

\draw (w1) -- (v1) -- (u1) -- (v2) -- (u2) -- (w2) -- (v2);
\draw (u1) -- (w2);
\draw (u2) -- (2.3,0);

\node[below=2pt, scale=0.9] at (w1) {$v_n$};
\node[above=2pt, scale=0.9] at (v1) {$v_{n-1}$};
\node[below=2pt, scale=0.9] at (0.6,0) {$v_{n-2}$};
\node[above=2pt, scale=0.9] at (v2) {$v_{n-3}$};
\node[below=2pt, scale=0.9] at (w2) {$v_{n-4}$};
\node[below=2pt, scale=0.9] at (2.3,0) {$v_{n-5}$};

\foreach \i in {u1,u2,v1,v2,w1,w2}
\filldraw (\i) circle (0.07);
\end{tikzpicture} 
}
\subcaptionbox{$\widetilde{D}_3$ \label{subfig:tilde-D3}}[.3\textwidth]{%
\begin{tikzpicture}[scale=0.9]
\coordinate (v1) at (0,0);
\coordinate (v2) at (1,0);
\coordinate (v3) at (3.6,0);

\coordinate (u1) at (1.8, 0.5);
\coordinate (w1) at (1.8, -0.5);
\coordinate (u2) at (2.8, 0.5);
\coordinate (w2) at (2.8, -0.5);

\draw (v1) -- (v2) -- (u1) -- (u2) -- (v3) -- (3.8,0);
\draw (v2) -- (w1) -- (w2) -- (v3);
\draw (u1) -- (w1);
\draw (u2) -- (w2);

\node[above=2pt, scale=0.9] at (u1) {$v_{n-2}$};
\node[below=2pt, scale=0.9] at (w1) {$v_{n-3}$};
\node[below=2pt, scale=0.9] at (v1) {$v_n$};
\node[below=2pt, scale=0.9] at (v2) {$v_{n-1}$};
\node[above=2pt, scale=0.9] at (u2) {$v_{n-4}$};
\node[below=2pt, scale=0.9] at (w2) {$v_{n-5}$};

\foreach \i in {u1,u2,v1,v2,v3,w1,w2}
\filldraw (\i) circle (0.07);
\end{tikzpicture}  
}
\caption{Graphs $D_3$, $D_4$ and $\widetilde{D}_3$ in \autoref{thm:proof-conj-degree-sequ-for-degree3}}
\label{fig:D3-D4-tilde-D3}
\end{figure}

\begin{proof}
Let $S$ be the set of vertices of $G$ with degree less than $3$. It suffices to show that 
$S=1$. We assume $|S|=2$ by contradiction in light of \autoref{coro:S-less-Delta-1}.

Since the local switching preserves degree sequence, it is enough to consider the graph 
$G$ obtained by \autoref{lem:transfer-first-few-vertices-Delta-3} and 
\autoref{lem:induced-middle-vertices-Delta=3}. So the induced subgraph on the last few 
vertices (with respect to the components of Perron vector $\bm{x}$) in $G$ is either 
$D_3$ or $D_4$ as shown in Fig.\,\ref{fig:D3-D4-tilde-D3}\,\subref{subfig:D3}\,--\,\subref{subfig:D4} 
depending on the parity of $n$. The remaining proof is split into the following two cases.
\vspace{1.5mm}

\noindent {\bfseries Case 1.} $n$ is even. The induced subgraph of the last few vertices 
must be $D_3$. Now, we construct a graph $\widetilde{G}$ from $G$ by replacing $D_3$ with 
$\widetilde{D}_3$, and define a vector $\bm{y}$ for $\widetilde{G}$ as follows:
\[
y_u =
\begin{cases}
    x_n + x_{n-3} - x_{n-2}, & u = v_{n-1}, \\
    x_n + x_{n-4} - x_{n-2}, & u \in \{v_{n-2}, v_{n-3}\}, \\
    x_u, & \text{otherwise}.
\end{cases}
\]
Since $x_n = x_{n-1}$ and $x_{n-4} = x_{n-5}$, we have 
$\bm{x}^{\mathrm{T}} L(G) \bm{x} = \bm{y}^{\mathrm{T}} L(\widetilde{G}) \bm{y}$.
Combining with \eqref{eq:Delta-rho}, we derive that
\[
3 - \lambda_1(\widetilde{G}) \leq \frac{3 - \lambda_1(G)}{\|\bm{y}\|_2^2}.
\]
Below we shall show that $\|\bm{y}\|_2 > 1$. Set for short, $a:= x_n=x_{n-1}$, 
$b:= x_{n-2}$, $c:= x_{n-3}$, $d:= x_{n-4}=x_{n-5}$, and $\lambda:=\lambda_1(G)$. 
By eigenvalue equations, we have
\[
(\lambda-1) a = b,~~ 
\lambda b = 2a+c,~~
\lambda c = b+2d,
\]
which is equivalent to
\begin{equation}\label{eq:equation-a-b-c}
b = (\lambda-1) a,~~
c = (\lambda^2 - \lambda -2) a,~~
d = \frac{\lambda^3-\lambda^2-3\lambda+1}{2} a.
\end{equation}
Observe that the square of the norm $\|\bm{y}\|$ is equal to
\begin{align*}
\|\bm{y}\|_2^2 
& = \|\bm{x}\|_2^2 + (y_n^2+y_{n-1}^2+2y_{n-2}^2) - (2x_n^2+x_{n-2}^2+x_{n-3}^2) \\
& = 1 + (a+c-b)^2+2(a+d-b)^2 - (a^2+b^2+c^2).
\end{align*}
Substituting \eqref{eq:equation-a-b-c} into the above equation, we deduce that
\[
\|\bm{y}\|_2^2 -1 = \frac{a^2}{2} [(\lambda^3 + 3\lambda^2 - \lambda - 8) (\lambda^2 - 3\lambda + 1) (\lambda - 2)-3] > 0,
\]
the last inequality is due to the fact $\lambda > \lambda_1(G_1) > 2.8$.
\vspace{1.5mm}

\noindent {\bfseries Case 2.} $n$ is odd. In this case, we first assume $n=4k+1$. One can deduce that $G$ is the 
graph as shown in Fig.\,\ref{fig:figure-in-thm3-1} (above). To finish the proof we 
shall construct a graph $\widetilde{G}$ shown in Fig.\,\ref{fig:figure-in-thm3-1} 
(below) such that $\lambda_1(\widetilde{G}) > \lambda_1(G)$, which yields the desired 
contradiction.

In order to compare $\lambda_1(\widetilde{G})$ with $\lambda_1(G)$ we define a vector 
$\bm{y}$ for $\widetilde{G}$ in the following way: 
\begin{equation}\label{eq:assignment-y}
\begin{split}
y_{3i} & = x_{3i}, ~~ i\in [k-1], \\
y_{3i-1} & = \frac{x_{3i-1} + x_{3i}}{\lambda-1}, ~~ i\in [k-1], \\
y_{3i-2} & = x_{3i-1}, ~~ i\in [k], \\
y_{3k-1} & = 2x_{3k-1}-x_{3k-2}, \\
y_{3k} & = x_{3k}+x_{3k-1}-x_{3k-2}.
\end{split}
\end{equation}
We will utilize \eqref{eq:Delta-rho} to get a contradiction. To this end, we need 
to estimate $\|\bm{y}\|_2$ and $\bm{y}^{\mathrm{T}} L(\widetilde{G}) \bm{y}$, respectively. 

Firstly, we give an estimation on the norm of $\bm{y}$. Observe that
\[
\|\bm{y}\|_2^2 = \sum_{i=1}^k y_{3i}^2 + 2\sum_{i=1}^k y_{3i-1}^2 + \sum_{i=1}^k y_{3i-2}^2 + y_{3k}^2,
\]
and the norm of $\bm{x}$ is equal to one, i.e.,
\[
1 = \sum_{i=1}^k x_{3i}^2 + \sum_{i=1}^k x_{3i-1}^2 + 2\sum_{i=1}^k x_{3i-2}^2 
+ x_{3k-1}^2+x_{3k}^2-x_1^2.
\]
Hence, the difference between $\|\bm{y}\|_2^2$ and $1$ is equal to
\begin{equation}\label{eq:norm-difference-y-and-x}
\|\bm{y}\|_2^2 - 1 = 2\sum_{i=1}^k (y_{3i-1}^2 - x_{3i-2}^2) + 2y_{3k}^2 - 2x_{3k}^2 - x_{3k-1}^2 + x_1^2.
\end{equation}
Using eigenvalue equations and \eqref{eq:assignment-y}, we deduce that
\begin{equation}\label{eq:relation-y-and-x}
y_{3k-1} = \frac{-\lambda^2+3\lambda+2}{2} x_{3k},~~
y_{3k} = \frac{-\lambda^2+2\lambda+5}{2} x_{3k}.
\end{equation}
Also we have
\begin{equation}\label{eq:x3k}
x_{3k-1} = \frac{\lambda - 1}{2} x_{3k},~~
x_{3k-2} = \frac{\lambda^2 - \lambda - 4}{2} x_{3k}.
\end{equation}
Substituting \eqref{eq:relation-y-and-x} and \eqref{eq:x3k} into \eqref{eq:norm-difference-y-and-x},
and noting that $x_{3i-2} = (x_{3i-3} + x_{3i-1})/(\lambda-1)$, we conclude that
\begin{align*}
\|\bm{y}\|_2^2 - 1 = & ~\frac{2}{(\lambda-1)^2} \sum_{i=2}^{k-1} (x_{3i} - x_{3i-3}) (x_{3i-3} + 2x_{3i-1} + x_{3i}) \\ 
& ~ + \frac{2\lambda^4 - 16\lambda^3 + 11\lambda^2 + 50\lambda + 17}{4} x_{3k}^2  
+ \frac{(2\lambda^2 - 1)(\lambda^2 + 2\lambda - 1)}{(\lambda - 1)^2} x_1^2.
\end{align*}
To proceed further, we observe that
\[
x_{3i-1} = \frac{2x_{3i-3} + (\lambda-1) x_{3i}}{(\lambda-2)(\lambda+1)},~~
i=2,3,\ldots,k-1,
\]
from which we acquire that
\[
x_{3i-3}+2x_{3i-1}+x_{3i} > \frac{\lambda^2+\lambda-4}{(\lambda-2)(\lambda+1)} (x_{3i-3}+x_{3i}).
\]
As a consequence,
\begin{align*}
\|\bm{y}\|_2^2-1 > & ~ \frac{2(\lambda^2+\lambda-4)}{(\lambda-1)^2(\lambda-2)(\lambda+1)} (x_{3k-3}^2-x_3^2) \\
& ~ + \frac{2\lambda^4-16\lambda^3+11\lambda^2+50\lambda+17}{4} x_{3k}^2 
+ \frac{(2\lambda^2 - 1)(\lambda^2 + 2\lambda - 1)}{(\lambda - 1)^2} x_1^2.
\end{align*}
Again, by eigenvalue equations we find $2x_{3k-3} = (\lambda-1)(\lambda^2-\lambda-5) x_{3k}$ and 
$x_3 = (\lambda^2 - 1)x_1$. Substituting them into the above inequality gives
\begin{align*}
\|\bm{y}\|_2^2 - 1 
& > (3 {-} \lambda) \bigg(-\frac{4\lambda^5 {-} 8\lambda^4 {-} 31\lambda^3 {-} 4\lambda^2 {+} 75\lambda {+} 78}{4(\lambda+1)(\lambda-2)} x_{3k}^2
+ \frac{\lambda^3 - 4\lambda^2 - 5\lambda + 6}{(3 - \lambda)(\lambda - 1)^2(\lambda - 2)} x_1^2\bigg) \\
& =: (3 - \lambda) f(\lambda).
\end{align*}

Secondly, we shall give an upper bound on $\bm{y}^{\mathrm{T}} L(\widetilde{G}) \bm{y}$ in 
terms of $3-\lambda$. To this end, set for short
\begin{align*}
P_i:= & ~ 2 (y_{3i-2}-y_{3i-1})^2 + 2 (y_{3i-1} - y_{3i})^2 + (y_{3i}-y_{3i+1})^2, \\
Q_i:= & ~ (x_{3i-1}-x_{3i})^2 + 2 (x_{3i} -x_{3i+1})^2 + 2 (x_{3i+1}-x_{3i+2})^2.
\end{align*}
Calculating the quadratic form $\bm{y}^{\mathrm{T}} L(\widetilde{G}) \bm{y}$ for the vector 
$\bm{y}$, we get
\[
\bm{y}^{\mathrm{T}} L(\widetilde{G}) \bm{y} 
= \sum_{i=1}^{k-1} P_i + 2 (y_{3k-2} - y_{3k-1})^2 + 4(y_{3k-1} - y_{3k})^2.
\]
Substituting \eqref{eq:assignment-y} and \eqref{eq:x3k} into the above equation gives rise to 

\begin{equation}\label{eq:yLy}
\bm{y}^{\mathrm{T}} L(\widetilde{G}) \bm{y} = \sum_{i=1}^{k-1} P_i + 
\frac{(\lambda - 3)^2 (\lambda + 1)^2}{2} x_{3k}^2 + 4(x_{3k-1} - x_{3k})^2.
\end{equation}
On the other hand, using \eqref{eq:Delta-rho-sum-form} and the fact $x_2=\lambda x_1$ we see
\begin{equation}\label{eq:3-lambda-Qi}
\begin{split}
3 - \lambda 
& = \sum_{i=1}^{k-1} Q_i + 4(x_{3k-1}-x_{3k})^2 + 2x_1^2+x_2^2+(x_1-x_2)^2 \\
& = \sum_{i=1}^{k-1} Q_i + 4(x_{3k-1} - x_{3k})^2 + (2\lambda^2 - 2\lambda +3) x_1^2.
\end{split}
\end{equation}
Now, let us estimate the difference between $P_i$ and $Q_i$. Using \eqref{eq:assignment-y} 
and $x_{3i}=(\lambda-1)x_{3i+1}-x_{3i+2}$ gives
\[
P_i - Q_i = \frac{(\lambda-3)^2}{(\lambda-1)^2} (x_{3i-1} - x_{3i+2}) (x_{3i-1} + 2(\lambda-1) x_{3i+1} - x_{3i+2}).
\]
Furthermore, by $(\lambda^2-\lambda-2) x_{3i+1} = x_{3i-1} + \lambda x_{3i+2}$ we see
\[
x_{3i-1} + 2(\lambda-1) x_{3i+1} - x_{3i+2} < \frac{\lambda^2-\lambda+2}{(\lambda+1)(\lambda-2)} (x_{3i-1}+x_{3i+2}).
\]
It follows that
\begin{equation}\label{eq:Pi-Qi}
P_i - Q_i < \frac{(\lambda-3)^2 (\lambda^2 - \lambda + 2)}{(\lambda + 1) (\lambda - 1)^2 (\lambda - 2)}
(x_{3i-1}^2 - x_{3i+2}^2).
\end{equation}
Putting \eqref{eq:yLy}, \eqref{eq:3-lambda-Qi} and \eqref{eq:Pi-Qi} together, we derive, in light of  
$y_1 = x_2$, that
\begin{align*}
& (y_1^2 + \bm{y}^{\mathrm{T}} L(\widetilde{G}) \bm{y}) - (3 - \lambda) \\
< & ~ \frac{(\lambda-3)^2(\lambda^2-\lambda+2)}{(\lambda-1)^2(\lambda+1)(\lambda-2)} (x_2^2-x_{3k-1}^2) 
+ \frac{(\lambda - 3)^2 (\lambda + 1)^2}{2} x_{3k}^2 - (\lambda^2 -2\lambda + 3) x_1^2 \\
= & ~ (\lambda-3)^2 \bigg(\frac{(\lambda+2) (2\lambda^3-2\lambda^2-3\lambda-3)}{4(\lambda+1)(\lambda-2)} x_{3k}^2
- \frac{2\lambda^5 - 7\lambda^4 + 13\lambda^3 - 23\lambda^2 + 13\lambda - 6}{(\lambda - 3)^2(\lambda+1) (\lambda-1)^2 (\lambda-2)} x_1^2\bigg) \\
=: & ~ (\lambda - 3)^2 g(\lambda).
\end{align*}

Finally, applying \eqref{eq:Delta-rho} to $\widetilde{G}$ and $\bm{y}$ we have 
\begin{align*}
3 - \lambda_1(\widetilde{G}) 
& \leq \frac{(3 - \lambda) + \big(y_1^2 + \bm{y}^{\mathrm{T}} L(\widetilde{G}) \bm{y} - (3 - \lambda)\big)}{\|\bm{y}\|_2^2} \\
& < (3 - \lambda)\cdot \frac{1 + (3 - \lambda)\cdot g(\lambda)}{1 + (3 - \lambda)\cdot f(\lambda)} \\
& < 3 - \lambda.
\end{align*}
The last inequality follows from the fact $g(\lambda) < f(\lambda)$ whenever $\lambda>2.8$.
This completes the proof of \autoref{thm:proof-conj-degree-sequ-for-degree3}. Using the same 
argument as above, we can also obtain a contradiction for the case $n\equiv 3 \!\!\mod 4$.
This completes the proof of the theorem.
\end{proof}

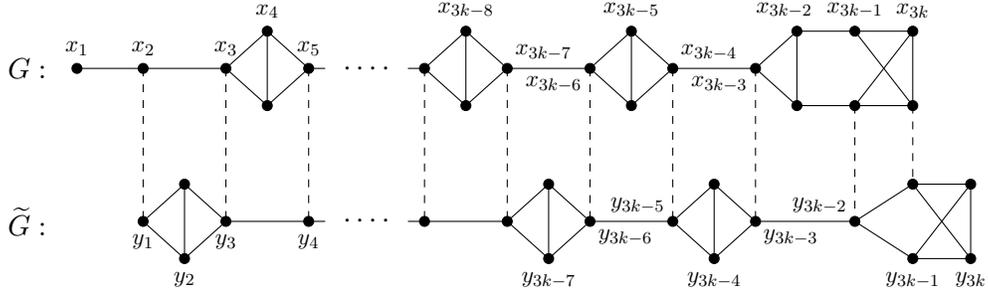
\begin{figure}[htbp]
\centering
\begin{tikzpicture}[scale=1.1]
% above
\coordinate (v1) at (0,0);
\coordinate (v2) at (0.8,0);
\coordinate (v3) at (1.8,0);
\coordinate (v4) at (2.8,0);
\coordinate (v5) at (4.2,0);
\coordinate (v6) at (5.2,0);
\coordinate (v7) at (6.2,0);
\coordinate (v8) at (7.2,0);
\coordinate (v9) at (8.2,0);

\coordinate (u1) at (2.3,0.45);
\coordinate (u2) at (4.7,0.45);
\coordinate (u3) at (6.7,0.45);
\coordinate (u4) at (8.7,0.45);
\coordinate (u5) at (9.4,0.45);
\coordinate (u6) at (10.1,0.45);

\coordinate (w1) at (2.3,-0.45);
\coordinate (w2) at (4.7,-0.45);
\coordinate (w3) at (6.7,-0.45);
\coordinate (w4) at (8.7,-0.45);
\coordinate (w5) at (9.4,-0.45);
\coordinate (w6) at (10.1,-0.45);

% below
\coordinate (a1) at (0.8,-1.85);
\coordinate (a2) at (1.8,-1.85);
\coordinate (a3) at (2.8,-1.85);
\coordinate (a4) at (4.2,-1.85);
\coordinate (a5) at (5.2,-1.85);
\coordinate (a6) at (6.2,-1.85);
\coordinate (a7) at (7.2,-1.85);
\coordinate (a8) at (8.2,-1.85);
\coordinate (a9) at (9.4,-1.85);

\coordinate (b1) at (1.3,-1.4);
\coordinate (b2) at (5.7,-1.4);
\coordinate (b3) at (7.7,-1.4);
\coordinate (b4) at (10.1,-1.4);
\coordinate (b5) at (10.8,-1.4);

\coordinate (c1) at (1.3,-2.3);
\coordinate (c2) at (5.7,-2.3);
\coordinate (c3) at (7.7,-2.3);
\coordinate (c4) at (10.1,-2.3);
\coordinate (c5) at (10.8,-2.3);

% above
\node[above=2pt, scale=0.8] at (v1) {$x_1$};
\node[above=2pt, scale=0.8] at (v2) {$x_2$};
\node[above=2pt, scale=0.8] at (v3) {$x_3$};
\node[above=2pt, scale=0.8] at (u1) {$x_4$};
\node[above=2pt, scale=0.8] at (u2) {$x_{3k-8}$};
\node[above=2pt, scale=0.8] at (v4) {$x_5$};
\node[anchor=south west, scale=0.8] at (v6) {$x_{3k-7}$};
\node[anchor=north east, scale=0.8] at (v7) {$x_{3k-6}$};
\node[above=2pt, scale=0.8] at (u3) {$x_{3k-5}$};
\node[anchor=south west, scale=0.8] at (v8) {$x_{3k-4}$};
\node[anchor=north east, scale=0.8] at (v9) {$x_{3k{-}3}$};
\node[above=2pt, scale=0.8] at (8.55,0.45) {$x_{3k-2}$};
\node[above=2pt, scale=0.8] at (u5) {$x_{3k-1}$};
\node[above=2pt, scale=0.8] at (u6) {$x_{3k}$};
\node at (3.5,0) {$\cdots\cdot$};
\node at (-0.6,0) {$G:$};

% below 
\node[below=2pt, scale=0.8] at (a1) {$y_1$};
\node[below=2pt, scale=0.8] at (c1) {$y_2$};
\node[below=2pt, scale=0.8] at (a2) {$y_3$};
\node[below=2pt, scale=0.8] at (a3) {$y_4$};
\node[below=2pt, scale=0.8] at (c2) {$y_{3k-7}$};
\node[anchor=north west, scale=0.8] at (a6) {$y_{3k-6}$};
\node[anchor=south east, scale=0.8] at (a7) {$y_{3k-5}$};
\node[below=2pt, scale=0.8] at (c3) {$y_{3k-4}$};
\node[anchor=north west, scale=0.8] at (a8) {$y_{3k-3}$};
\node[anchor=south east, scale=0.8] at (a9) {$y_{3k-2}$};
\node[below=2pt, scale=0.8] at (c4) {$y_{3k-1}$};
\node[below=2pt, scale=0.8] at (c5) {$y_{3k}$};
\node at (3.5,-1.85) {$\cdots\cdot$};
\node at (-0.6,-1.85) {$\widetilde{G}:$};

% above
\draw (v1) -- (v2) -- (v3) -- (u1) -- (v4) -- (w1) -- (v3);
\draw (v5) -- (u2) -- (v6) -- (w2) -- (v5);
\draw (u1) -- (w1);
\draw (u2) -- (w2);
\draw (v6) -- (v7) -- (u3) -- (v8) -- (w3) -- (v7);
\draw (u3) -- (w3);
\draw (v8) -- (v9) -- (u4) -- (u5) -- (u6) -- (w6) -- (w5) -- (w4) -- (v9);
\draw (u4) -- (w4);
\draw (u5) -- (w6);
\draw (w5) -- (u6);
\draw (v4) -- (3,0);
\draw (v5) -- (4,0);

% below
\draw (a1) -- (b1) -- (a2) -- (a3) -- (3,-1.85);
\draw (a1) -- (c1) -- (b1);
\draw (c1) -- (a2);
\draw (4,-1.85) -- (a4) -- (a5) -- (b2) -- (a6) -- (c2) -- (a5);
\draw (a6) -- (a7) -- (b3) -- (a8) -- (c3) -- (a7);
\draw (a8) -- (a9) -- (b4) -- (c5) -- (b5) -- (c4) -- (a9);
\draw (b4) -- (b5);
\draw (b2) -- (c2);
\draw (b3) -- (c3);
\draw (c4) -- (c5);

% between
\draw[dashed] (v2) -- (a1);
\draw[dashed] (v3) -- (a2);
\draw[dashed] (v4) -- (a3);
\draw[dashed] (v5) -- (a4);
\draw[dashed] (v6) -- (a5);
\draw[dashed] (v7) -- (a6);
\draw[dashed] (v8) -- (a7);
\draw[dashed] (v9) -- (a8);
\draw[dashed] (w5) -- (a9);
\draw[dashed] (w6) -- (b4);

\foreach \i in {u1,u2,u3,u4,u5,u6,v1,v2,v3,v4,v5,v6,v7,v8,v9,w1,w2,w3,%
w4,w5,w6,a1,a2,a3,a4,a5,a6,a7,a8,a9,b1,b2,b3,b4,b5,c1,c2,c3,c4,c5}
\filldraw (\i) circle (0.06);
\end{tikzpicture}
\caption{Graphs $G$ and $\widetilde{G}$ in Case 2 in the proof of 
\autoref{thm:proof-conj-degree-sequ-for-degree3}, where two ends of 
each vertical dashed line have the same component}
\label{fig:figure-in-thm3-1}
\end{figure}

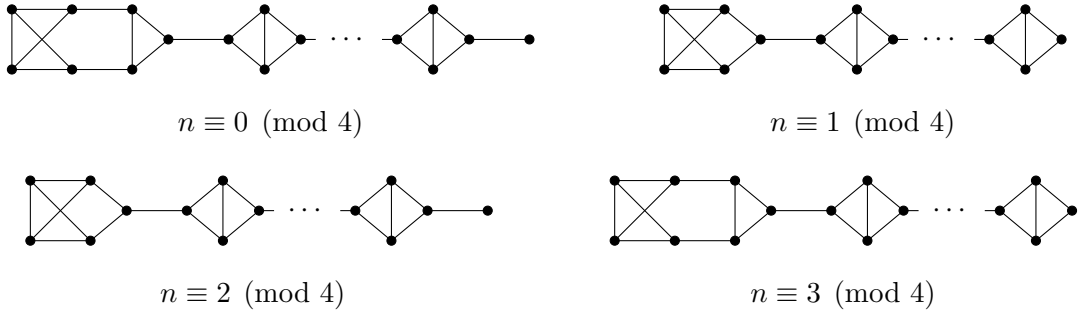
\begin{figure}[htbp]
\begin{minipage}{.51\textwidth}
\centering
\begin{tikzpicture}[scale=0.8]
\coordinate (v1) at (0,0.5);
\coordinate (v2) at (1,0.5);
\coordinate (v3) at (2,0.5);
\coordinate (v4) at (4.2,0.5);
\coordinate (v5) at (7,0.5);

\coordinate (w1) at (0,-0.5);
\coordinate (w2) at (1,-0.5);
\coordinate (w3) at (2,-0.5);
\coordinate (w4) at (4.2,-0.5);
\coordinate (w5) at (7,-0.5);

\coordinate (u1) at (2.6,0);
\coordinate (u2) at (3.6,0);
\coordinate (u3) at (4.8,0);
\coordinate (u4) at (6.4,0);
\coordinate (u5) at (7.6,0);
\coordinate (u6) at (8.6,0);

\draw (v1) -- (v2) -- (v3) -- (u1) -- (w3) -- (w2) -- (w1) -- (v1) -- (w2);
\draw (v2) -- (w1);
\draw (v3) -- (w3);
\draw (v4) -- (w4);
\draw (v5) -- (w5);
\draw (u1) -- (u2) -- (v4) -- (u3) -- (w4) -- (u2);
\draw (u3) -- (5.05,0);
\draw (6.15,0) -- (u4) -- (v5) -- (u5) -- (w5) -- (u4);
\draw (u5) -- (u6);

\foreach \i in {u1,u2,u3,u4,u5,u6,v1,v2,v3,v4,v5,w1,w2,w3,w4,w5}
\filldraw (\i) circle (0.075);

\node at (5.6,0) {$\cdots$};
\node at (4.3,-1.4) {$n\equiv 0\pmod{4}$};
\end{tikzpicture}
\end{minipage}
\begin{minipage}{.48\textwidth}
\centering
\begin{tikzpicture}[scale=0.8]
\coordinate (u1) at (1.6,0);
\coordinate (u2) at (2.6,0);
\coordinate (u3) at (3.8,0);
\coordinate (u4) at (5.4,0);
\coordinate (u5) at (6.6,0);

\coordinate (v1) at (0,0.5);
\coordinate (v2) at (1,0.5);
\coordinate (v3) at (3.2,0.5);
\coordinate (v4) at (6,0.5);

\coordinate (w1) at (0,-0.5);
\coordinate (w2) at (1,-0.5);
\coordinate (w3) at (3.2,-0.5);
\coordinate (w4) at (6,-0.5);

\draw (v1) -- (v2) -- (u1) -- (w2) -- (w1) -- (v1) -- (w2);
\draw (u1) -- (u2) -- (v3) -- (u3) -- (w3) -- (u2);
\draw (5.15,0) -- (u4) -- (v4) -- (u5) -- (w4) -- (u4);
\draw (v2) -- (w1);
\draw (v3) -- (w3);
\draw (v4) -- (w4); 
\draw (u3) -- (4.05,0);

\foreach \i in {u1,u2,u3,u4,u5,v1,v2,v3,v4,w1,w2,w3,w4}
\filldraw (\i) circle (0.075);
\node at (4.6,0) {$\cdots$};
\node at (3.3,-1.4) {$n\equiv 1 \pmod{4}$};
\end{tikzpicture}
\end{minipage}
\vspace{3mm}

\begin{minipage}{.49\textwidth}
\centering
\begin{tikzpicture}[scale=0.8]
\coordinate (u1) at (1.6,0);
\coordinate (u2) at (2.6,0);
\coordinate (u3) at (3.8,0);
\coordinate (u4) at (5.4,0);
\coordinate (u5) at (6.6,0);
\coordinate (u6) at (7.6,0);

\coordinate (v1) at (0,0.5);
\coordinate (v2) at (1,0.5);
\coordinate (v3) at (3.2,0.5);
\coordinate (v4) at (6,0.5);

\coordinate (w1) at (0,-0.5);
\coordinate (w2) at (1,-0.5);
\coordinate (w3) at (3.2,-0.5);
\coordinate (w4) at (6,-0.5);

\draw (v1) -- (v2) -- (u1) -- (w2) -- (w1) -- (v1) -- (w2);
\draw (u1) -- (u2) -- (v3) -- (u3) -- (w3) -- (u2);
\draw (5.15,0) -- (u4) -- (v4) -- (u5) -- (w4) -- (u4);
\draw (v2) -- (w1);
\draw (v3) -- (w3);
\draw (v4) -- (w4); 
\draw (u3) -- (4.05,0);
\draw (u5) -- (u6);

\foreach \i in {u1,u2,u3,u4,u5,u6,v1,v2,v3,v4,w1,w2,w3,w4}
\filldraw (\i) circle (0.075);
\node at (4.6,0) {$\cdots$};
\node at (3.7,-1.4) {$n\equiv 2 \pmod{4}$};
\end{tikzpicture}
\end{minipage}
\begin{minipage}{.49\textwidth}
\centering
\begin{tikzpicture}[scale=0.8]
\coordinate (v1) at (0,0.5);
\coordinate (v2) at (1,0.5);
\coordinate (v3) at (2,0.5);
\coordinate (v4) at (4.2,0.5);
\coordinate (v5) at (7,0.5);

\coordinate (w1) at (0,-0.5);
\coordinate (w2) at (1,-0.5);
\coordinate (w3) at (2,-0.5);
\coordinate (w4) at (4.2,-0.5);
\coordinate (w5) at (7,-0.5);

\coordinate (u1) at (2.6,0);
\coordinate (u2) at (3.6,0);
\coordinate (u3) at (4.8,0);
\coordinate (u4) at (6.4,0);
\coordinate (u5) at (7.6,0);

\draw (v1) -- (v2) -- (v3) -- (u1) -- (w3) -- (w2) -- (w1) -- (v1) -- (w2);
\draw (v2) -- (w1);
\draw (v3) -- (w3);
\draw (v4) -- (w4);
\draw (v5) -- (w5);
\draw (u1) -- (u2) -- (v4) -- (u3) -- (w4) -- (u2);
\draw (u3) -- (5.05,0);
\draw (6.15,0) -- (u4) -- (v5) -- (u5) -- (w5) -- (u4);

\foreach \i in {u1,u2,u3,u4,u5,v1,v2,v3,v4,v5,w1,w2,w3,w4,w5}
\filldraw (\i) circle (0.075);

\node at (5.6,0) {$\cdots$};
\node at (3.8,-1.4) {$n\equiv 3 \pmod{4}$};
\end{tikzpicture}
\end{minipage}
\caption{The unique extremal graph in $\mathcal{G}(n,3)$ depending on the value of $n\!\!\mod{4}$}
\label{fig:maximum-spectral-radius-Deltea-3}
\end{figure}

\subsection{Uniqueness of the extremal graph for $\Delta=3$}

\begin{theorem}\label{thm:extremal-graph-degree-3}
Among all connected nonregular graphs with $n$ vertices and maximum degree $3$, the graph 
depicted in Fig.\,\ref{fig:maximum-spectral-radius-Deltea-3} is the unique graph attaining 
maximum spectral radius.
\end{theorem}

\begin{proof}
According to \autoref{lem:transfer-first-few-vertices-Delta-3}, \autoref{lem:induced-middle-vertices-Delta=3} 
and \autoref{thm:proof-conj-degree-sequ-for-degree3}, any graph $G$ in $\mathcal{G}(n,3)$
can be transferred by proper local switchings into the graph depicted in 
Fig.\,\ref{fig:maximum-spectral-radius-Deltea-3}, say $H$. Let $\ls_1,\ldots,\ls_t$ be a sequence 
of proper switchings which turn $G$ into $H$. Consider the graphs $G=G_0,G_1,\ldots,G_t=H$ 
in which $G_i$ is obtained from $G_{i-1}$ by applying $\ls_i$. Since $G\in\mathcal{G}(n,3)$, 
we have $\lambda_1(G) = \lambda_1(H)$, and therefore $\lambda_1(G_i) = \lambda_1(H)$ for each $i\in [t]$.
Let $\bm{x}$ be the Perron vector of $G_{t-1}$ and $\ls_t = \ls (a,b,c,d)$. By \autoref{lem:local-switching}, 
we have $x_a=x_d$, $x_b=x_c$, and $\bm{x}$ is also the Perron vector of $H$. On the other hand, 
a simple argument implies that the components of $\bm{x}$ on the vertices of $H$ are strictly 
decreasing from left to right, and the vertices on the same vertical line have the same components.
Hence, $a$, $d$ (and $b$, $c$) must lie on the same vertical line in Fig.\,\ref{fig:maximum-spectral-radius-Deltea-3}.
Observe that $ac\in E(H)$ and $bd\in E(H)$. Then $\{a,b,c,d\}\subset \{v_1,v_2,\ldots,v_6\}$. 
Note that $\ls (a,c,b,d)$ is the reverse of $\ls (a,b,c,d)$, and therefore, when applied 
on $H$, yields $G_{t-1}$. Hence, $G_{t-1}$ must be isomorphic to $H$. Likewise, for 
$i=0,1,\ldots,t-2$, $G_i$ is isomorphic to $H$. This completes the proof of the theorem.
\end{proof}

\section{Structure of extremal graphs in $\mathcal{G}(n,4)$}
\label{sec:structure-Delta=4}

In this section, we determine the structure of extremal graphs $G$ in $\mathcal{G}(n,4)$,
and confirm \autoref{conj:degree-sequence} for $\Delta=4$.

\subsection{Forbidden induced subgraphs}

We first show that certain induced subgraphs cannot appear in $G$.

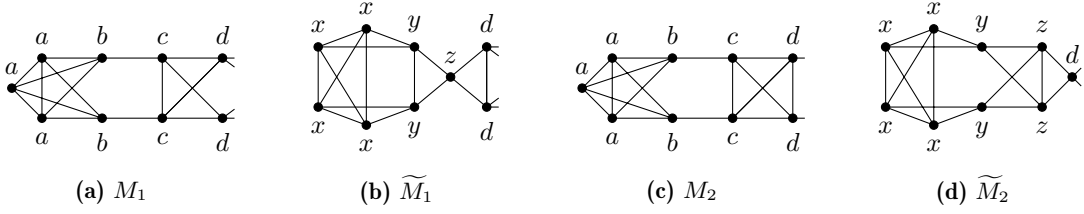
\begin{figure}[htbp]
\centering
\subcaptionbox{$M_1$ \label{subfig:M1}}[.26\textwidth]{%
\begin{tikzpicture}[scale=0.8]
\coordinate (u1) at (0,0);
\coordinate (v1) at (0.5,0.5);
\coordinate (v2) at (1.5,0.5);
\coordinate (v3) at (2.5,0.5);
\coordinate (v4) at (3.5,0.5);
\coordinate (w1) at (0.5,-0.5);
\coordinate (w2) at (1.5,-0.5);
\coordinate (w3) at (2.5,-0.5);
\coordinate (w4) at (3.5,-0.5);

\draw (u1) -- (v1) -- (v2) -- (v3) -- (v4) -- (w3) -- %
(w2) -- (w1) -- (u1) -- (v2) -- (w1) -- (v1) -- (w2) -- (u1);
\draw (3.7,0.5) -- (v4) -- (w3) -- (v3) -- (w4) -- (3.7,-0.5);
\draw (w3) -- (w4) -- (3.7,-0.35);
\draw (v4) -- (3.7,0.35);

\node[above=2pt, scale=0.9] at (v4) {$d$};
\node[below=2pt, scale=0.9] at (w4) {$d$};
\node[above=2pt, scale=0.9] at (v3) {$c$};
\node[below=2pt, scale=0.9] at (w3) {$c$};
\node[above=2pt, scale=0.9] at (v2) {$b$};
\node[below=2pt, scale=0.9] at (w2) {$b$};
\node[above=2pt, scale=0.9] at (v1) {$a$};
\node[below=2pt, scale=0.9] at (w1) {$a$};
\node[above=2pt, scale=0.9] at (u1) {$a$};

\foreach \i in {u1,v1,v2,v3,v4,w1,w2,w3,w4}
\filldraw (\i) circle (0.07);
\end{tikzpicture}
}
\subcaptionbox{$\widetilde{M}_1$ \label{subfig:tilde-M1}}[.21\textwidth]{%
\begin{tikzpicture}[scale=0.8]
\coordinate (u1) at (2.2,0);
\coordinate (v1) at (0,0.5);
\coordinate (v2) at (0.8,0.8);
\coordinate (v3) at (1.6,0.5);
\coordinate (v4) at (2.8,0.5);
\coordinate (w1) at (0,-0.5);
\coordinate (w2) at (0.8,-0.8);
\coordinate (w3) at (1.6,-0.5);
\coordinate (w4) at (2.8,-0.5);

\draw (v1) -- (v2) -- (v3) -- (u1) -- (v4) -- %
(w4) -- (u1) -- (w3) -- (w2) -- (w1) -- (v1) -- (w2) -- (v2) -- (w1);
\draw (v1) -- (v3) -- (w3) -- (w1);
\draw (3,0.35) -- (v4) -- (w4) -- (3,-0.35);
\draw (v4) -- (3,0.5);
\draw (w4) -- (3,-0.5);

\node[above=2pt, scale=0.9] at (v1) {$x$};
\node[below=2pt, scale=0.9] at (w1) {$x$};
\node[above=2pt, scale=0.9] at (v2) {$x$};
\node[below=2pt, scale=0.9] at (w2) {$x$};
\node[above=2pt, scale=0.9] at (v3) {$y$};
\node[below=2pt, scale=0.9] at (w3) {$y$};
\node[above=2pt, scale=0.9] at (u1) {$z$};
\node[above=2pt, scale=0.9] at (v4) {$d$};
\node[below=2pt, scale=0.9] at (w4) {$d$};

\foreach \i in {u1,v1,v2,v3,v4,w1,w2,w3,w4}
\filldraw (\i) circle (0.07);
\end{tikzpicture}
}
\subcaptionbox{$M_2$ \label{subfig:M2}}[.26\textwidth]{%
\begin{tikzpicture}[scale=0.8]
\coordinate (u1) at (0,0);
\coordinate (v1) at (0.5,0.5);
\coordinate (v2) at (1.5,0.5);
\coordinate (v3) at (2.5,0.5);
\coordinate (v4) at (3.5,0.5);
\coordinate (w1) at (0.5,-0.5);
\coordinate (w2) at (1.5,-0.5);
\coordinate (w3) at (2.5,-0.5);
\coordinate (w4) at (3.5,-0.5);

\draw (u1) -- (v1) -- (v2) -- (v3) -- (v4) -- (w3) -- %
(w2) -- (w1) -- (u1) -- (v2) -- (w1) -- (v1) -- (w2) -- (u1);
\draw (3.7,0.5) -- (v4) -- (w3) -- (v3) -- (w4) -- (3.7,-0.5);
\draw (w3) -- (w4) -- (v4);

\node[above=2pt, scale=0.9] at (v4) {$d$};
\node[below=2pt, scale=0.9] at (w4) {$d$};
\node[above=2pt, scale=0.9] at (v3) {$c$};
\node[below=2pt, scale=0.9] at (w3) {$c$};
\node[above=2pt, scale=0.9] at (v2) {$b$};
\node[below=2pt, scale=0.9] at (w2) {$b$};
\node[above=2pt, scale=0.9] at (v1) {$a$};
\node[below=2pt, scale=0.9] at (w1) {$a$};
\node[above=2pt, scale=0.9] at (u1) {$a$};

\foreach \i in {u1,v1,v2,v3,v4,w1,w2,w3,w4}
\filldraw (\i) circle (0.07);
\end{tikzpicture}
}
\subcaptionbox{$\widetilde{M}_2$ \label{subfig:tilde-M2}}[.22\textwidth]{%
\begin{tikzpicture}[scale=0.8]
\coordinate (u1) at (3.1,0);
\coordinate (v1) at (0,0.5);
\coordinate (v2) at (0.8,0.8);
\coordinate (v3) at (1.6,0.5);
\coordinate (v4) at (2.6,0.5);
\coordinate (w1) at (0,-0.5);
\coordinate (w2) at (0.8,-0.8);
\coordinate (w3) at (1.6,-0.5);
\coordinate (w4) at (2.6,-0.5);

\draw (v1) -- (v2) -- (v3) -- (v4) -- (u1) -- (w4) -- (w3) -- (w2) -- %
(w1) -- (v1) -- (w2) -- (v2) -- (w1) -- (w3) -- (v4) -- (w4) -- (v3) -- (v1);
\draw (3.25,0.15) -- (u1) -- (3.25,-0.15);

\node[above=2pt, scale=0.9] at (v1) {$x$};
\node[below=2pt, scale=0.9] at (w1) {$x$};
\node[above=2pt, scale=0.9] at (v2) {$x$};
\node[below=2pt, scale=0.9] at (w2) {$x$};
\node[above=2pt, scale=0.9] at (v3) {$y$};
\node[below=2pt, scale=0.9] at (w3) {$y$};
\node[above=2pt, scale=0.9] at (u1) {$d$};
\node[above=2pt, scale=0.9] at (v4) {$z$};
\node[below=2pt, scale=0.9] at (w4) {$z$};

\foreach \i in {u1,v1,v2,v3,v4,w1,w2,w3,w4}
\filldraw (\i) circle (0.07);
\end{tikzpicture}
}
\caption{Forbidden graphs $M_1$ and $M_2$}
\label{fig:M1-M2}
\end{figure}

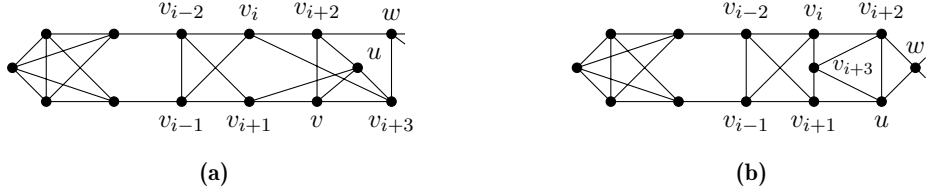
\begin{figure}[htbp]
\centering
\subcaptionbox{\label{subfig:graph1-in-forbbidden-M1}}[.45\textwidth]{%
\begin{tikzpicture}[scale=0.9]
\coordinate (u1) at (0,0);
\coordinate (u2) at (5.1,0);
\coordinate (v1) at (0.5,0.5);
\coordinate (w1) at (0.5,-0.5);
\coordinate (v2) at (1.5,0.5);
\coordinate (w2) at (1.5,-0.5);
\coordinate (v3) at (2.5,0.5);
\coordinate (w3) at (2.5,-0.5);
\coordinate (v4) at (3.5,0.5);
\coordinate (w4) at (3.5,-0.5);
\coordinate (v5) at (4.5,0.5);
\coordinate (w5) at (4.5,-0.5);
\coordinate (v6) at (5.6,0.5);
\coordinate (w6) at (5.6,-0.5);

\draw (u1) -- (v1) -- (v2) -- (v3) -- (v4) -- (v5) -- (v6) -- (w6) -- (w5) -- 
(w4) -- (w3) -- (w2) -- (w1) -- (u1) -- (v2) -- (w1) -- (v1) -- (w2) -- (u1);
\draw (v3) -- (w4) -- (u2) -- (v5) -- (w5) -- (u2) -- (w6) -- (v4) -- (w3) -- (v3);
\draw (5.8,0.5) -- (v6) -- (5.8,0.35);

\node[above=2pt, scale=0.9] at (v3) {$v_{i-2}$};
\node[below=2pt, scale=0.9] at (w3) {$v_{i-1}$};
\node[above=2pt, scale=0.9] at (v4) {$v_i$};
\node[below=2pt, scale=0.9] at (w4) {$v_{i+1}$};
\node[above=2pt, scale=0.9] at (v5) {$v_{i+2}$};
\node[below=2pt, scale=0.9] at (w5) {$v$};
\node[above=2pt, scale=0.9] at (v6) {$w$};
\node[below=2pt, scale=0.9] at (w6) {$v_{i+3}$};
\node[anchor=south west, scale=0.9] at (u2) {$u$};

\foreach \i in {u1,u2,v1,v2,v3,v4,v5,v6,w1,w2,w3,w4,w5,w6}
\filldraw (\i) circle (0.07);
\end{tikzpicture}
}
\subcaptionbox{\label{subfig:graph2-in-forbbidden-M1}}[.45\textwidth]{%
\begin{tikzpicture}[scale=0.9]
\coordinate (u1) at (0,0);
\coordinate (u2) at (3.5,0);
\coordinate (u3) at (5,0);
\coordinate (v1) at (0.5,0.5);
\coordinate (w1) at (0.5,-0.5);
\coordinate (v2) at (1.5,0.5);
\coordinate (w2) at (1.5,-0.5);
\coordinate (v3) at (2.5,0.5);
\coordinate (w3) at (2.5,-0.5);
\coordinate (v4) at (3.5,0.5);
\coordinate (w4) at (3.5,-0.5);
\coordinate (v5) at (4.5,0.5);
\coordinate (w5) at (4.5,-0.5);

\draw (u1) -- (v1) -- (v2) -- (v3) -- (v4) -- (v5) --% 
(w5) -- (w4) -- (w3) -- (w2) -- (w1) -- (u1) -- (v2) -- (w1) --%
(v1) -- (w2) -- (u1);
\draw (v3) -- (w3) -- (v4) -- (u2) -- (w4) -- (v3);
\draw (v5) -- (u2) -- (w5);
\draw (v5) -- (u3) -- (w5);
\draw (5.15,0.15) -- (u3) -- (5.15,-0.15);

\node[above=2pt, scale=0.9] at (v3) {$v_{i-2}$};
\node[below=2pt, scale=0.9] at (w3) {$v_{i-1}$};
\node[above=2pt, scale=0.9] at (v4) {$v_i$};
\node[below=2pt, scale=0.9] at (w4) {$v_{i+1}$};
\node[right=4pt, scale=0.8] at (u2) {$v_{i+3}$};
\node[above=2pt, scale=0.9] at (v5) {$v_{i+2}$};
\node[below=2pt, scale=0.9] at (w5) {$u$};
\node[above=2pt, scale=0.9] at (u3) {$w$};

\foreach \i in {u1,u2,u3,v1,v2,v3,v4,v5,w1,w2,w3,w4,w5}
\filldraw (\i) circle (0.07);
\end{tikzpicture}
}
\caption{Two possible situations in \autoref{lem:forbidden-M1}}
\label{fig:graph12-in-forbidden-M1}
\end{figure}

\begin{lemma}\label{lem:forbidden-M1}
The graph $G$ does not contain the graph $M_1$ shown in Fig.\,\ref{fig:M1-M2} as an induced subgraph.
\end{lemma}

\begin{proof}
For a contradiction suppose that $G$ contains $M_1$ as an induced subgraph. We first show 
that the two rightmost vertices of $M_1$, say $v_i$ and $v_{i+1}$, share two common neighbors 
other than $v_{i-2}$, $v_{i-1}$ (see Fig.\,\ref{fig:graph12-in-forbidden-M1}). By 
\autoref{coro:common-use-1} we may assume $v_i\sim v_{i+2}$ and $v_i\sim v_{i+3}$. Let $u$, 
$v\in N(v_{i+1})\setminus\{v_{i-2}, v_{i-1}\}$. If $|\{v_{i+2}, v_{i+3}\}\cap \{u, v\}| = 0$,
then $u\sim v$, for otherwise, $\ls (v_{i+1},v, v_{i+2},u)$ connects $v_{i+1}$ to $v_{i+2}$,
and therefore $\{v_{i+2}, v_{i+3}\}\cap \{u, v\} \neq\emptyset$. So we arrive at the graph 
as illustrated in Fig.\,\ref{fig:graph12-in-forbidden-M1}\,\subref{subfig:graph1-in-forbbidden-M1}.
Also, $\ls (v_{i+1},u, v_{i+2},w)$ makes $v_{i+1}\sim v_{i+2}$. Hence, we assume 
$|\{v_{i+2}, v_{i+3}\}\cap \{u, v\}| = 1$ (see Fig.\,\ref{fig:graph12-in-forbidden-M1}\,\subref{subfig:graph2-in-forbbidden-M1}). 
Without loss of generality, assume that $v_{i+1}\sim v_{i+3}$. By eigenvalue equations, 
\[ 
(\lambda_1(G) + 1) (x_{i+2} - x_u) = x_i - x_{i+1}.
\] 
On the other hand, we have 
\[ 
\lambda_1(G) (x_i - x_{i+1}) = x_{i+2} - x_u.
\] 
So we deduce that $x_i=x_{i+1}$ and $x_{i+2}=x_u$. In particular, $x_{i+2}=x_{i+3}$. Again 
by eigenvalue equations for $v_{i+2}$ and $v_{i+3}$, we see $x_{i+1}=x_w$. Hence, we 
immediately have $x_i=\cdots=x_u=x_w$, yielding $\lambda_1(G)=4$, a contradiction.
Therefore, $v_i$ and $v_{i+1}$ share four neighbors. Hence, by eigenvalue equations, 
we can label the components of the Perron vector $\bm{x}$ on $M_1$ as depicted in 
Fig.\,\ref{fig:M1-M2}\,\subref{subfig:M1}. 

Next, we will construct a graph $\widetilde{G}$ such that $\lambda_1(\widetilde{G}) > \lambda_1(G)$, 
which yields a contradiction. Before doing this, set $\lambda:=\lambda_1(G)$. By eigenvalue equations we find
\begin{equation}\label{eq:equation-bcd-1-degree4}
b = \frac{\lambda-2}{2} a,~~
c = \frac{\lambda^2-2\lambda-6}{2} a,~~
d = \frac{\lambda^3-3\lambda^2-5\lambda+8}{4} a.
\end{equation}
Now we replace $M_1$ by $\widetilde{M}_1$ to obtain $\widetilde{G}$. We define a vector $\bm{y}$
on $V(\widetilde{G})$ such that its components on $\widetilde{M}_1$ are as given in 
Fig.\,\ref{fig:M1-M2}\,\subref{subfig:tilde-M1}, and on the rest of vertices agree with 
$\bm{x}$. Here, $x$, $y$, $z$ are given by 
\begin{equation}\label{eq:define-xyz-degree4}
\begin{split}
x & = \frac{(\sqrt{2}-1) a}{4} \big( -\lambda^3 + 5\lambda^2 - (2-\sqrt{3})(\sqrt{3}-\sqrt{2})\lambda - (\sqrt{6}-1)(\sqrt{6}-\sqrt{2})\big), \\
y & = \frac{(\sqrt{2}-1) a}{4} \big(-\lambda^3 + 5\lambda^2 + (\sqrt{2}+1)^2\lambda - 4\sqrt{2}-24 \big), \\
z & = \frac{(\sqrt{2}-1) a}{4} \big(-\lambda^3 + (7+2\sqrt{2})\lambda^2 - (3+4\sqrt{2})\lambda - 32-12\sqrt{2} \big).
\end{split}
\end{equation}
Taking \eqref{eq:equation-bcd-1-degree4} into consideration, we see that
\[
(x-y)^2=\frac{3(a-b)^2}{2},~~
(y-z)^2=(b-c)^2,~~
(z-d)^2=2(c-d)^2,
\]
from which we obtain that $\bm{x}^{\mathrm{T}} L(G) \bm{x} = \bm{y}^{\mathrm{T}} L(\widetilde{G}) \bm{y}$.
Using \eqref{eq:Delta-rho} we find that
\[
4-\lambda_1(\widetilde{G}) \leq \frac{4 -\lambda}{\|\bm{y}\|_2^2}.
\]
To complete the proof, we will prove that $\|\bm{y}\|_2>1$ below. Note that
\begin{equation}\label{eq:norm-y-degree4}
\|\bm{y}\|_2^2 = \|\bm{x}\|_2^2 + 4x^2+2y^2+z^2 - (3a^2+2b^2+2c^2).
\end{equation}
By substituting \eqref{eq:equation-bcd-1-degree4} and \eqref{eq:define-xyz-degree4} into 
\eqref{eq:norm-y-degree4}, we derive that $\|\bm{y}\|_2^2 = 1+f(\lambda)a^2$,
where $f(x)$ is given by
\begin{align*}
f(x) & =\frac{21-14\sqrt{2}}{16} x^6 + \frac{68\sqrt{2}-103}{8} x^5 + 
\Big(\frac{475}{16} - \frac{159\sqrt{2}}{8} + \sqrt{3} - \frac{\sqrt{6}}{2}\Big) x^4 \\
& ~~~ + \Big(\frac{455}{8} + \frac{9\sqrt{6}}{2} - 9\sqrt{3} - 34\sqrt{2}\Big) x^3 + 
\Big(\frac{1181\sqrt{2}}{8} + 19\sqrt{3} - \frac{21\sqrt{6}}{2} - \frac{3683}{16}\Big) x^2 \\
& ~~~ + \Big(\frac{21\sqrt{2}}{2} + 24\sqrt{3} - 6\sqrt{6} - 50\Big) x 
- 260\sqrt{2} - 80\sqrt{3} + 32\sqrt{6} + 451.
\end{align*}
It is straightforward to check that $f(0)>0$, $f(2)<0$, $f(3.5)>0$, $f(5)<0$ and $f(7)>0$.
Hence $f(x)=0$ has at least four positive roots. On the other hand, by Descartes's Rule of 
Signs (see \cite{Wang2004}), $f(x) =0$ has at most four positive roots. Hence, $f(x)$ has 
the same sign in interval $(3.5,4)$. Noting that $f(3.5)>0$ and $f(4)=0$, we have $f(\lambda)>0$
as $\lambda>\lambda_1(H_3)>3.7$, which yields that $\lambda_1(\widetilde{G}) > \lambda$, 
a contradiction completing the proof of \autoref{lem:forbidden-M1}.
\end{proof}

The proofs of the following two lemmas are similar to those of \autoref{lem:forbidden-M1}, 
and we defer them to Appendix \ref{appendix}.

\begin{lemma}\label{lem:forbidden-M2}
The graph $G$ does not contain the graph $M_2$ shown in Fig.\,\ref{fig:M1-M2} as an induced subgraph.
\end{lemma}

\begin{lemma}\label{lem:forbidden-M3}
The graph $G$ does not contain the graph $M_3$ shown in Fig.\,\ref{fig:M3-M4} as an induced subgraph.
\end{lemma}

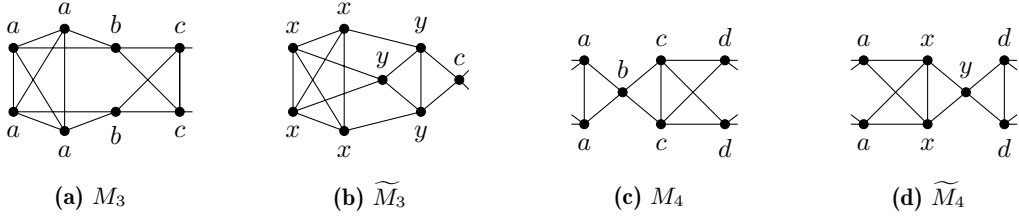
\begin{figure}[htbp]
\centering
\subcaptionbox{$M_3$ \label{subfig:M3}}[.23\textwidth]{%
\begin{tikzpicture}[scale=0.85]
\coordinate (v1) at (0,0.5);
\coordinate (v2) at (0.8,0.8);
\coordinate (v3) at (1.6,0.5);
\coordinate (v4) at (2.6,0.5);

\coordinate (w1) at (0,-0.5);
\coordinate (w2) at (0.8,-0.8);
\coordinate (w3) at (1.6,-0.5);
\coordinate (w4) at (2.6,-0.5);

\draw (v1) -- (v2) -- (v3) -- (v4) -- (w4) -- (w3) -- (w2) -- (w1) -- %
(v1) -- (v3) -- (w4) -- (v4) -- (w3) -- (w1) -- (v2) -- (w2) -- (v1);
\draw (2.8,0.5) -- (v4);
\draw (2.8,-0.5) -- (w4);

\node[above=2pt, scale=0.9] at (v1) {$a$};
\node[below=2pt, scale=0.9] at (w1) {$a$};
\node[above=2pt, scale=0.9] at (v2) {$a$};
\node[below=2pt, scale=0.9] at (w2) {$a$};
\node[above=2pt, scale=0.9] at (v3) {$b$};
\node[below=2pt, scale=0.9] at (w3) {$b$};
\node[above=2pt, scale=0.9] at (v4) {$c$};
\node[below=2pt, scale=0.9] at (w4) {$c$};

\foreach \i in {v1,v2,v3,v4,w1,w2,w3,w4}
\filldraw (\i) circle (0.07);
\end{tikzpicture}
}
\subcaptionbox{$\widetilde{M}_3$ \label{subfig:tilde-M3}}[.23\textwidth]{%
\begin{tikzpicture}[scale=0.85]
\coordinate (u1) at (1.4,0);
\coordinate (u2) at (2.6,0);
\coordinate (v1) at (0,0.5);
\coordinate (v2) at (0.8,0.8);
\coordinate (v3) at (2,0.5);
\coordinate (w1) at (0,-0.5);
\coordinate (w2) at (0.8,-0.8);
\coordinate (w3) at (2,-0.5);

\draw (v1) -- (v2) -- (v3) -- (u2) -- (w3) -- (w2) -- (w1) -- %
(v1) -- (w2) -- (v2) -- (w1) -- (u1) -- (v1);
\draw (u1) -- (v3) -- (w3) -- (u1);
\draw (2.75,0.15) -- (u2) -- (2.75,-0.15);

\node[above=2pt, scale=0.9] at (v1) {$x$};
\node[below=2pt, scale=0.9] at (w1) {$x$};
\node[above=2pt, scale=0.9] at (v2) {$x$};
\node[below=2pt, scale=0.9] at (w2) {$x$};
\node[above=1pt, scale=0.9] at (u1) {$y$};
\node[above=2pt, scale=0.9] at (v3) {$y$};
\node[below=2pt, scale=0.9] at (w3) {$y$};
\node[above=2pt, scale=0.9] at (u2) {$c$};

\foreach \i in {u1,u2,v1,v2,v3,w1,w2,w3}
\filldraw (\i) circle (0.07);
\end{tikzpicture}
}
\subcaptionbox{$M_4$ \label{subfig:M4}}[.23\textwidth]{%
\begin{tikzpicture}[scale=0.85]
\coordinate (u1) at (0,0);
\coordinate (v1) at (-0.6,0.5);
\coordinate (v2) at (0.6,0.5);
\coordinate (v3) at (1.6,0.5);

\coordinate (w1) at (-0.6,-0.5);
\coordinate (w2) at (0.6,-0.5);
\coordinate (w3) at (1.6,-0.5);

\draw (v1) -- (u1) -- (v2) -- (v3) -- (w2) -- (u1) -- (w1) -- (v1); 
\draw (w3) -- (w2) -- (v2) -- (w3);
\draw (-0.8,0.5) -- (v1) -- (-0.8,0.35);
\draw (-0.8,-0.5) -- (w1) -- (-0.8,-0.35);
\draw (1.8,0.5) -- (v3) -- (1.8,0.35);
\draw (1.8,-0.5) -- (w3) -- (1.8,-0.35);

\node[above=2pt, scale=0.9] at (v1) {$a$};
\node[below=2pt, scale=0.9] at (w1) {$a$};
\node[above=2pt, scale=0.9] at (u1) {$b$};
\node[above=2pt, scale=0.9] at (v2) {$c$};
\node[below=2pt, scale=0.9] at (w2) {$c$};
\node[above=2pt, scale=0.9] at (v3) {$d$};
\node[below=2pt, scale=0.9] at (w3) {$d$};

\foreach \i in {u1,v1,v2,v3,w1,w2,w3}
\filldraw (\i) circle (0.07);
\end{tikzpicture}
}
\subcaptionbox{$\widetilde{M}_4$ \label{subfig:tilde-M4}}[.23\textwidth]{%
\begin{tikzpicture}[scale=0.85]
\coordinate (u1) at (1.6,0);
\coordinate (v1) at (0,0.5);
\coordinate (v2) at (1,0.5);
\coordinate (v3) at (2.2,0.5);

\coordinate (w1) at (0,-0.5);
\coordinate (w2) at (1,-0.5);
\coordinate (w3) at (2.2,-0.5);

\draw (v1) -- (v2) -- (u1) -- (w2) -- (w1) -- (v2);
\draw (u1) -- (v3) -- (w3) -- (u1);
\draw (v1) -- (w2) -- (v2);
\draw (-0.2,0.5) -- (v1) -- (-0.2,0.35);
\draw (-0.2,-0.5) -- (w1) -- (-0.2,-0.35);
\draw (2.4,0.5) -- (v3) -- (2.4,0.35);
\draw (2.4,-0.5) -- (w3) -- (2.4,-0.35);

\node[above=2pt, scale=0.9] at (v1) {$a$};
\node[below=2pt, scale=0.9] at (w1) {$a$};
\node[above=2pt, scale=0.9] at (u1) {$y$};
\node[above=2pt, scale=0.9] at (v2) {$x$};
\node[below=2pt, scale=0.9] at (w2) {$x$};
\node[above=2pt, scale=0.9] at (v3) {$d$};
\node[below=2pt, scale=0.9] at (w3) {$d$};

\foreach \i in {u1,v1,v2,v3,w1,w2,w3}
\filldraw (\i) circle (0.07);
\end{tikzpicture}
}
\caption{Forbidden graphs $M_3$ and $M_4$}
\label{fig:M3-M4}
\end{figure}

Below we will forbid more subgraphs. The proof of the next lemma use similar arguments 
as before but involve a few new technical details on the choice of vector $\bm{y}$.

\begin{lemma}\label{lem:forbidden-M4}
The graph $G$ does not contain the graph $M_4$ shown in Fig.\,\ref{fig:M3-M4} as an induced subgraph.
\end{lemma}

\begin{proof}
Assume by contradiction that $G$ contains $M_4$ as an induced subgraph. Using similar 
arguments as the proof in \autoref{lem:forbidden-M1}, the two rightmost vertices 
share four common neighbors. Hence, by eigenvalue equations we can label the components 
of $\bm{x}$ on $M_4$ as depicted in Fig.\,\ref{fig:M3-M4}\,\subref{subfig:M4}.

Now we replace $M_4$ by $\widetilde{M}_4$ to obtain $\widetilde{G}$. Define a vector 
$\bm{y}$ on $V(\widetilde{G})$ such that its components on $\widetilde{M}_4$ are as 
given in Fig.\,\ref{fig:M3-M4}\,\subref{subfig:tilde-M4}, and on the rest of vertices 
agree with $\bm{x}$. Here, $x$, $y$ are given by 
\begin{equation}\label{eq:x-y-in-M4}
x = \frac{2(\lambda a + d)}{\lambda^2 - \lambda - 2},~~
y = \frac{4a + 2(\lambda-1) d}{\lambda^2 - \lambda - 2}.
\end{equation}
By \eqref{eq:Delta-rho} and the eigenvalue equations we obtain that
\begin{align*}
4 - \lambda_1(\widetilde{G}) 
& \leq \frac{\bm{y}^{\mathrm{T}} L(\widetilde{G})\bm{y} - \bm{x}^{\mathrm{T}} L(G)\bm{x} + (4-\lambda)}{\|\bm{y}\|_2^2} \\
& = \frac{ 4(a {-} x)^2 {+} 2(x {-} y)^2 {+} 2(y {-} d)^2 {-} \big(2(a {-} b)^2 {+} 2(b {-} c)^2 {+} 4(c {-} d)^2\big) {+} (4 {-} \lambda)}{\|\bm{y}\|_2^2} \\
& = (4-\lambda)\cdot \frac{(\lambda-2)^2 + 2(a^2-d^2) (4-\lambda)}{(\lambda-2)^2 + 4(a^2-d^2)} \\
& < 4-\lambda.
\end{align*}
The equality in the third line is duo to \eqref{eq:x-y-in-M4}, the eigenvalue equations 
and the fact $\|\bm{y}\|_2^2 = 1 + 2x^2+y^2-b^2-2c^2$. This completes the proof of the lemma.
\end{proof}

With the similar line of reasoning as \autoref{lem:forbidden-M4}, we can prove the 
following lemmas. The details can be found in Appendix \ref{appendix}.

\begin{lemma}\label{lem:forbidden-M5}
The graph $G$ does not contain the graph $M_5$ shown in Fig.\,\ref{fig:M5} as an induced subgraph.
\end{lemma}

\begin{lemma}\label{lem:forbidden-M6}
The graph $G$ does not contain the graph $M_6$ shown in Fig.\,\ref{fig:M6-M7} as an induced subgraph.
\end{lemma}

\begin{lemma}\label{lem:forbidden-M7}
The graph $G$ does not contain the graph $M_7$ shown in Fig.\,\ref{fig:M6-M7} as an induced subgraph.
\end{lemma}

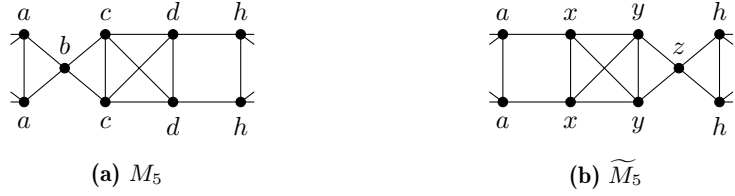
\begin{figure}[htbp]
\centering
\subcaptionbox{$M_5$ \label{subfig:M5}}[.4\textwidth]{%
\begin{tikzpicture}[scale=0.9]
\coordinate (u1) at (0,0);
\coordinate (v1) at (-0.6,0.5);
\coordinate (v2) at (0.6,0.5);
\coordinate (v3) at (1.6,0.5);
\coordinate (v4) at (2.6,0.5);

\coordinate (w1) at (-0.6,-0.5);
\coordinate (w2) at (0.6,-0.5);
\coordinate (w3) at (1.6,-0.5);
\coordinate (w4) at (2.6,-0.5);

\draw (v1) -- (u1) -- (v2) -- (v3) -- (v4) -- (w4) -- (w3) -- (w2) -- (u1) -- (w1) -- (v1); 
\draw (v2) -- (w2) -- (v3) -- (w3) -- (v2);
\draw (-0.8,0.5) -- (v1) -- (-0.8,0.35);
\draw (-0.8,-0.5) -- (w1) -- (-0.8,-0.35);
\draw (2.8,0.5) -- (v4) -- (2.8,0.35);
\draw (2.8,-0.5) -- (w4) -- (2.8,-0.35);

\node[above=2pt, scale=0.9] at (v1) {$a$};
\node[below=2pt, scale=0.9] at (w1) {$a$};
\node[above=2pt, scale=0.9] at (u1) {$b$};
\node[above=2pt, scale=0.9] at (v2) {$c$};
\node[below=2pt, scale=0.9] at (w2) {$c$};
\node[above=2pt, scale=0.9] at (v3) {$d$};
\node[below=2pt, scale=0.9] at (w3) {$d$};
\node[above=2pt, scale=0.9] at (v4) {$h$};
\node[below=2pt, scale=0.9] at (w4) {$h$};

\foreach \i in {u1,v1,v2,v3,v4,w1,w2,w3,w4}
\filldraw (\i) circle (0.07);
\end{tikzpicture}
}
\subcaptionbox{$\widetilde{M}_5$ \label{subfig:tilde-M5}}[.4\textwidth]{%
\begin{tikzpicture}[scale=0.9]
\coordinate (u1) at (2.6,0);
\coordinate (v1) at (0,0.5);
\coordinate (v2) at (1,0.5);
\coordinate (v3) at (2,0.5);
\coordinate (v4) at (3.2,0.5);

\coordinate (w1) at (0,-0.5);
\coordinate (w2) at (1,-0.5);
\coordinate (w3) at (2,-0.5);
\coordinate (w4) at (3.2,-0.5);

\draw (v1) -- (v2) -- (v3) -- (u1) -- (w3) -- (w2) -- (w1) -- (v1);
\draw (v2) -- (w2) -- (v3) -- (w3) -- (v2);
\draw (u1) -- (v4) -- (w4) -- (u1);
\draw (-0.2,0.5) -- (v1) -- (-0.2,0.35);
\draw (-0.2,-0.5) -- (w1) -- (-0.2,-0.35);
\draw (3.4,0.5) -- (v4) -- (3.4,0.35);
\draw (3.4,-0.5) -- (w4) -- (3.4,-0.35);

\node[above=2pt, scale=0.9] at (v1) {$a$};
\node[below=2pt, scale=0.9] at (w1) {$a$};
\node[above=2pt, scale=0.9] at (v2) {$x$};
\node[below=2pt, scale=0.9] at (w2) {$x$};
\node[above=2pt, scale=0.9] at (v3) {$y$};
\node[below=2pt, scale=0.9] at (w3) {$y$};
\node[above=2pt, scale=0.9] at (u1) {$z$};
\node[above=2pt, scale=0.9] at (v4) {$h$};
\node[below=2pt, scale=0.9] at (w4) {$h$};

\foreach \i in {u1,v1,v2,v3,v4,w1,w2,w3,w4}
\filldraw (\i) circle (0.07);
\end{tikzpicture}
}
\caption{Forbidden graph $M_5$}
\label{fig:M5}
\end{figure}

\begin{figure}[htbp]
\centering
\subcaptionbox{$M_6$ \label{subfig:M6}}[.23\textwidth]{%
\begin{tikzpicture}[scale=0.9]
\coordinate (u1) at (0,0);
\coordinate (u2) at (1.2,0);
\coordinate (v1) at (-0.6,0.5);
\coordinate (v2) at (0.6,0.5);
\coordinate (v3) at (1.8,0.5);

\coordinate (w1) at (-0.6,-0.5);
\coordinate (w2) at (0.6,-0.5);
\coordinate (w3) at (1.8,-0.5);

\draw (v1) -- (w1) -- (u1) -- (v1);
\draw (u1) -- (v2) -- (v3) -- (w3) -- (w2) -- (u1); 
\draw (w2) -- (v2) -- (u2) -- (v3);
\draw (w2) -- (u2) -- (w3);
\draw (-0.8,0.5) -- (v1) -- (-0.8,0.35);
\draw (-0.8,-0.5) -- (w1) -- (-0.8,-0.35);
\draw (v3) -- (2,0.5);
\draw (w3) -- (2,-0.5);

\node[above=2pt, scale=0.9] at (v1) {$a$};
\node[below=2pt, scale=0.9] at (w1) {$a$};
\node[above=2pt, scale=0.9] at (u1) {$b$};
\node[above=2pt, scale=0.9] at (v2) {$c$};
\node[below=2pt, scale=0.9] at (w2) {$c$};
\node[above, scale=0.9] at (u2) {$d$};
\node[above=2pt, scale=0.9] at (v3) {$y$};
\node[below=2pt, scale=0.9] at (w3) {$y$};

\foreach \i in {u1,u2,v1,v2,v3,w1,w2,w3}
\filldraw (\i) circle (0.07);
\end{tikzpicture}
}
\subcaptionbox{$\widetilde{M}_6$ \label{subfig:tilde-M6}}[.23\textwidth]{%
\begin{tikzpicture}[scale=0.9]
\coordinate (u1) at (0,0);
\coordinate (u2) at (1.2,0);
\coordinate (v1) at (0.6,0.5);
\coordinate (v2) at (1.8,0.5);
\coordinate (v3) at (2.8,0.5);

\coordinate (w1) at (0.6,-0.5);
\coordinate (w2) at (1.8,-0.5);
\coordinate (w3) at (2.8,-0.5);

\draw (0.4,-0.5) -- (w1) -- (u1) -- (v1) -- (u2) -- (w1) -- (v1) -- (0.4,0.5);
\draw (u2) -- (v2) -- (v3) -- (w3) -- (w2) -- (u2);
\draw (v2) -- (w2) -- (v3) -- (3,0.5);
\draw (v2) -- (w3) -- (3,-0.5);
\draw (-0.15,0.15) -- (u1) -- (-0.15,-0.15);

\node[above=2pt, scale=0.9] at (u1) {$\alpha_1$};
\node[above=2pt, scale=0.9] at (v1) {$\alpha_2$};
\node[below=2pt, scale=0.9] at (w1) {$\alpha_2$};
\node[above=2pt, scale=0.9] at (u2) {$\alpha_3$};
\node[above=2pt, scale=0.9] at (v2) {$\alpha_4$};
\node[below=2pt, scale=0.9] at (w2) {$\alpha_4$};
\node[above=2pt, scale=0.9] at (v3) {$\alpha_5$};
\node[below=2pt, scale=0.9] at (w3) {$\alpha_5$};

\foreach \i in {u1,u2,v1,v2,v3,w1,w2,w3}
\filldraw (\i) circle (0.07);
\end{tikzpicture}
}
\subcaptionbox{$M_7$ \label{subfig:M7}}[.23\textwidth]{%
\begin{tikzpicture}[scale=0.9]
\coordinate (u1) at (0.6,0);
\coordinate (u2) at (1.8,0);
\coordinate (v1) at (0,0.5);
\coordinate (v2) at (1.2,0.5);
\coordinate (v3) at (2.4,0.5);

\coordinate (w1) at (0,-0.5);
\coordinate (w2) at (1.2,-0.5);
\coordinate (w3) at (2.4,-0.5);

\draw (v1) -- (w1) -- (u1) -- (v1);
\draw (u1) -- (v2) -- (w2) -- (u1);
\draw (v2) -- (v3) -- (u2) -- (w2) -- (w3) -- (u2) -- (v2);
\draw (-0.2,0.5) -- (v1) -- (-0.2,0.35);
\draw (-0.2,-0.5) -- (w1) -- (-0.2,-0.35);
\draw (2.6,0.5) -- (v3) -- (2.6,0.35);
\draw (2.6,-0.5) -- (w3) -- (2.6,-0.35);

\node[above=2pt, scale=0.9] at (v1) {$a$};
\node[below=2pt, scale=0.9] at (w1) {$a$};
\node[above=2pt, scale=0.9] at (u1) {$b$};
\node[above=2pt, scale=0.9] at (v2) {$c$};
\node[below=2pt, scale=0.9] at (w2) {$c$};
\node[above=1pt, scale=0.9] at (u2) {$d$};
\node[above=2pt, scale=0.9] at (v3) {$h$};
\node[below=2pt, scale=0.9] at (w3) {$h$};

\foreach \i in {u1,u2,v1,v2,v3,w1,w2,w3}
\filldraw (\i) circle (0.07);
\end{tikzpicture}
}
\subcaptionbox{$\widetilde{M}_7$ \label{subfig:tilde-M7}}[.23\textwidth]{%
\begin{tikzpicture}[scale=0.9]
\coordinate (u1) at (0.6,0);
\coordinate (u2) at (1.8,0);
\coordinate (v1) at (0,0.5);
\coordinate (v2) at (1.2,0.5);
\coordinate (v3) at (2.4,0.5);

\coordinate (w1) at (0,-0.5);
\coordinate (w2) at (1.2,-0.5);
\coordinate (w3) at (2.4,-0.5);

\draw (v1) -- (v2) -- (w2) -- (w1) -- (u1) -- (v2) -- (u2) -- (w2) -- (u1) -- (v1);
\draw (u2) -- (v3) -- (w3) -- (u2);
\draw (-0.2,0.5) -- (v1) -- (-0.2,0.35);
\draw (-0.2,-0.5) -- (w1) -- (-0.2,-0.35);
\draw (2.6,0.5) -- (v3) -- (2.6,0.35);
\draw (2.6,-0.5) -- (w3) -- (2.6,-0.35);

\node[above=2pt, scale=0.9] at (u1) {$x$};
\node[above=2pt, scale=0.9] at (v1) {$a$};
\node[below=2pt, scale=0.9] at (w1) {$a$};
\node[above=2pt, scale=0.9] at (u2) {$z$};
\node[above=2pt, scale=0.9] at (v2) {$y$};
\node[below=2pt, scale=0.9] at (w2) {$y$};
\node[above=2pt, scale=0.9] at (v3) {$h$};
\node[below=2pt, scale=0.9] at (w3) {$h$};

\foreach \i in {u1,u2,v1,v2,v3,w1,w2,w3}
\filldraw (\i) circle (0.07);
\end{tikzpicture}
}
\caption{Forbidden graphs $M_6$ and $M_7$}
\label{fig:M6-M7}
\end{figure}
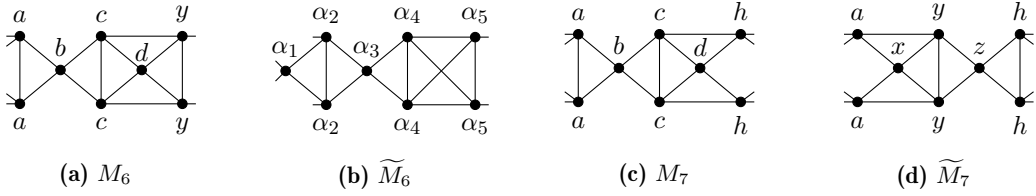

\subsection{Transferring the extremal graphs into the path-like structure}

We first prove that the first few vertices of $G$ can be reconnected to obtain one of 
graphs shown in Fig.\,\ref{fig:H1-H6}.

\begin{lemma}\label{lem:transfer-first-few-vertices-Delta-4}
Let $n\geq 10$. The induced subgraph on the first few vertices in $G$ can be transferred 
into one of $H_1$, $H_2$, $H_3$, $H_4$ and $H_5$.
\end{lemma}

\begin{proof}
We divide our proof into four steps.

{\bfseries A.} Connecting $v_1$ to $v_i$, $i=2,3,4,5$. This can be done by \autoref{coro:common-use-2}.

{\bfseries B.} Connecting $v_2$ to $v_3$. If $v_2$ is not adjacent to $v_3$, then we may 
assume that $v_2$, $v_3$ share three neighbors in $G\setminus\{v_1\}$, and all the three 
neighbors are adjacent to each other. This contradicts the fact that $G\setminus\{v_1\}$ 
is connected by \autoref{lem:G-delete-vertices-connected}.    

{\bfseries C.} Connecting $v_2$ to $v_4$. Let $u_1$, $u_2$ be the neighbors of $v_2$ other 
than $v_1$, $v_3$; and $u_3$, $u_4$, $u_5$ be the neighbors of $v_4$ other than $v_1$. 
We may assume that $v_3$ is not adjacent to $v_4$, for otherwise, at least one of $u_1$, 
$u_2$, say $u_1$, is not adjacent to $v_3$. Then $\ls(v_2,u_1,v_4,v_3)$ connects $v_2$ to $v_4$.

Now we continue the process by considering the intersection of $\{u_1,u_2\}$ and $\{u_3,u_4,u_5\}$.
If $|\{u_1,u_2\}\cap\{u_3,u_4,u_5\}|=0$, in this case we may assume $u_i\sim u_j$, 
$i=1,2$, $j=3,4,5$. Hence, there is a neighbor of $v_3$ other than $v_1,v_2$ (say $w$) 
is not adjacent to a neighbor of $v_4$ other than $v_1$ (say $u_3$). Then $\ls(v_3,w, v_4,u_3)$ 
makes $v_3$ adjacent to $v_4$, a contradiction to our assumption. 
If $|\{u_1,u_2\}\cap\{u_3,u_4,u_5\}|=1$, then there is a neighbor of $v_2$ (say $u_1$) 
is not adjacent to some neighbor of $v_4$ (say $u_5$). Then $\ls (v_2,u_1, v_4,u_5)$
connects $v_2$ to $v_4$. If $|\{u_1,u_2\}\cap\{u_3,u_4,u_5\}|=2$, the desired switching to 
$v_3\sim v_4$ is available similarly, a contradiction to our assumption. 

{\bfseries D.} Based on previous construction, we shall reconnect the subsequent few vertices 
by considering the following two cases.
\vspace{1mm}

\noindent {\bfseries Case 1.} $v_2\sim v_5$. In this case, we first connect $v_3$ to $v_4$.

(1). Connecting $v_3$ to $v_4$.
If $v_4\sim v_5$, we can choose a vertex $u\in N(v_3)\setminus\{v_1,v_2\}$ such that $u\nsim v_5$, 
and then $\ls (v_3,u,v_4,v_5)$ connects $v_3$ to $v_4$. So we assume $v_4$ is not adjacent to 
$v_5$. Furthermore, we can also assume $v_3\nsim v_5$. Indeed, if $v_3\sim v_5$, then there 
exists a vertex $v\in N(v_4)\setminus\{v_1,v_2\}$ such that $v$ is not adjacent to $v_5$. Then 
$\ls (v_3,v_5,v_4,v)$ makes $v_3\sim v_4$. Therefore, we obtain one of graphs in 
Fig.\,\ref{fig:graphs-in-case1}\,\subref{subfig:a1} -- \subref{subfig:c1}. For the one 
in Fig.\,\ref{fig:graphs-in-case1}\,\subref{subfig:a1}, $v_5$ has a neighbor $s$ satisfying 
$s\nsim u$. Then $\ls (u,v_4,s,v_5)$ and $\ls (u,v_3,v_5,v_4)$ make $v_3\sim v_4$. For the 
other two cases, we can similarly connect $v_3$ to $v_4$.

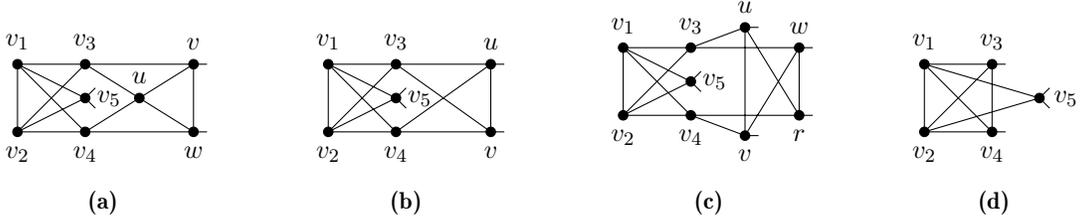
\begin{figure}[htbp]
\centering
\subcaptionbox{\label{subfig:a1}}[.25\textwidth]{%
\begin{tikzpicture}[scale=0.9]
\coordinate (u1) at (0,0.5);
\coordinate (u2) at (1,0.5);
\coordinate (u3) at (2.6,0.5);
\coordinate (v1) at (1,0);
\coordinate (v2) at (1.8,0);
\coordinate (w1) at (0,-0.5);
\coordinate (w2) at (1,-0.5);
\coordinate (w3) at (2.6,-0.5);

\draw (u2) -- (v2) -- (w2) -- (u1) -- (w1) -- (w2) -- (w3) --%
(u3) -- (u2) -- (u1) -- (v1) -- (w1) -- (u2);
\draw (u3) -- (v2) -- (w3);
\draw (1.15,0.15) -- (v1) -- (1.15,-0.15);
\draw (2.8,0.5) -- (u3);
\draw (2.8,-0.5) -- (w3);

\node[above=2pt, scale=0.9] at (u1) {$v_1$};
\node[below=2pt, scale=0.9] at (w1) {$v_2$};
\node[above=2pt, scale=0.9] at (u2) {$v_3$};
\node[below=2pt, scale=0.9] at (w2) {$v_4$};
\node[right=1pt, scale=0.9] at (v1) {$v_5$};
\node[above=2pt, scale=0.9] at (u3) {$v$};
\node[below=2pt, scale=0.9] at (w3) {$w$};
\node[above=2pt, scale=0.9] at (v2) {$u$};

\foreach \i in {u1,u2,u3,v1,v2,w1,w2,w3}
\filldraw (\i) circle (0.07);
\end{tikzpicture}
}
\subcaptionbox{\label{subfig:b1}}[.25\textwidth]{%
\begin{tikzpicture}[scale=0.9]
\coordinate (u1) at (0,0.5);
\coordinate (u2) at (1,0.5);
\coordinate (u3) at (2.4,0.5);
\coordinate (v1) at (1,0);
\coordinate (w1) at (0,-0.5);
\coordinate (w2) at (1,-0.5);
\coordinate (w3) at (2.4,-0.5);

\draw (u3) -- (w2) -- (u1) -- (w1) -- (w2) -- (w3) --%
(u3) -- (u2) -- (u1) -- (v1) -- (w1) -- (u2) -- (w3);
\draw (1.15,0.15) -- (v1) -- (1.15,-0.15);
\draw (2.6,0.5) -- (u3);
\draw (2.6,-0.5) -- (w3);

\node[above=2pt, scale=0.9] at (u1) {$v_1$};
\node[below=2pt, scale=0.9] at (w1) {$v_2$};
\node[above=2pt, scale=0.9] at (u2) {$v_3$};
\node[below=2pt, scale=0.9] at (w2) {$v_4$};
\node[right=1pt, scale=0.9] at (v1) {$v_5$};
\node[above=2pt, scale=0.9] at (u3) {$u$};
\node[below=2pt, scale=0.9] at (w3) {$v$};

\foreach \i in {u1,u2,u3,v1,w1,w2,w3}
\filldraw (\i) circle (0.07);
\end{tikzpicture}
}
\subcaptionbox{\label{subfig:c1}}[.25\textwidth]{%
\begin{tikzpicture}[scale=0.9]
\coordinate (u1) at (0,0.5);
\coordinate (u2) at (1,0.5);
\coordinate (u3) at (1.8,0.8);
\coordinate (u4) at (2.6,0.5);
\coordinate (v1) at (1,0);
\coordinate (w1) at (0,-0.5);
\coordinate (w2) at (1,-0.5);
\coordinate (w3) at (1.8,-0.8);
\coordinate (w4) at (2.6,-0.5);

\draw (w2) -- (u1) -- (w1) -- (w2) -- (w3);
\draw (u3) -- (u2) -- (u1) -- (v1) -- (w1) -- (u2);
\draw (w2) -- (w4) -- (u3) -- (w3) -- (u4) -- (u2);
\draw (w4) -- (u4);
\draw (1.15,0.15) -- (v1) -- (1.15,-0.15);
\draw (u3) -- (2,0.8);
\draw (w3) -- (2,-0.8);
\draw (u4) -- (2.8,0.5);
\draw (w4) -- (2.8,-0.5);

\node[above=2pt, scale=0.9] at (u1) {$v_1$};
\node[below=2pt, scale=0.9] at (w1) {$v_2$};
\node[above=2pt, scale=0.9] at (u2) {$v_3$};
\node[below=2pt, scale=0.9] at (w2) {$v_4$};
\node[right=1pt, scale=0.9] at (v1) {$v_5$};
\node[above=2pt, scale=0.9] at (u3) {$u$};
\node[below=2pt, scale=0.9] at (w3) {$v$};
\node[above=2pt, scale=0.9] at (u4) {$w$};
\node[below=2pt, scale=0.9] at (w4) {$r$};

\foreach \i in {u1,u2,u3,u4,v1,w1,w2,w3,w4}
\filldraw (\i) circle (0.07);
\end{tikzpicture}
}
\subcaptionbox{\label{subfig:d1}}[.22\textwidth]{%
\begin{tikzpicture}[scale=0.9]
\coordinate (u1) at (1.7,0);
\coordinate (v1) at (0,0.5);
\coordinate (w1) at (0,-0.5);
\coordinate (v2) at (1,0.5);
\coordinate (w2) at (1,-0.5);

\draw (u1) -- (v1) -- (v2) -- (w2) -- (w1) -- (v2);
\draw (w2) -- (v1) -- (w1) -- (u1);
\draw (v2) -- (1.2,0.5);
\draw (w2) -- (1.2,-0.5);
\draw (1.85,0.15) -- (u1) -- (1.85,-0.15);

\node[above=2pt, scale=0.9] at (v1) {$v_1$};
\node[below=2pt, scale=0.9] at (w1) {$v_2$};
\node[above=2pt, scale=0.9] at (v2) {$v_3$};
\node[below=2pt, scale=0.9] at (w2) {$v_4$};
\node[right=2pt, scale=0.9] at (u1) {$v_5$};

\foreach \i in {u1,v1,v2,w1,w2}
\filldraw (\i) circle (0.07);
\end{tikzpicture}
}
\caption{Some possible situations in Case 1 in \autoref{lem:transfer-first-few-vertices-Delta-4}}
\label{fig:graphs-in-case1}
\end{figure}

Until now we have transferred the first five vertices into the one in 
Fig.\,\ref{fig:graphs-in-case1}\,\subref{subfig:d1} under the condition 
$v_2\sim v_5$. We continue to connect the remaining vertices by the following steps.

(2). If $v_3\sim v_5$, in this case, we have $v_4\nsim v_5$. From \autoref{lem:common-use}, 
we may assume $v_4\sim v_6$. Now, if $v_5\sim v_6$, we obtain $H_3$. So we assume 
$v_5\nsim v_6$. Again by \autoref{lem:common-use}, we can connect $v_5$ to $v_7$. 
Furthermore, we may assume $v_6\sim v_7$. Otherwise, there are $u\in N(v_6)\setminus\{v_4\}$
and $v\in N(v_7)\setminus\{v_5\}$ such that $u\nsim v$. Then $\ls (v_6,u, v_7,v)$ connects
$v_6$ to $v_7$. If $N(v_6)\setminus\{v_4,v_7\} \neq N(v_7)\setminus\{v_5,v_6\}$, then 
there is a vertex $w\in N(v_6)\setminus\{v_4,v_7\}$ such that $w$ is not adjacent to $v_7$. 
Then $\ls (v_5,v_7, v_6,w)$ connects $v_5$ to $v_6$, and we obtain $H_3$. If 
$N(v_6)\setminus\{v_4,v_7\} = N(v_7)\setminus\{v_5,v_6\}$, by \autoref{coro:common-use-1}
we may assume $N(v_6)\setminus\{v_4,v_7\}=\{v_8,v_9\}$. By \autoref{lem:forbidden-M1} 
and \autoref{lem:forbidden-M2} we obtain $H_1$. 

If $v_3\nsim v_5$, then we can assume $v_3\sim v_6$. We consider the following two subcases.
\par\vspace{1mm}

\noindent {\bfseries Subcase 1.1.} $v_4\sim v_6$. If $v_5\sim v_6$ we obtain the graph in 
Fig.\,\ref{fig:graphs-in-claim1-and-case2}\,\subref{subfig:a2}, which is isomorphic to $H_4$. 
Likewise, if $v_5\nsim v_6$ we obtain the graph which is isomorphic to $H_2$.
\par\vspace{1mm}

\noindent {\bfseries Subcase 1.2.} $v_4\nsim v_6$. We have the following claim.

\begin{claim}\label{claim1}
$v_4\sim v_5$
\end{claim}

\noindent {\it Proof of \autoref{claim1}.}
If $v_4\nsim v_5$, by \autoref{lem:common-use}, we assume $v_4\sim v_7$. Obviously, we may assume 
$v_5\sim v_6$ or $v_5\sim v_7$. Furthermore, we may assume $v_5\sim v_6$ and $v_5\sim v_7$.
Since $v_3\nsim v_5$, we may assume $v_6\sim v_7$. Next, we shall prove 
$N(v_6)\setminus\{v_3,v_5,v_7\} = N(v_7)\setminus\{v_4,v_5,v_6\}$, and therefore 
obtain $H_5$. Otherwise, we obtain the one in 
Fig.\,\ref{fig:graphs-in-claim1-and-case2}\,\subref{subfig:b2}. Since $\ls (v_4,v_7, v_6,v_8)$ 
is a proper local switching, by \autoref{lem:local-switching} we see $x_4=x_5=\cdots=x_8$.
Also we have $x_3=x_4$. So by eigenvalue equations for $v_6$, we have $\lambda_1(G)=4$, 
a contradiction. This completes the proof of the claim.
\par\vspace{1mm}

Finally, we can arrive at the graph
either $H_3$ or $H_1$ depending on $v_5\sim v_6$ or $v_5\nsim v_6$.

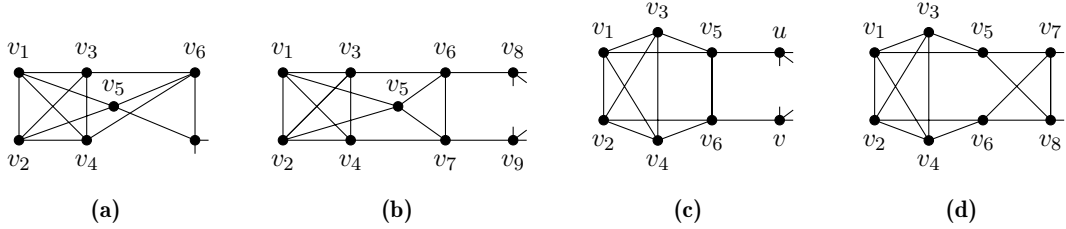
\begin{figure}[htbp]
\centering
\subcaptionbox{\label{subfig:a2}}[.23\textwidth]{%
\begin{tikzpicture}[scale=0.9]
\coordinate (u1) at (1.4,0);
\coordinate (v1) at (0,0.5);
\coordinate (v2) at (1,0.5);
\coordinate (v3) at (2.6,0.5);
\coordinate (w1) at (0,-0.5);
\coordinate (w2) at (1,-0.5);
\coordinate (w3) at (2.6,-0.5);

\draw (v1) -- (w1) -- (w2) -- (v3);
\draw (v3) -- (u1) -- (w3) -- (v3) -- (v2) -- (v1) -- (u1) -- (w1) -- (v2) -- (w2) -- (v1);
\draw (2.8,-0.5) -- (w3) -- (2.6,-0.7);

\node[above=2pt, scale=0.9] at (v1) {$v_1$};
\node[below=2pt, scale=0.9] at (w1) {$v_2$};
\node[above=2pt, scale=0.9] at (v2) {$v_3$};
\node[below=2pt, scale=0.9] at (w2) {$v_4$};
\node[above=1pt, scale=0.9] at (u1) {$v_5$};
\node[above=2pt, scale=0.9] at (v3) {$v_6$};

\foreach \i in {u1,v1,v2,v3,w1,w2,w3}
\filldraw (\i) circle (0.07);
\end{tikzpicture}
}
\subcaptionbox{\label{subfig:b2}}[.25\textwidth]{%
\begin{tikzpicture}[scale=0.9]
\coordinate (u1) at (1.7,0);
\coordinate (v1) at (0,0.5);
\coordinate (v2) at (1,0.5);
\coordinate (v3) at (2.4,0.5);
\coordinate (v4) at (3.4,0.5);
\coordinate (w1) at (0,-0.5);
\coordinate (w2) at (1,-0.5);
\coordinate (w3) at (2.4,-0.5);
\coordinate (w4) at (3.4,-0.5);

\node[above=2pt, scale=0.9] at (v1) {$v_1$};
\node[below=2pt, scale=0.9] at (w1) {$v_2$};
\node[above=2pt, scale=0.9] at (v2) {$v_3$};
\node[below=2pt, scale=0.9] at (w2) {$v_4$};
\node[above=2pt, scale=0.9] at (v3) {$v_6$};
\node[below=2pt, scale=0.9] at (w3) {$v_7$};
\node[above=1pt, scale=0.9] at (u1) {$v_5$};
\node[above=2pt, scale=0.9] at (v4) {$v_8$};
\node[below=2pt, scale=0.9] at (w4) {$v_9$};

\draw (3.4,-0.3) -- (w4) -- (w3) -- (w2) -- (w1) -- (v1) -- (u1) -- (w1) -- (v2) %
-- (w2) -- (v1) -- (v2) -- (v3) -- (v4) -- (3.4,0.3);
\draw (u1) -- (v3) -- (w3) -- (u1);
\draw (w1) -- (v2);
\draw (3.6,0.5) -- (v4) -- (3.6,0.35);
\draw (3.6,-0.5) -- (w4) -- (3.6,-0.35); 

\foreach \i in {u1,v1,v2,v3,v4,w1,w2,w3,w4}
\filldraw (\i) circle (0.07);
\end{tikzpicture}
}
\subcaptionbox{\label{subfig:c2}}[.23\textwidth]{%
\begin{tikzpicture}[scale=0.9]
\coordinate (v1) at (0,0.5);
\coordinate (v2) at (0.8,0.8);
\coordinate (v3) at (1.6,0.5);
\coordinate (v4) at (2.6,0.5);

\coordinate (w1) at (0,-0.5);
\coordinate (w2) at (0.8,-0.8);
\coordinate (w3) at (1.6,-0.5);
\coordinate (w4) at (2.6,-0.5);

\draw (v1) -- (v2) -- (v3) -- (w3) -- (w2) -- %
(w1) -- (v1) -- (v3) -- (w3) -- (w1) -- (v2) -- (w2) -- (v1);
\draw (v3) -- (v4) -- (2.8,0.35);
\draw (w3) -- (w4) -- (2.8,-0.35);
\draw (2.8,0.5) -- (v4) -- (2.6,0.3);
\draw (2.8,-0.5) -- (w4) -- (2.6,-0.3);

\node[above=2pt, scale=0.9] at (v1) {$v_1$};
\node[below=2pt, scale=0.9] at (w1) {$v_2$};
\node[above=2pt, scale=0.9] at (v2) {$v_3$};
\node[below=2pt, scale=0.9] at (w2) {$v_4$};
\node[above=2pt, scale=0.9] at (v3) {$v_5$};
\node[below=2pt, scale=0.9] at (w3) {$v_6$};
\node[above=2pt, scale=0.9] at (v4) {$u$};
\node[below=2pt, scale=0.9] at (w4) {$v$};

\foreach \i in {v1,v2,v3,v4,w1,w2,w3,w4}
\filldraw (\i) circle (0.07);
\end{tikzpicture}
}
\subcaptionbox{\label{subfig:d2}}[.22\textwidth]{%
\begin{tikzpicture}[scale=0.9]
\coordinate (v1) at (0,0.5);
\coordinate (v2) at (0.8,0.8);
\coordinate (v3) at (1.6,0.5);
\coordinate (v4) at (2.6,0.5);

\coordinate (w1) at (0,-0.5);
\coordinate (w2) at (0.8,-0.8);
\coordinate (w3) at (1.6,-0.5);
\coordinate (w4) at (2.6,-0.5);

\draw (v1) -- (v2) -- (v3) -- (v4) -- (w4) -- (w3) -- (w2) -- (w1) -- %
(v1) -- (v3) -- (w4) -- (v4) -- (w3) -- (w1) -- (v2) -- (w2) -- (v1);
\draw (2.8,0.5) -- (v4);
\draw (2.8,-0.5) -- (w4);

\node[above=2pt, scale=0.9] at (v1) {$v_1$};
\node[below=2pt, scale=0.9] at (w1) {$v_2$};
\node[above=2pt, scale=0.9] at (v2) {$v_3$};
\node[below=2pt, scale=0.9] at (w2) {$v_4$};
\node[above=2pt, scale=0.9] at (v3) {$v_5$};
\node[below=2pt, scale=0.9] at (w3) {$v_6$};
\node[above=2pt, scale=0.9] at (v4) {$v_7$};
\node[below=2pt, scale=0.9] at (w4) {$v_8$};

\foreach \i in {v1,v2,v3,v4,w1,w2,w3,w4}
\filldraw (\i) circle (0.07);
\end{tikzpicture}
}
\caption{Some auxiliary graphs in the proof of \autoref{lem:transfer-first-few-vertices-Delta-4}}
\label{fig:graphs-in-claim1-and-case2}
\end{figure}

\vspace{1mm}
\noindent {\bfseries Case 2.} $v_2\nsim v_5$. In this case, we may assume $v_2\sim v_6$.

(1). First, we may assume that either $v_5\sim v_3$ or $v_5\sim v_4$. Otherwise,
we have $v_5\nsim v_i$, $i=2,3,4$. Hence, there exist $u\in N(v_5)\setminus\{v_1\}$
and $v\in (N(v_2)\cup N(v_3)\cup N(v_4))\setminus\{v_1,v_2,v_3,v_4\}$ such that $u\nsim v$.
Then the desired switch to $v_5\sim v_i$ for some $i\in\{2,3,4\}$ exists. Likewise, 
we can also assume that $v_3\sim v_4$ or $v_3\sim v_5$. 

(2). Connecting $v_3$ to $v_5$. If $v_3\nsim v_5$, then $v_4\sim v_5$ and $v_3\sim v_4$.
Applying $\ls (v_2,v_6, v_5,v_4)$ to $G$, we connect $v_2$ to $v_5$, which contradicts 
to our assumption.

(3). Connecting $v_3$ to $v_4$.  If $v_3\nsim v_4$, we may assume $v_3\sim v_6$. 
Otherwise, $\ls (v_2,v_6, v_5,v_3)$ connects $v_2$ to $v_5$. By considering the 
neighbors of $v_3$ and $v_4$, the desired switch for $v_3\sim v_4$ is available.

\begin{claim}\label{claim2}
$v_4\nsim v_5$.
\end{claim} 

\noindent {\it Proof of \autoref{claim2}.}
Suppose by contradiction that $v_4\sim v_5$. Using eigenvalue equations we immediately 
have $x_1=x_2=x_3=x_4$. Since $\ls (v_2,v_6, v_5,v_3)$ is a proper local switching, we 
have $x_5=x_6$ by \autoref{lem:local-switching}. We deduce that the components of the 
other three neighbors of $v_6$ other that $v_2$ must be $x_1$. Hence, $x_1=x_2=\cdots=x_6$, 
a contradiction to the fact that $\lambda_1(G) < 4$. This completes the proof of this claim.
\vspace{2mm}

\begin{claim}\label{claim3}
$v_4\sim v_6$.
\end{claim} 

\noindent {\it Proof of \autoref{claim3}.}
Suppose by contradiction that $v_4\nsim v_6$. Using eigenvalue equations for $v_2$ and $v_3$ 
we immediately have $x_2=x_3$ and $x_5=x_6$. Together with the eigenvalue equation for $v_1$
gives $x_1=x_2$. If $v_5\nsim v_6$, we may assume that $v_5$, $v_6$ share two common neighbors,
say $u$, $v$. Letting $w$ be the fourth neighbor of $v_6\notin\{u,v,v_2\}$ and using the 
eigenvalue equations for $v_5$ and $v_6$, we deduce that $0=\lambda_1(G) (x_5-x_6)=x_1-x_w$.
Thus, $w=v_4$, for otherwise, $\lambda_1(G)=4$. Using the same arguments, we can deduce a contradiction
for the case $v_5\sim v_6$. This completes the proof of this claim.
\par\vspace{2mm}

(4). If $v_5\sim v_6$, 
then we obtain $H_4$ when $v_5$ and $v_6$ share a common neighbor. If $N(v_5)\cap N(v_6)=\emptyset$, 
we obtain the one in Fig.\,\ref{fig:graphs-in-claim1-and-case2}\,\subref{subfig:c2}. 
Since $\ls (v_2,v_6, v_5,u)$ is a proper local switching, we have $x_2=x_7$. By eigenvalue 
equation for $v_1$ we see $x_1=x_2$, and therefore $x_1=\cdots=x_7$. Then $\lambda_1(G)=4$, 
a contradiction. If $v_5\nsim v_6$, similarly to the proof as before, we obtain the graph 
in Fig.\,\ref{fig:graphs-in-claim1-and-case2}\,\subref{subfig:d2}. By \autoref{lem:forbidden-M3},
we obtain $H_2$. This completes the proof of the lemma.
\end{proof}

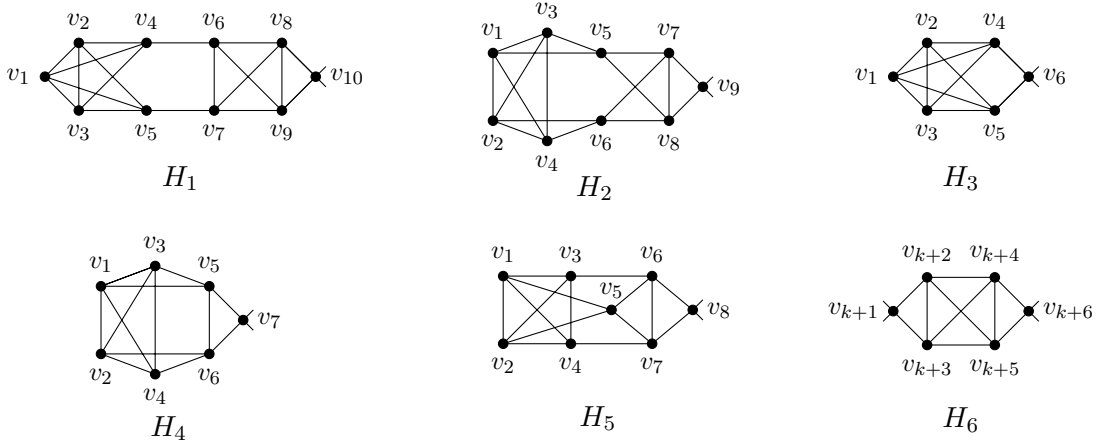
\begin{figure}[htbp]
\begin{minipage}{.39\textwidth}
\centering
\begin{tikzpicture}[scale=0.9]
\coordinate (u1) at (0,0);
\coordinate (u2) at (4,0);
\coordinate (v1) at (0.5,0.5);
\coordinate (v2) at (1.5,0.5);
\coordinate (v3) at (2.5,0.5);
\coordinate (v4) at (3.5,0.5);
\coordinate (w1) at (0.5,-0.5);
\coordinate (w2) at (1.5,-0.5);
\coordinate (w3) at (2.5,-0.5);
\coordinate (w4) at (3.5,-0.5);

\draw (u1) -- (v1) -- (v2) -- (v3) -- (v4) -- (u2) -- (w4) -- %
(w3) -- (w2) -- (w1) -- (u1) -- (v2) -- (w1) -- (v1) -- (w2) -- (u1);
\draw (v3) -- (w4) -- (v4) -- (w3) -- (v3);
\draw (v4) -- (u2) -- (w4);
\draw (4.15,0.15) -- (u2) -- (4.15,-0.15);

\node[left=2pt, scale=0.9] at (u1) {$v_1$};
\node[right=2pt, scale=0.9] at (u2) {$v_{10}$};
\node[above=2pt, scale=0.9] at (v1) {$v_2$};
\node[below=2pt, scale=0.9] at (w1) {$v_3$};
\node[above=2pt, scale=0.9] at (v2) {$v_4$};
\node[below=2pt, scale=0.9] at (w2) {$v_5$};
\node[above=2pt, scale=0.9] at (v3) {$v_6$};
\node[below=2pt, scale=0.9] at (w3) {$v_7$};
\node[above=2pt, scale=0.9] at (v4) {$v_8$};
\node[below=2pt, scale=0.9] at (w4) {$v_9$};

\foreach \i in {u1,u2,v1,v2,v3,v4,w1,w2,w3,w4}
\filldraw (\i) circle (0.07);
\node at (2,-1.5) {$H_1$};
\end{tikzpicture}
\end{minipage}
\begin{minipage}{.32\textwidth}
\centering
\begin{tikzpicture}[scale=0.9]
\coordinate (u1) at (3.1,0);
\coordinate (v1) at (0,0.5);
\coordinate (v2) at (0.8,0.8);
\coordinate (v3) at (1.6,0.5);
\coordinate (v4) at (2.6,0.5);
\coordinate (w1) at (0,-0.5);
\coordinate (w2) at (0.8,-0.8);
\coordinate (w3) at (1.6,-0.5);
\coordinate (w4) at (2.6,-0.5);

\draw (v3) -- (v1) -- (v2) -- (v3) -- (v4) -- (u1) -- (w4) -- (w3) --% 
(w2) -- (w1) -- (v1) -- (w2) -- (v2) -- (w1);
\draw (v3) -- (w4) -- (v4) -- (w3) -- (w1); 
\draw (3.25,0.15) -- (u1) -- (3.25,-0.15);

\node[above=2pt, scale=0.9] at (v1) {$v_1$};
\node[below=2pt, scale=0.9] at (w1) {$v_2$};
\node[above=2pt, scale=0.9] at (v2) {$v_3$};
\node[below=2pt, scale=0.9] at (w2) {$v_4$};
\node[above=2pt, scale=0.9] at (v3) {$v_5$};
\node[below=2pt, scale=0.9] at (w3) {$v_6$};
\node[above=2pt, scale=0.9] at (v4) {$v_7$};
\node[below=2pt, scale=0.9] at (w4) {$v_8$};
\node[right=2pt, scale=0.9] at (u1) {$v_9$}; 

\foreach \i in {u1,v1,v2,v3,v4,w1,w2,w3,w4}
\filldraw (\i) circle (0.07);
\node at (1.5,-1.5) {$H_2$};
\end{tikzpicture}
\end{minipage}
\begin{minipage}{.26\textwidth}
\centering
\begin{tikzpicture}[scale=0.9]
\coordinate (u1) at (0,0);
\coordinate (u2) at (2,0);
\coordinate (v1) at (0.5,0.5);
\coordinate (v2) at (1.5,0.5);
\coordinate (w1) at (0.5,-0.5);
\coordinate (w2) at (1.5,-0.5);

\draw (u1) -- (v1) -- (v2) -- (u2) -- (w2) -- (w1) -- (u1) --% 
(v2) -- (w1) -- (v1) -- (w2) -- (u1);
\draw (v2) -- (u2) -- (w2); 
\draw (2.15,0.15) -- (u2) -- (2.15,-0.15);

\node[left=2pt, scale=0.9] at (u1) {$v_1$};
\node[above=2pt, scale=0.9] at (v1) {$v_2$};
\node[below=2pt, scale=0.9] at (w1) {$v_3$};
\node[above=2pt, scale=0.9] at (v2) {$v_4$};
\node[below=2pt, scale=0.9] at (w2) {$v_5$};
\node[right=2pt, scale=0.9] at (u2) {$v_6$};

\foreach \i in {u1,u2,v1,v2,w1,w2}
\filldraw (\i) circle (0.07);
\node at (1,-1.5) {$H_3$};
\end{tikzpicture}
\end{minipage}
\vspace{2mm}

\begin{minipage}{.39\textwidth}
\centering
\begin{tikzpicture}[scale=0.9]
\coordinate (u1) at (2.1,0);
\coordinate (v1) at (0,0.5);
\coordinate (v2) at (0.8,0.8);
\coordinate (v3) at (1.6,0.5);
\coordinate (w1) at (0,-0.5);
\coordinate (w2) at (0.8,-0.8);
\coordinate (w3) at (1.6,-0.5);

\draw (v1) -- (v2) -- (v3) -- (u1) -- (w3) -- (w2) -- (w1) -- (v1) -- %
(v2) -- (w2) -- (v1) -- (v3) -- (w3) -- (w1) -- (v2);
\draw (2.25,0.15) -- (u1) -- (2.25,-0.15);

\node[above=2pt, scale=0.9] at (v1) {$v_1$};
\node[below=2pt, scale=0.9] at (w1) {$v_2$};
\node[above=2pt, scale=0.9] at (v2) {$v_3$};
\node[below=2pt, scale=0.9] at (w2) {$v_4$};
\node[above=2pt, scale=0.9] at (v3) {$v_5$};
\node[below=2pt, scale=0.9] at (w3) {$v_6$};
\node[right=2pt, scale=0.9] at (u1) {$v_7$};

\foreach \i in {u1,v1,v2,v3,w1,w2,w3}
\filldraw (\i) circle (0.07);
\node at (1,-1.6) {$H_4$};
\end{tikzpicture}
\end{minipage}
\begin{minipage}{.32\textwidth}
\centering
\begin{tikzpicture}[scale=0.9]
\coordinate (u1) at (1.6,0);
\coordinate (u2) at (2.8,0);
\coordinate (v1) at (0,0.5);
\coordinate (v2) at (1,0.5);
\coordinate (v3) at (2.2,0.5);
\coordinate (w1) at (0,-0.5);
\coordinate (w2) at (1,-0.5);
\coordinate (w3) at (2.2,-0.5);

\draw (v1) -- (v2) -- (v3) -- (u2) -- (w3) -- (w2) --% 
(w1) -- (v1) -- (w2) -- (v2) -- (w1) -- (u1) -- (v1);
\draw (u1) -- (v3) -- (w3) -- (u1);
\draw (2.95,0.15) -- (u2) -- (2.95,-0.15);

\node[above=2pt, scale=0.9] at (v1) {$v_1$};
\node[below=2pt, scale=0.9] at (w1) {$v_2$};
\node[above=2pt, scale=0.9] at (v2) {$v_3$};
\node[below=2pt, scale=0.9] at (w2) {$v_4$};
\node[above=1pt, scale=0.9] at (u1) {$v_5$};
\node[above=2pt, scale=0.9] at (v3) {$v_6$};
\node[below=2pt, scale=0.9] at (w3) {$v_7$};
\node[right=2pt, scale=0.9] at (u2) {$v_8$};

\foreach \i in {u1,u2,v1,v2,v3,w1,w2,w3}
\filldraw (\i) circle (0.07);
\node at (1.4,-1.6) {$H_5$};
\end{tikzpicture}
\end{minipage}
\begin{minipage}{.26\textwidth}
\centering
\begin{tikzpicture}[scale=0.9]
\coordinate (u1) at (0,0);
\coordinate (u2) at (2,0);
\coordinate (v1) at (0.5,0.5);
\coordinate (w1) at (0.5,-0.5);
\coordinate (v2) at (1.5,0.5);
\coordinate (w2) at (1.5,-0.5);

\draw (u1) -- (v1) -- (v2) -- (u2) -- (w2) -- (w1) -- (v2) -- (w2) -- (v1) -- (w1) -- (u1);
\draw (-0.15,0.15) -- (u1) -- (-0.15,-0.15);
\draw (2.15,0.15) -- (u2) -- (2.15,-0.15);

\node[left=2pt, scale=0.9] at (u1) {$v_{k+1}$};
\node[above=2pt, scale=0.9] at (v1) {$v_{k+2}$};
\node[below=2pt, scale=0.9] at (w1) {$v_{k+3}$};
\node[above=2pt, scale=0.9] at (v2) {$v_{k+4}$};
\node[below=2pt, scale=0.9] at (w2) {$v_{k+5}$};
\node[right=2pt, scale=0.9] at (u2) {$v_{k+6}$};

\foreach \i in {u1,u2,v1,v2,w1,w2}
\filldraw (\i) circle (0.07);
\node at (1,-1.6) {$H_6$};
\end{tikzpicture}  
\end{minipage}
\caption{Graphs $H_1$ -- $H_6$}
\label{fig:H1-H6}
\end{figure}

Below we continue reconnecting $G$ by a sequence of proper local switchings to 
construct structure on the remaining vertices.

\begin{lemma}\label{lem:induced-middle-vertices-Delta=4}
Let $H$ be an induced subgraph of $G$ on vertices $v_{k+1}$,$\ldots$,$v_{k+6}$, where 
the first vertex $v_{k+1}$ has degree two in $H$, and $N_G(v_i)\cap\mathcal{M}_k=\emptyset$, 
$i\geq k+2$. Then $H$ can be transferred into $H_6$ as an induced subgraph of $G$, 
not decreasing the spectral radius of $G$.
\end{lemma}

\begin{proof}
By \autoref{coro:common-use-1}, we may assume $v_{k+1}\sim v_{k+2}$ and $v_{k+1}\sim v_{k+3}$.
Also, we may assume $v_{k+2}\sim v_{k+3}$, otherwise the desired switch for $v_{k+2}\sim v_{k+3}$ 
is available. Again by \autoref{coro:common-use-1}, we have $v_{k+2}\sim v_{k+4}$ and 
$v_{k+2}\sim v_{k+5}$. Below we will make $v_{k+3}\sim v_{k+4}$ by consider the following two cases.
\par\vspace{1mm}

\noindent {\bfseries Case 1.} $v_{k+3}\sim v_{k+5}$. In this case, if $v_{k+3}\nsim v_{k+4}$, 
then $v_{k+4}$ has a neighbor $u$ which is not adjacent to $v_{k+5}$. Then 
$\ls (v_{k+3},v_{k+5}, v_{k+4},u)$ makes $v_{k+3}$ and $v_{k+4}$ adjacent.

\noindent {\bfseries Case 2.} $v_{k+3}\nsim v_{k+5}$. We can assume $v_{k+4}\sim v_{k+5}$. Otherwise, 
$v_{k+4}$ and $v_{k+5}$ must share three neighbors all of which are adjacent to each other, which 
is a contradiction to the fact that $G\setminus\mathcal{M}_{k+2}$ is connected. We may assume that 
$v_{k+3}$ and $v_{k+5}$ share two common neighbors $v$ and $w$, otherwise we can connect $v_{k+3}$ 
to $v_{k+4}$ by a proper local switching. Furthermore, we may assume that $v\sim v_{k+4}$, 
$w\sim v_{k+4}$ and $v\sim w$ due to our assumption $v_{k+3}\nsim v_{k+5}$. This is a contradiction 
to \autoref{prop:preserving-connected}.
\par\vspace{1mm}

Next, we will connect $v_{k+3}$ to $v_{k+5}$. We may assume $v_{k+3}\sim v_{k+5}$ or $v_{k+4}\sim v_{k+5}$
using similar arguments as before. If $v_{k+3}\nsim v_{k+5}$, then we have $v_{k+4}\sim v_{k+5}$, and
$v_{k+3}\sim v_{k+6}$ from \autoref{lem:common-use}. If $v_{k+5}\sim v_{k+6}$, we arrive at the one as 
shown in Fig.\,\ref{fig:graph-in-Lem-47}\,\subref{subfig:graph1-case2-lem-47}, which is impossible 
by \autoref{lem:forbidden-M6}. Likewise, if $v_{k+5}\nsim v_{k+6}$, we arrive at the graph in 
Fig.\,\ref{fig:graph-in-Lem-47}\,\subref{subfig:graph2-case2-lem-47}, which is a contradiction 
to \autoref{lem:forbidden-M7}.

Finally, we have $v_{k+4}\sim v_{k+5}$. Otherwise, $G$ contains $M_4$ as an induced subgraph,
contradicting to \autoref{lem:forbidden-M4}. Furthermore, we conclude that $v_{k+4}$ and $v_{k+5}$ 
share a common neighbor other than $v_{k+2}$, $v_{k+3}$ by \autoref{lem:forbidden-M5}.
That is exactly the graph $H_6$.
\end{proof}

\begin{figure}[htbp]
\centering
\subcaptionbox{\label{subfig:graph1-case2-lem-47}}[.4\textwidth]{%
\begin{tikzpicture}[scale=0.9]
\coordinate (u1) at (0,0);
\coordinate (u2) at (1.5,0);
\coordinate (v1) at (0.6,0.55);
\coordinate (w1) at (0.6,-0.55);
\coordinate (v2) at (2.4,0.55);
\coordinate (w2) at (2.4,-0.55);

\draw (u1) -- (v1) -- (v2) -- (w2) -- (w1) -- (u1);
\draw (v1) -- (w1) -- (u2) -- (v2) -- (2.6,0.55);
\draw (v1) -- (u2) -- (w2) -- (2.6,-0.55);
\draw (-0.15,0.15) -- (u1) -- (-0.15,-0.15);

\node[left=2pt, scale=0.9] at (u1) {$v_{k+1}$};
\node[above=2pt, scale=0.9] at (v1) {$v_{k+2}$};
\node[below=2pt, scale=0.9] at (w1) {$v_{k+3}$};
\node[above=2.5pt, scale=0.8] at (u2) {$v_{k+4}$};
\node[above=2pt, scale=0.9] at (v2) {$v_{k+5}$};
\node[below=2pt, scale=0.9] at (w2) {$v_{k+6}$};

\foreach \i in {u1,u2,v1,v2,w1,w2}
\filldraw (\i) circle (0.07);
\end{tikzpicture}
}
\subcaptionbox{\label{subfig:graph2-case2-lem-47}}[.4\textwidth]{%
\begin{tikzpicture}[scale=0.9]
\coordinate (u1) at (0,0);
\coordinate (u2) at (1.3,0);
\coordinate (v1) at (0.6,0.55);
\coordinate (w1) at (0.6,-0.55);
\coordinate (v2) at (2,0.55);
\coordinate (w2) at (2,-0.55);
\coordinate (v3) at (3,0.55);
\coordinate (w3) at (3,-0.55);

\draw (u1) -- (v1) -- (v2) -- (v3) -- (w3) -- (w2) -- (w1) -- (u1);
\draw (v1) -- (w1) -- (u2) -- (v1);
\draw (v2) -- (u2) -- (w2) -- (v3);
\draw (v2) -- (w3);
\draw (-0.15,0.15) -- (u1) -- (-0.15,-0.15);
\draw (3.2,0.55) -- (v3);
\draw (3.2,-0.55) -- (w3);

\node[left=2pt, scale=0.9] at (u1) {$v_{k+1}$};
\node[above=2pt, scale=0.9] at (v1) {$v_{k+2}$};
\node[below=2pt, scale=0.9] at (w1) {$v_{k+3}$};
\node[right=2pt, scale=0.9] at (u2) {$v_{k+4}$};
\node[above=2pt, scale=0.9] at (v2) {$v_{k+5}$};
\node[below=2pt, scale=0.9] at (w2) {$v_{k+6}$};

\foreach \i in {u1,u2,v1,v2,v3,w1,w2,w3}
\filldraw (\i) circle (0.07);
\end{tikzpicture}
}
\caption{Two possible situations in \autoref{lem:induced-middle-vertices-Delta=4}}
\label{fig:graph-in-Lem-47}
\end{figure}
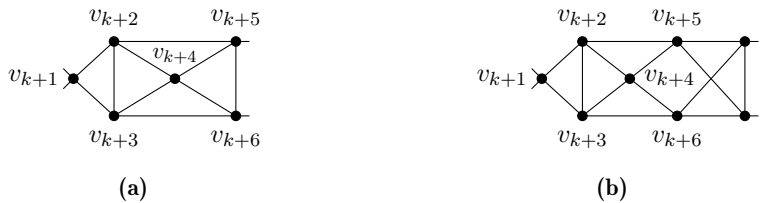

\begin{figure}[htbp]
\centering
\subcaptionbox{}[.2\textwidth]{%
\begin{tikzpicture}[scale=0.8]
\coordinate (u1) at (0,0);
\coordinate (u2) at (1,0);
\coordinate (u3) at (3,0);

\coordinate (v1) at (0.5,0.5);
\coordinate (w1) at (0.5,-0.5);
\coordinate (v2) at (1.5,0.5);
\coordinate (w2) at (1.5,-0.5);
\coordinate (v3) at (2.5,0.5);
\coordinate (w3) at (2.5,-0.5);

\draw (u1) -- (w1) -- (u2) -- (v1) -- (u1);
\draw (v1) -- (w1);
\draw (u2) -- (v2) -- (v3) -- (u3) -- (w3) -- (w2) -- (u2);
\draw (v2) -- (w2) -- (v3) -- (w3) -- (v2);
\draw (u3) -- (3.15,0.16);
\draw (u3) -- (3.15,-0.16);

\foreach \i in {u1,u2,u3,v1,v2,v3,w1,w2,w3}
\filldraw (\i) circle (0.07);
\end{tikzpicture}
}
\subcaptionbox{}[.25\textwidth]{%
\begin{tikzpicture}[scale=0.8]
\coordinate (u1) at (1.5,0);

\coordinate (v1) at (0,0.5);
\coordinate (w1) at (0,-0.5);
\coordinate (v2) at (1,0.5);
\coordinate (w2) at (1,-0.5);
\coordinate (v3) at (2,0.5);
\coordinate (w3) at (2,-0.5);
\coordinate (v4) at (3,0.5);
\coordinate (w4) at (3,-0.5);

\draw (v1) -- (w1) -- (w2) -- (v2) -- (v1) -- (w2) -- (u1) -- (v2) -- (w1);
\draw (u1) -- (v3) -- (v4) -- (w4) -- (w3) -- (u1);
\draw (w4) -- (v3) -- (w3) -- (v4);
\draw (v4) -- (3.2,0.5);
\draw (w4) -- (3.2,-0.5);

\foreach \i in {u1,v1,v2,v3,v4,w1,w2,w3,w4}
\filldraw (\i) circle (0.07);
\end{tikzpicture}
}
\subcaptionbox{}[.2\textwidth]{%
\begin{tikzpicture}[scale=0.9]
\coordinate (u1) at (0.5,0);
\coordinate (u2) at (2.5,0);

\coordinate (v1) at (0,0.5);
\coordinate (w1) at (0,-0.5);
\coordinate (v2) at (1,0.5);
\coordinate (w2) at (1,-0.5);
\coordinate (v3) at (2,0.5);
\coordinate (w3) at (2,-0.5);

\draw (v1) -- (w1) -- (u1) -- (v1);
\draw (u1) -- (v2) -- (v3) -- (u2) -- (w3) -- (w2) -- (u1);
\draw (v2) -- (w2) -- (v3) -- (w3) -- (v2);
\draw (u2) -- (2.65,0.15);
\draw (u2) -- (2.65,-0.15);

\foreach \i in {u1,u2,v1,v2,v3,w1,w2,w3}
\filldraw (\i) circle (0.07);
\end{tikzpicture}
}
\subcaptionbox{}[.25\textwidth]{%
\begin{tikzpicture}[scale=0.8]
\coordinate (u1) at (2.5,0);
\coordinate (u2) at (4.5,0);

\coordinate (v1) at (0,0.5);
\coordinate (w1) at (0,-0.5);
\coordinate (v2) at (1,0.5);
\coordinate (w2) at (1,-0.5);
\coordinate (v3) at (2,0.5);
\coordinate (w3) at (2,-0.5);
\coordinate (v4) at (3,0.5);
\coordinate (w4) at (3,-0.5);
\coordinate (v5) at (4,0.5);
\coordinate (w5) at (4,-0.5);

\draw (w1) -- (v1) -- (v2) -- (v3) -- (u1) -- (w3) -- (w2) -- (v1);
\draw (v2) -- (w2) -- (v3) -- (w3) -- (v2);
\draw (u1) -- (v4) -- (v5) -- (u2) -- (w5) -- (w4) -- (v4) -- (w5) -- (v5) -- (w4) -- (u1);
\draw (u2) -- (4.65,0.16);
\draw (u2) -- (4.65,-0.16);

\foreach \i in {u1,u2,v1,v2,v3,v4,v5,w1,w2,w3,w4,w5}
\filldraw (\i) circle (0.07);
\end{tikzpicture}
}
\caption{Some possible situations of the last few vertices}
\label{fig:last-few-vertices-S-degree4}
\end{figure}
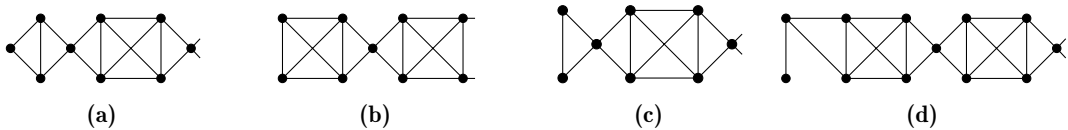

\subsection{Proof of \autoref{conj:degree-sequence} for $\Delta=4$}

The following theorem confirms the \autoref{conj:degree-sequence} for $\Delta=4$.

\begin{theorem}\label{thm:proof-conj-degree-sequ-for-degree4}
Let $G\in\mathcal{G}(n,4)$. Then $G$ has degree sequence $(4,\ldots,4,2)$.
\end{theorem}

\begin{proof}
Assume by contradiction that $|S|\in\{2,3\}$. By \autoref{lem:S-clique}, $S$ induced 
a clique in $G$. Since the local switching preserves degree sequence, we may assume 
the induced subgraphs on the last few vertices are shown as in Fig.\,\ref{fig:last-few-vertices-S-degree4}
by \autoref{lem:subset-neighbours}. The remaining proof goes along the same lines as 
\autoref{thm:proof-conj-degree-sequ-for-degree3} and is therefore omitted.
\end{proof}

\subsection{Uniqueness of the extremal graph for $\Delta=4$}

\begin{theorem}\label{thm:extremal-graph-degree-4}
Among all connected nonregular graphs with $n$ vertices and maximum degree $4$, the graph 
depicted in Fig.\,\ref{fig:maximum-spectral-radius-Deltea-4} is the unique graph attaining 
maximum spectral radius.
\end{theorem}

\begin{proof}
The proof of this theorem can be completed by the method analogous to that used in \autoref{thm:extremal-graph-degree-3}.
\end{proof}

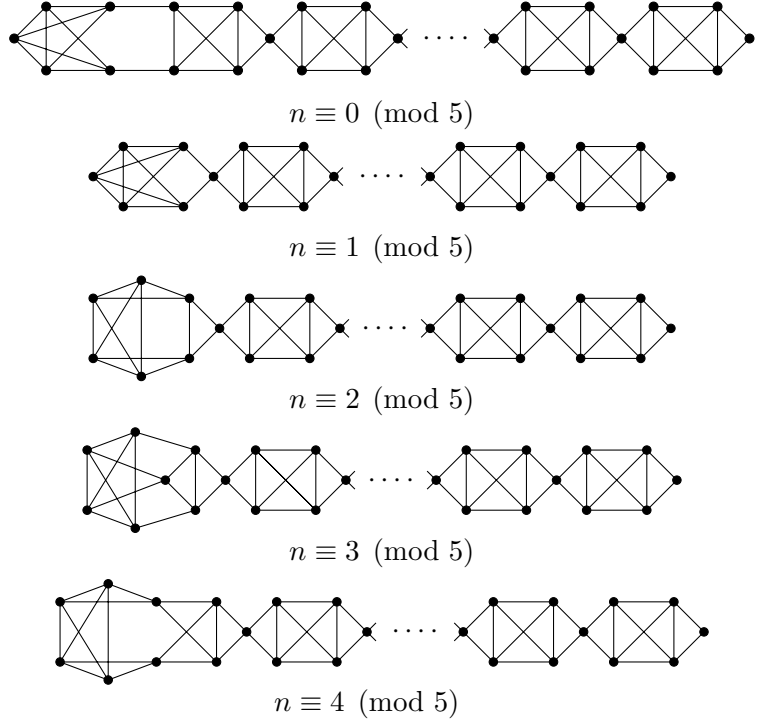
\begin{figure}[htbp]
\centering
\begin{tikzpicture}[scale=0.85]
\coordinate (u1) at (0,0);
\coordinate (u2) at (4,0);
\coordinate (u3) at (6,0);
\coordinate (u4) at (7.5,0);
\coordinate (u5) at (9.5,0);
\coordinate (u6) at (11.5,0);

\coordinate (v1) at (0.5,0.5);
\coordinate (v2) at (1.5,0.5);
\coordinate (v3) at (2.5,0.5);
\coordinate (v4) at (3.5,0.5);
\coordinate (v5) at (4.5,0.5);
\coordinate (v6) at (5.5,0.5);
\coordinate (v7) at (8,0.5);
\coordinate (v8) at (9,0.5);
\coordinate (v9) at (10,0.5);
\coordinate (v10) at (11,0.5);

\coordinate (w1) at (0.5,-0.5);
\coordinate (w2) at (1.5,-0.5);
\coordinate (w3) at (2.5,-0.5);
\coordinate (w4) at (3.5,-0.5);
\coordinate (w5) at (4.5,-0.5);
\coordinate (w6) at (5.5,-0.5);
\coordinate (w7) at (8,-0.5);
\coordinate (w8) at (9,-0.5);
\coordinate (w9) at (10,-0.5);
\coordinate (w10) at (11,-0.5);

\draw (u1) --(v1) -- (v2) -- (v3) -- (v4) -- (u2) -- (w4) -- %
(w3) -- (w2) -- (w1) -- (u1) -- (v2) -- (w1) -- (v1) -- (w2) -- (u1);
\draw (v3) -- (w3) -- (v4) -- (w4) -- (v3);
\draw (u2) -- (v5) -- (v6) -- (u3) -- (w6) -- (w5) -- (u2);
\draw (v5) -- (w6) -- (v6) -- (w5) -- (v5);
\draw (u4) -- (v7) -- (v8) -- (u5) -- (w8) -- (w7) -- (u4);
\draw (v7) -- (w8) -- (v8) -- (w7) -- (v7);
\draw (u5) -- (v9) -- (v10) -- (u6) -- (w10) -- (w9) -- (u5);
\draw (v9) -- (w9) -- (v10) -- (w10) -- (v9);
\draw (6.15,0.15) -- (u3) -- (6.15,-0.15);
\draw (7.35,0.15) -- (u4) -- (7.35,-0.15);

\foreach \i in {u1,u2,u3,u4,u5,u6,v1,v2,v3,v4,v5,v6,v7,v8,v9,v10,w1,w2,w3,w4,w5,w6,w7,w8,w9,w10}
\filldraw (\i) circle (0.07);
\node at (6.75,0) {$\cdots\cdot$};
\node at (5.75,-1.2) {$n\equiv 0\pmod{5}$};
\end{tikzpicture}

\begin{tikzpicture}[scale=0.8]
\coordinate (u1) at (0,0);
\coordinate (u2) at (2,0);
\coordinate (u3) at (4,0);
\coordinate (u4) at (5.6,0);
\coordinate (u5) at (7.6,0);
\coordinate (u6) at (9.6,0);

\coordinate (v1) at (0.5,0.5);
\coordinate (v2) at (1.5,0.5);
\coordinate (v3) at (2.5,0.5);
\coordinate (v4) at (3.5,0.5);
\coordinate (v5) at (6.1,0.5);
\coordinate (v6) at (7.1,0.5);
\coordinate (v7) at (8.1,0.5);
\coordinate (v8) at (9.1,0.5);

\coordinate (w1) at (0.5,-0.5);
\coordinate (w2) at (1.5,-0.5);
\coordinate (w3) at (2.5,-0.5);
\coordinate (w4) at (3.5,-0.5);
\coordinate (w5) at (6.1,-0.5);
\coordinate (w6) at (7.1,-0.5);
\coordinate (w7) at (8.1,-0.5);
\coordinate (w8) at (9.1,-0.5);

\draw (u1) --(v1) -- (v2) -- (u2) -- (w2) -- (w1) -- (u1) -- (v2) -- %
(w1) -- (v1) -- (w2) -- (u1);
\draw (u2) -- (v3) -- (v4) -- (u3) -- (w4) -- (w3) -- (u2);
\draw (v3) -- (w4) -- (v4) -- (w3) -- (v3);
\draw (u4) -- (v5) -- (v6) -- (u5) -- (w6) -- (w5) -- (u4);
\draw (v5) -- (w6) -- (v6) -- (w5) -- (v5);
\draw (u5) -- (v7) -- (v8) -- (u6) -- (w8) -- (w7) -- (u5); 
\draw (v7) -- (w8) -- (v8) -- (w7) -- (v7);

\draw (4.15,-0.15) -- (u3) -- (4.15,0.15);
\draw (5.45,-0.15) -- (u4) -- (5.45,0.15);
\foreach \i in {u1,u2,u3,u4,u5,u6,v1,v2,v3,v4,v5,v6,v7,v8,w1,w2,w3,w4,w5,w6,w7,w8}
\filldraw (\i) circle (0.07);
\node at (4.8,0) {$\cdots\cdot$};
\node at (4.8,-1.2) {$n\equiv 1\pmod{5}$};
\end{tikzpicture}

\begin{tikzpicture}[scale=0.8]
\coordinate (u1) at (2.1,0);
\coordinate (u2) at (4.1,0);
\coordinate (u3) at (5.6,0);
\coordinate (u4) at (7.6,0);
\coordinate (u5) at (9.6,0);

\coordinate (v1) at (0,0.5);
\coordinate (v2) at (0.8,0.8);
\coordinate (v3) at (1.6,0.5);
\coordinate (v4) at (2.6,0.5);
\coordinate (v5) at (3.6,0.5);
\coordinate (v6) at (6.1,0.5);
\coordinate (v7) at (7.1,0.5);
\coordinate (v8) at (8.1,0.5);
\coordinate (v9) at (9.1,0.5);

\coordinate (w1) at (0,-0.5);
\coordinate (w2) at (0.8,-0.8);
\coordinate (w3) at (1.6,-0.5);
\coordinate (w4) at (2.6,-0.5);
\coordinate (w5) at (3.6,-0.5);
\coordinate (w6) at (6.1,-0.5);
\coordinate (w7) at (7.1,-0.5);
\coordinate (w8) at (8.1,-0.5);
\coordinate (w9) at (9.1,-0.5);

\draw (v1) -- (v2) -- (v3) -- (u1) -- (w3) -- (w2) -- (w1) -- %
(v1) -- (v3) -- (w3) -- (w1) -- (v2) -- (w2) -- (v1);
\draw (u1) -- (v4) -- (v5) -- (u2) -- (w5) -- (w4) -- (u1);
\draw (v4) -- (w5) -- (v5) -- (w4) -- (v4);
\draw (u3) -- (v6) -- (v7) -- (u4) -- (w7) -- (w6) -- (u3);
\draw (v6) -- (w7) -- (v7) -- (w6) -- (v6);
\draw (u4) -- (v8) -- (v9) -- (u5) -- (w9) -- (w8) -- (u4);
\draw (v8) -- (w9) -- (v9) -- (w8) -- (v8);
\draw (4.25,-0.15) -- (u2) -- (4.25,0.15);
\draw (5.45,-0.15) -- (u3) -- (5.45,0.15);
\foreach \i in {u1,u2,u3,u4,u5,v1,v2,v3,v4,v5,v6,v7,v8,v9,w1,w2,w3,w4,w5,w6,w7,w8,w9}
\filldraw (\i) circle (0.07);
\node at (4.85,0) {$\cdots\cdot$};
\node at (4.8,-1.2) {$n\equiv 2\pmod{5}$};
\end{tikzpicture}

\begin{tikzpicture}[scale=0.8]
\coordinate (u1) at (1.3,0);
\coordinate (u2) at (2.3,0);
\coordinate (u3) at (4.3,0);
\coordinate (u4) at (5.8,0);
\coordinate (u5) at (7.8,0);
\coordinate (u6) at (9.8,0);

\coordinate (v1) at (0,0.5);
\coordinate (v2) at (0.8,0.8);
\coordinate (v3) at (1.8,0.5);
\coordinate (v4) at (2.8,0.5);
\coordinate (v5) at (3.8,0.5);
\coordinate (v6) at (6.3,0.5);
\coordinate (v7) at (7.3,0.5);
\coordinate (v8) at (8.3,0.5);
\coordinate (v9) at (9.3,0.5);

\coordinate (w1) at (0,-0.5);
\coordinate (w2) at (0.8,-0.8);
\coordinate (w3) at (1.8,-0.5);
\coordinate (w4) at (2.8,-0.5);
\coordinate (w5) at (3.8,-0.5);
\coordinate (w6) at (6.3,-0.5);
\coordinate (w7) at (7.3,-0.5);
\coordinate (w8) at (8.3,-0.5);
\coordinate (w9) at (9.3,-0.5);

\draw (v1) -- (v2) -- (v3) -- (u2) -- (w3) -- (w2) -- (w1) -- %
(v1) -- (u1) -- (w1) -- (v2) -- (w2) -- (v1);
\draw (v3) -- (u1) -- (w3) -- (v3);
\draw (u2) -- (v4) -- (v5) -- (u3) -- (w5) -- (w4) -- (u2);
\draw (v4) -- (w5) -- (v5) -- (w4) -- (v4) -- (w5);
\draw (u4) -- (v6) -- (v7) -- (u5) -- (w7) -- (w6) -- (u4);
\draw (v6) -- (w7) -- (v7) -- (w6) -- (v6);
\draw (u5) -- (v8) -- (v9) -- (u6) -- (w9) -- (w8) -- (u5);
\draw (v8) -- (w9) -- (v9) -- (w8) -- (v8);

\draw (4.45,-0.15) -- (u3) -- (4.45,0.15);
\draw (5.65,-0.15) -- (u4) -- (5.65,0.15);
\foreach \i in {u1,u2,u3,u4,u5,u6,v1,v2,v3,v4,v5,v6,v7,v8,v9,w1,w2,w3,w4,w5,w6,w7,w8,w9}
\filldraw (\i) circle (0.07);
\node at (5.05,0) {$\cdots\cdot$};
\node at (4.9,-1.2) {$n\equiv 3\pmod{5}$};
\end{tikzpicture}

\begin{tikzpicture}[scale=0.8]
\coordinate (u1) at (3.1,0);
\coordinate (u2) at (5.1,0);
\coordinate (u3) at (6.7,0);
\coordinate (u4) at (8.7,0);
\coordinate (u5) at (10.7,0);

\coordinate (v1) at (0,0.5);
\coordinate (v2) at (0.8,0.8);
\coordinate (v3) at (1.6,0.5);
\coordinate (v4) at (2.6,0.5);
\coordinate (v5) at (3.6,0.5);
\coordinate (v6) at (4.6,0.5);
\coordinate (v7) at (7.2,0.5);
\coordinate (v8) at (8.2,0.5);
\coordinate (v9) at (9.2,0.5);
\coordinate (v10) at (10.2,0.5);

\coordinate (w1) at (0,-0.5);
\coordinate (w2) at (0.8,-0.8);
\coordinate (w3) at (1.6,-0.5);
\coordinate (w4) at (2.6,-0.5);
\coordinate (w5) at (3.6,-0.5);
\coordinate (w6) at (4.6,-0.5);
\coordinate (w7) at (7.2,-0.5);
\coordinate (w8) at (8.2,-0.5);
\coordinate (w9) at (9.2,-0.5);
\coordinate (w10) at (10.2,-0.5);

\draw (v1) -- (v2) -- (v3) -- (v4) -- (u1) -- (w4) -- (w3) -- (w2) -- %
(w1) -- (v1) -- (w2) -- (v2) -- (w1) -- (w3) -- (v4) -- (w4) -- (v3) -- (v1);
\draw (u1) -- (v5) -- (v6) -- (u2) -- (w6) -- (w5) -- (u1);
\draw (v5) -- (w6) -- (v6) -- (w5) -- (v5);
\draw (u3) -- (v7) -- (v8) -- (u4) -- (w8) -- (w7) -- (u3);
\draw (v7) -- (w8) -- (v8) -- (w7) -- (v7);
\draw (u4) -- (v9) -- (v10) -- (u5) -- (w10) -- (w9) -- (u4);
\draw (v9) -- (w10) -- (v10) -- (w9) -- (v9); 

\draw (5.25,-0.15) -- (u2) -- (5.25,0.15);
\draw (6.55,-0.15) -- (u3) -- (6.55,0.15);
\foreach \i in {u1,u2,u3,u4,u5,v1,v2,v3,v4,v5,v6,v7,v8,v9,v10,w1,w2,w3,w4,w5,w6,w7,w8,w9,w10}
\filldraw (\i) circle (0.07);
\node at (5.9,0) {$\cdots\cdot$};
\node at (5.1,-1.2) {$n\equiv 4\pmod{5}$};
\end{tikzpicture}
\caption{The unique extremal graph in $\mathcal{G} (n,4)$ depending on the value of $n\!\!\mod{5}$}
\label{fig:maximum-spectral-radius-Deltea-4}
\end{figure}

\section{Disproof of Conjecture \ref{conj:limit}}
\label{sec:disprove-conj-1-1}

\subsection{Bounding the difference between $\Delta$ and $\lambda_1(n,\Delta)$}

To state our results, we need more definitions. The \emph{coalescence} (at the vertex $v$) of two 
disjoint graphs $G_1$ and $G_2$, denoted by $G_1\odot G_2$, is the graph attained by identifying 
a vertex $v_1\in V(G_1)$ and a vertex $v_2\in V(G_2)$, merging the two vertices into a single vertex 
$v$. A finite sequence of nonnegative integers is called (connected) \emph{graphic} if the terms 
in the sequence can be realized as the degrees of vertices of a finite (connected) simple graph.

The coming lemma, due to Liu and Li \cite{LiuLi2008}, guarantees the connectedness of graphs with 
a given degree sequence.

\begin{lemma}[\cite{LiuLi2008}]\label{lem:connected-graphic}
The sequence $(d_1, d_2,\ldots, d_n)$ with $d_1 \geq d_2 \geq\cdots\geq d_n$ and $d_{n-1}\geq2$, 
$d_n\geq 1$, is graphic if and only if it is connected graphic.
\end{lemma}

\begin{proof}
It suffices to prove the necessity. Assume that $G$ is a simple graph with degree sequence $(d_1,d_2,\ldots,d_n)$ 
satisfying the condition in \autoref{lem:connected-graphic}. Suppose, without loss of generality, 
$G$ has two components $G_1$ and $G_2$. Since $d_{n-1}\geq 2$ and $d_n\geq 1$, we may suppose that 
each vertex in $G_1$ has degree at least two. Hence, there exists an edge $u_1v_1$ in $G_1$, which 
is not a cut edge. Let $u_2v_2$ be an edge in $G_2$. Then $\ls (G; u_1,v_1, u_2,v_2)$ is connected 
with the same degree sequence as $G$. 
\end{proof}

Applying the Erd\H os-Gallai theorem \cite[p.\,11]{BondyMurty2008} to the $n$-tuple $(\Delta,\ldots,\Delta,\Delta-i)$,
we know that $(\Delta,\ldots,\Delta,\Delta-i)$ is graphic if $n\Delta-i$ is even and $n\geq\Delta+2$. 
By \autoref{lem:connected-graphic}, for $i=1,2$ we denote by $\mathcal{T}_n^{(i)}$ the set of connected 
graphs on $n$ vertices having degree sequence $(\Delta,\ldots,\Delta,\Delta-i)$.

Let $G_1,\ldots,G_{k-1}$ be $(k-1)$ disjoint copies of the complete graph $K_{\Delta}$, and 
$\{v_1^{(i)}, v_2^{(i)},\ldots,v_{\Delta}^{(i)}\}$ be the vertex set of $G_i$ for each $i\in [k-1]$.
Pick $k$ new vertices $u_1,u_2,\ldots,u_k$, and join $u_i$ to each $v_1^{(i)},\ldots,v_p^{(i)}$, 
$v_{p+1}^{(i-1)},\ldots,v_{\Delta}^{(i-1)}$ for $i\in [k]$ (define $\{v_{p+1}^{(0)},\ldots,v_{\Delta}^{(0)}\} = \emptyset$
and $v_1^{(k)},\ldots,v_p^{(k)}=\emptyset$). We denote by $G_{\Delta,k-1}^p$ the resulting graphs. 

Let $\Delta$ be a fixed integer and $n=k(\Delta+1)+\alpha$, where $\alpha$ is given by 
\[
\alpha = \begin{cases}
n\hspace{-2mm}\mod{(\Delta + 1)}, & \text{if}\ n\hspace{-2mm}\mod{(\Delta + 1)} > 0, \\
\Delta + 1, & \text{if}\ n\hspace{-2mm}\mod{(\Delta + 1)} =0.
\end{cases}
\]
We define two sets $\mathcal{H}_n^{(1)}$ and $\mathcal{H}_n^{(2)}$ as follows:
\begin{enumerate}
\item[$(1)$] Let $\Delta$ be odd. If $n$ is odd, we denote by $\mathcal{H}_n^{(1)}$ the set of connected 
graphs $G_{\Delta, k-1}^{\Delta-1}\odot F$, where $G_{\Delta, k-1}^{\Delta-1}\odot F$ is the coalescence 
of $G_{\Delta, k-1}^{\Delta-1}$ and $F$ at the vertex $u_k$, and $F$ is a graph in $\mathcal{T}_{\Delta+\alpha+1}^{(1)}$;
if $n$ is even, we denote by $\mathcal{H}_n^{(1)}$ the set of graphs obtain from $G_{\Delta, k-1}^{\Delta-1}\odot F$ 
by adding a pendant edge at $u_1$, where $F\in\mathcal{T}_{\Delta+\alpha}^{(1)}$. Fig.\,\ref{fig:example-G-524} 
shows an example of graph $G_{5,2}^4\odot F$ in $\mathcal{H}_n^{(1)}$.

\item[$(2)$] Let $\Delta$ be even. We denote by $\mathcal{H}_n^{(2)}$ the set of connected graphs  
$G_{\Delta, k-1}^{\Delta-2}\odot F$, where $F$ is a graph in $\mathcal{T}_{\Delta+\alpha+1}^{(2)}$.
\end{enumerate}

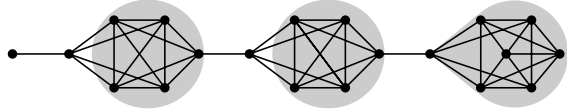
\begin{figure}[htbp]
\centering
\begin{tikzpicture}[scale=0.75, line width=.6pt]
\coordinate (u0) at (-1,0);
\coordinate (u1) at (0,0);
\coordinate (u2) at (2.3,0);
\coordinate (u3) at (3.2,0);
\coordinate (u4) at (5.5,0);
\coordinate (u5) at (6.4,0);
\coordinate (u6) at (7.75,0);
\coordinate (u7) at (8.7,0);

\coordinate (v1) at (0.8,0.6);
\coordinate (v2) at (1.7,0.6);
\coordinate (v3) at (4,0.6);
\coordinate (v4) at (4.9,0.6);
\coordinate (v5) at (7.3,0.6);
\coordinate (v6) at (8.2,0.6);

\coordinate (w1) at (0.8,-0.6);
\coordinate (w2) at (1.7,-0.6);
\coordinate (w3) at (4,-0.6);
\coordinate (w4) at (4.9,-0.6);
\coordinate (w5) at (7.3,-0.6);
\coordinate (w6) at (8.2,-0.6);

\filldraw[fill=gray!40, draw=gray!40] (1.4,0) ellipse [x radius=0.98,y radius=0.95];
\filldraw[fill=gray!40, draw=gray!40] (4.6,0) ellipse [x radius=0.98,y radius=0.95];
\fill[gray!40] (u5) -- (7.25,0.7) -- (7.25,-0.7) -- cycle;
\fill[gray!40] (7.9,0) circle (0.95);

\draw (w2) -- (u1) -- (v1) -- (v2) -- (u2) -- (w2) -- (w1) -- (u1) -- (v2) --%
(w2) -- (v1) -- (w1) -- (u2) -- (v1);
\draw (u0) -- (u1);
\draw (v2) -- (w1);
\draw (u2) -- (u3) -- (v3) -- (v4) -- (u4) -- (w4) -- (w3) -- (u4) -- (v3) -- %
(w3) -- (u3) -- (v4) -- (w4) -- (v3) -- (w4) -- (u3);
\draw (v4) -- (w3);
\draw (u4) -- (u5) -- (v5) -- (v6) -- (u7) -- (w6) -- (w5) -- (u5) -- (v6) --%
(w6) -- (u6) -- (u7) -- (v5) -- (u6) -- (w5) -- (v5);
\draw (w5) -- (u7);
\draw (v6) -- (u6);
\draw (u5) -- (w6);

\foreach \i in {u0,u1,u2,u3,u4,u5,u6,u7,v1,v2,v3,v4,v5,v6,w1,w2,w3,w4,w5,w6}
\filldraw (\i) circle (0.07);
\end{tikzpicture}
\caption{An example of $G_{5,2}^4\odot F$ in $\mathcal{H}_n^{(1)}$}
\label{fig:example-G-524}
\end{figure}

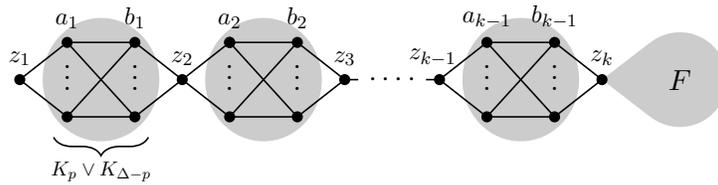
\begin{figure}[htbp]
\centering
\begin{tikzpicture}[scale=0.9, line width=.6pt]
\coordinate (u1) at (0,0);
\coordinate (u2) at (2.4,0);
\coordinate (u3) at (4.8,0);
\coordinate (u4) at (6.2,0);
\coordinate (u5) at (8.6,0);

\coordinate (v1) at (0.7,0.55);
\coordinate (w1) at (0.7,-0.55);
\coordinate (v2) at (1.7,0.55);
\coordinate (w2) at (1.7,-0.55);
\coordinate (v3) at (3.1,0.55);
\coordinate (w3) at (3.1,-0.55);
\coordinate (v4) at (4.1,0.55);
\coordinate (w4) at (4.1,-0.55);
\coordinate (v5) at (6.9,0.55);
\coordinate (w5) at (6.9,-0.55);
\coordinate (v6) at (7.9,0.55);
\coordinate (w6) at (7.9,-0.55);

\filldraw[fill=gray!40, draw=gray!40] (1.2,0) ellipse [x radius=0.85,y radius=0.9];
\filldraw[fill=gray!40, draw=gray!40] (3.6,0) ellipse [x radius=0.85,y radius=0.9];
\filldraw[fill=gray!40, draw=gray!40] (7.4,0) ellipse [x radius=0.85,y radius=0.9];
\foreach \i in {0.7,1.7,3.1,4.1,6.9,7.9}
\node at (\i,0.15) {$\vdots$};

\fill[gray!40] (8.6,0) -- (9.3,0.54) -- (9.3,-0.54) -- cycle;
\fill[gray!40] (9.75,0) circle (0.7);

\node[above=2pt, scale=0.85] at (v1) {$a_1$};
\node[above=2pt, scale=0.85] at (v2) {$b_1$};
\node[above=2pt, scale=0.85] at (v3) {$a_2$};
\node[above=2pt, scale=0.85] at (v4) {$b_2$};
\node[above=2pt, scale=0.85] at (v5) {$a_{k-1}$};
\node[above=2pt, scale=0.85] at (v6) {$b_{k-1}$};
\node[above=2pt, scale=0.85] at (u1) {$z_1$};
\node[above=2pt, scale=0.85] at (u2) {$z_2$};
\node[above=2pt, scale=0.85] at (u3) {$z_3$};
\node[above=2pt, scale=0.85] at (6.1,0) {$z_{k-1}$};
\node[above=2pt, scale=0.85] at (u5) {$z_k$};

\node at (5.5,0) {$\cdots\cdot$};
\node at (9.75,0) {$F$};

\draw [braket] (1.9,-0.5) -- (0.5,-0.5) node [midway,yshift=-13pt, scale=0.7] {$K_p\vee K_{\Delta-p}$};

\draw (u1) -- (v1) -- (v2) -- (u2) -- (w2) -- (w1) -- (u1);
\draw (u2) -- (v3) -- (v4) -- (u3) -- (w4) -- (w3) -- (u2);
\draw (u4) -- (v5) -- (v6) -- (u5) -- (w6) -- (w5) -- (u4);
\draw (v1) -- (w2);
\draw (w1) -- (v2);
\draw (v3) -- (w4);
\draw (v4) -- (w3);
\draw (v5) -- (w6);
\draw (w5) -- (v6);
\draw (u3) -- (5,0);
\draw (u4) -- (6,0);

\foreach \i in {u1,u2,u3,u4,u5,v1,v2,v3,v3,v4,v5,v6,w1,w2,w3,w4,w5,w6}
\filldraw (\i) circle (0.07);
\end{tikzpicture}
\caption{$G_{\Delta,k-1}^p\odot F$}
\label{fig:label-G-odot-F}
\end{figure}

The following theorem give a negative answer to \autoref{conj:limit}.

\begin{theorem}\label{thm:limit-upper-bound}
Let $\Delta$ be a fixed positive number with $\Delta\geq 3$. If $\Delta$ is odd, then
\[
\varlimsup_{n\to\infty} \frac{n^2(\Delta-\lambda_1(n,\Delta))}{\Delta-1} \leq \frac{\pi^2}{4}.
\]
If $\Delta$ is even, then
\[
\varlimsup_{n\to\infty} \frac{n^2(\Delta-\lambda_1(n,\Delta))}{\Delta-2} \leq \frac{\pi^2}{2}.
\]
\end{theorem}

\begin{proof}
Let $n=k(\Delta +1) + \alpha$ and $\alpha$ be defined as before. Let $G_{\Delta,k-1}^p \odot F$ 
be a graph in $\mathcal{H}_n^{(1)}$ or $\mathcal{H}_n^{(2)}$ depending on the parity of 
$\Delta$, where $p\in\{\Delta-1, \Delta-2\}$ and 
$F\in(\mathcal{T}_{\Delta+\alpha}^{(1)} \cup \mathcal{T}_{\Delta+\alpha+1}^{(1)} \cup \mathcal{T}_{\Delta+\alpha+1}^{(2)})$.
Obviously, $G_{\Delta,k-1}^p \odot F\in\mathcal{G} (n,\Delta)$. Since the largest eigenvalue of a 
graph is monotone with respect to vertex addition, it suffices to prove the assertion for the graph 
$G_{\Delta,k-1}^p \odot F$\,\footnote{If $n\!\!\mod{(\Delta+1)}=0$, it is enough to consider the graph 
$G_{\Delta,k-2}^p\odot F$. This is not essential and does not affect the results at all.} as shown 
in Fig.\,\ref{fig:label-G-odot-F}.

To finish the proof, we define a vector $\bm{z}$ for $G_{\Delta,k-1}^p \odot F$ such that its 
components on $G_{\Delta,k-1}^p$ are as given in Fig.\,\ref{fig:label-G-odot-F}, where
\begin{equation}\label{eq:xj}
z_j = \sin \frac{(2j-1)\pi}{4k},~ j\in [k],
\end{equation}
and for each $j\in [k-1]$, 
\begin{equation}\label{eq:aj-bj}
a_j = \frac{(p+1) z_j + (\Delta - p) z_{j+1}}{\Delta+1},~~
b_j = \frac{p z_j + (\Delta+1-p) z_{j+1}}{\Delta+1}.
\end{equation}
The components of $\bm{z}$ on $F$ are all equal to $z_k$. Set for short, 
$\lambda_1:=\lambda_1(G_{\Delta,k-1}^p \odot F)$. 
By \eqref{eq:Delta-rho} we conclude that
\begin{align*}
\Delta {-} \lambda_1 
& \leq \frac{(\Delta-p) z_1^2 + \bm{z}^{\mathrm{T}} L(G_{\Delta,k-1}^p) \bm{z}}{\|\bm{z}\|_2^2} \\
& = \frac{(\Delta - p) z_1^2 + \frac{p(\Delta-p)}{\Delta+1} \sum_{j=1}^{k-1} (z_j - z_{j+1})^2}%
{(\Delta + \alpha) z_k^2 + \sum_{j=1}^k z_j^2 + p\sum_{j=1}^{k-1} a_j^2 + (\Delta-p)\sum_{j=1}^{k-1} b_j^2},
\end{align*}
which simplifies via \eqref{eq:aj-bj} to 
\begin{equation}\label{eq:Delta-lambda-p}
\Delta {-} \lambda_1 \leq \frac{(\Delta-p) (\Delta+1)^2 z_1^2 + p(\Delta-p) (\Delta+1) \sum\limits_{j=1}^{k-1} (z_j - z_{j+1})^2}%
{[\Delta^3 {-} (2p {-} 3) \Delta^2 {+} (2p^2 {-} 4p {+} 3) \Delta {+} 4p^2 {+} 1] \sum\limits_{j=1}^k z_j^2 
{+} 2p(\Delta {-} p) (\Delta {+} 2) \sum\limits_{j=1}^{k-1} z_j z_{j+1}}.
\end{equation}

In what follows we shall estimate the terms in the right-hand side of \eqref{eq:Delta-lambda-p}, respectively.
Notice that the identities $\sin\theta_1 - \sin\theta_2 =2\cos\, (\theta_1+\theta_2)/2 \cdot \sin\, (\theta_1-\theta_2)/2$
and $2\cos^2\theta=1+\cos 2\theta$. In light of \eqref{eq:xj} we deduce that
\[ 
\sum_{j=1}^{k-1} (z_j - z_{j+1})^2 = 4\cdot\sin^2\frac{\pi}{4k}\cdot \sum_{j=1}^{k-1} \cos^2\frac{\pi j}{2k}
=2\cdot\sin^2\frac{\pi}{4k}\cdot\bigg(k-1+\sum_{j=1}^{k-1} \cos\frac{\pi j}{k}\bigg).
\]
In addition, using the fact $2\cos\theta=\e^{\theta\imgi} + \e^{-\theta\imgi}$, we have 
\begin{equation}\label{eq:sum-cos-1}
\begin{split}
2\sum_{j=1}^{k-1} \cos\frac{\pi j}{k} 
& = \sum_{j=1}^{k-1} \big(\e^{\pi j\imgi/k} + \e^{-\pi j\imgi/k}\big) \\
& = \frac{\e^{\pi\imgi/k} (1-\e^{\pi (k-1)\imgi/k})}{1-\e^{\pi\imgi/k}}
+\frac{\e^{-\pi\imgi/k} (1-\e^{-\pi (k-1)\imgi/k})}{1-\e^{-\pi\imgi/k}} \\
& = 0.
\end{split}
\end{equation}
Therefore, we immediately obtain that
\begin{equation}\label{eq:sum-difference-xi-square}
\sum_{j=1}^{k-1} (z_j - z_{j+1})^2 = 2(k-1)\cdot\sin^2\frac{\pi}{4k}.
\end{equation}
Next, we consider the term $\sum_{j=1}^k z_j^2$. According to \eqref{eq:xj} we conclude that
\[
\sum_{j=1}^k z_j^2 = \sum_{j=1}^k \sin^2\frac{(2j-1)\pi}{4k} 
=\frac{k}{2} - \frac{1}{2}\sum_{j=1}^k \cos\frac{(2j-1)\pi}{2k}
=\frac{k+1}{2}.
\]
Finally, we consider the term $\sum_{j=1}^{k-1} z_j z_{j+1}$. Since
$2\sin\theta_1\cdot\sin\theta_2 = \cos\, (\theta_1-\theta_2)-\cos\, (\theta_1+\theta_2)$, we obtain that
\begin{align*}
\sum_{j=1}^{k-1} z_j z_{j+1} 
& = \sum_{j=1}^{k-1} \sin\frac{(2j-1)\pi}{4k}\cdot\sin\frac{(2j+1)\pi}{4k} \\
& = \frac{1}{2} \sum_{j=1}^{k-1} \Big( \cos\frac{\pi}{2k} - \cos\frac{\pi j}{k} \Big) \\
& = \frac{k-1}{2} \cos\frac{\pi}{2k} - \frac{1}{2}\sum_{j=1}^{k-1} \cos\frac{\pi j}{k},
\end{align*}
which, together with \eqref{eq:sum-cos-1}, gives
\begin{equation}\label{eq:sum-xi-times-xi1}
\sum_{j=1}^{k-1} z_j z_{j+1} = \frac{k-1}{2}\cos\frac{\pi}{2k}.
\end{equation}

If $\Delta$ is odd, we choose $p=\Delta-1$ in \eqref{eq:Delta-lambda-p}.
Putting \eqref{eq:Delta-lambda-p} -- \eqref{eq:sum-xi-times-xi1} together, we have
\begin{align*}
\Delta - \lambda_1 
& \leq\frac{\big( 2(\Delta+1)^2 + 4(k-1)(\Delta^2-1)\big) \cdot\sin^2\frac{\pi}{4k}}{(\Delta^3+\Delta^2+\Delta+5) (k+1)%
+2(\Delta+2)(\Delta-1) (k-1)\cdot\cos\frac{\pi}{2k} }.
\end{align*}
Noting that $n=k(\Delta+1)+\alpha$, we thus obtain that
\[
\varlimsup_{n\to\infty} \frac{n^2 (\Delta-\lambda_1)}{\Delta-1} \leq\frac{\pi^2}{4}.
\]
If $\Delta$ is even, we choose $p=\Delta-2$ in \eqref{eq:Delta-lambda-p}.
Likewise, we get
\[
\varlimsup_{n\to\infty} \frac{n^2 (\Delta-\lambda_1)}{\Delta-2} \leq \frac{\pi^2}{2}.
\]
This completes the proof of the theorem.
\end{proof}

\subsection{Asymptotic behavior of $\Delta-\lambda_1(n,\Delta)$ for $\Delta\in\{3,4\}$}

\begin{lemma}\label{lem:x-min-x-max-bound}
Let $G\in\mathcal{G}(n,\Delta)$ with unit Perron vector $\bm{x}$. Then $x_{\min} = O\big(n^{-3/2}\big)$,
$x_{\max} = O(n^{-1/2})$.
\end{lemma}

\begin{proof}
Let $u$, $v$ be two vertices attaining $x_{\min}$ and $x_{\max}$, respectively. Obviously, 
$d(u) < \Delta$. By \eqref{eq:identity-sum-xi}, we find that 
\[ 
x_u < \sum_{w\in V(G)} (\Delta - d(w)) x_w 
= (\Delta-\lambda_1(G)) \sum_{w\in V(G)} x_w  < \sqrt{n}\, (\Delta-\lambda_1(G)).
\]
The last inequality follows from Cauchy--Schwarz inequality and $\|\bm{x}\|_2 = 1$.
Together with \autoref{thm:limit-upper-bound} gives the desired assertion.

Now we estimate $x_{\max}$. Let $P: v=u_0,u_1,\ldots,u_{\ell-1},u_{\ell}=u$ be a shortest 
path from $v$ to $u$. Applying \eqref{eq:Delta-rho-sum-form} and Cauchy--Schwarz inequality, 
we see 
\[ 
\Delta-\lambda_1(G) > \sum_{i=0}^{\ell-1} (x_{u_i} - x_{u_{i+1}})^2
\geq \frac{(x_v - x_u)^2}{\ell} > \frac{(x_v - x_u)^2}{n}.
\]
On the other hand, \autoref{thm:limit-upper-bound} implies that $\Delta - \lambda_1(G) = O(n^{-2})$,
from which and the above inequality we have $x_{\max} = O(n^{-1/2})$.
\end{proof}

With \autoref{thm:limit-upper-bound} and \autoref{lem:x-min-x-max-bound} in hand, we are ready 
to give an exact coefficient for the leading term of $\Delta-\lambda_1(n,\Delta)$ for $\Delta\in\{3,4\}$. 
Let $G\in\mathcal{G}(n,\Delta)$. Observe that $G\in\mathcal{H}_n^{(1)}$ for $\Delta=3$ and $G\in\mathcal{H}_n^{(2)}$ 
for $\Delta=4$. So we need to consider the asymptotic behavior of $\Delta-\lambda_1(G_{\Delta,k-1}^p\odot F)$,
where $p\in\{\Delta-1, \Delta-2\}$.

\begin{theorem}\label{thm:order-G-odot-F}
The following conclusions hold:
\begin{enumerate}
\item[$(1)$] Let $\Delta$ be odd and $G_{\Delta,k-1}^{\Delta-1}\odot F \in\mathcal{H}_n^{(1)}$. Then 
\[
\Delta - \lambda_1\big(G_{\Delta,k-1}^{\Delta-1}\odot F\big) = (1+o(1)) \frac{(\Delta-1)\pi^2}{4n^2}.
\]

\item[$(2)$] Let $\Delta$ be even and $G_{\Delta,k-1}^{\Delta-2}\odot F \in\mathcal{H}_n^{(2)}$. Then 
\[
\Delta - \lambda_1\big(G_{\Delta,k-1}^{\Delta-2}\odot F\big) = (1+o(1)) \frac{(\Delta-2)\pi^2}{2n^2}.
\]
\end{enumerate}
\end{theorem}

\begin{proof}
Referring to the proof of \autoref{thm:limit-upper-bound}, it is enough to give a matched lower bound 
on $\Delta - \lambda_1\big(G_{\Delta,k-1}^p\odot F\big)$ for $p\in\{\Delta-1, \Delta-2\}$. Furthermore, 
by the monotonicity of the spectral radius of graphs with respect to vertex addition, we only need to 
consider the graph shown in Fig.\,\ref{fig:label-G-odot-F}, whose components of Perron vector $\bm{x}$ 
are also labeled in Fig.\,\ref{fig:label-G-odot-F}. 

From \eqref{eq:Delta-rho-sum-form}, a routine computation gives rise to
\[
\Delta - \lambda_1\big(G_{\Delta,k-1}^p\odot F\big) > p\sum_{i=1}^{k-1} (z_i - a_i)^2 + 
p(\Delta - p) \sum_{i=1}^{k-1} (a_i - b_i)^2 + (\Delta - p) \sum_{i=1}^{k-1} (b_i - z_{i+1})^2.
\]  
Consider the function $f_i(x,y) := p(z_i - x)^2 + p(\Delta - p) (x - y)^2 + (\Delta - p) (y - z_{i+1})^2$.
One can check that $f_i(x,y)$ attaining minimum at 
\[
x = \frac{(p+1) z_i + (\Delta - p) z_{i+1}}{\Delta+1},~~
y = \frac{p z_i + (\Delta - p + 1) z_{i+1}}{\Delta + 1}.
\]
Putting all these together, we conclude that
\begin{equation}\label{eq:lower-bound-Delta-lambda-F}
\Delta - \lambda_1\big(G_{\Delta,k-1}^p\odot F\big) > \frac{p (\Delta - p)}{\Delta + 1} 
\sum_{i=1}^{k-1} (z_i - z_{i+1})^2.
\end{equation}
Below we shall give an estimation for the right-hand side of the above inequality.
By eigenvalue equations, we find that
\[
a_i\leq b_i \leq z_{i+1},~~i\in [k-1].
\]
Noting the fact that $\|\bm{x}\|_2=1$, we obtain that 
\begin{align*}
1 & = \sum_{i=1}^k z_i^2 + p \sum_{i=1}^{k-1} a_i^2 + (\Delta - p) \sum_{i=1}^{k-1} b_i^2 + 
\sum_{u\in V(F)\setminus V(G_{\Delta,k-1}^p)} x_u^2 \\
& \leq (\Delta + 1) \sum_{i=1}^k z_i^2 + \sum_{u\in V(F)\setminus V(G_{\Delta,k-1}^p)} x_u^2 \\
& = (\Delta + 1) \sum_{i=1}^k z_i^2 + O(n^{-1}).
\end{align*}
The last equality uses the fact that $x_{\max}=O(n^{-1/2})$. It follows that
\begin{equation}\label{eq:sum-z}
\sum_{i=1}^k z_i^2 \geq \frac{1}{\Delta + 1} - O(n^{-1}).
\end{equation}
For each $i\in [k]$, let $z_{-i} = z_i$, and denote $\bm{z} := (z_{-1},\ldots,z_{-k},z_k,\ldots,z_1)^{\mathrm{T}}$.
Consider the path $P_{2k}$ on $2k$ vertices, and notice that $z_{-1}=z_1=O(n^{-3/2})$ by \autoref{lem:x-min-x-max-bound}. 
According to \eqref{eq:lower-bound-Delta-lambda-F} we deduce that
\begin{align*}
\Delta - \lambda_1\big(G_{\Delta,k-1}^p\odot F\big) 
& > \frac{p (\Delta - p)}{2(\Delta + 1)} \sum_{\substack{i=-k,\\ i\notin\{-1,0\}}}^{k-1} (z_i - z_{i+1})^2 \\
& = \frac{p (\Delta - p)}{2(\Delta + 1)} \big(z_{-1}^2 + z_1^2 + \bm{z}^{\mathrm{T}} L(P_{2k}) \bm{z}\big) - O(n^{-3}) \\
& \geq \frac{p (\Delta - p)}{2(\Delta + 1)} \big(2 - \lambda_1(P_{2k})\big) \cdot\|\bm{z}\|_2^2 - O(n^{-3}).
\end{align*} 
The last inequality follows from \eqref{eq:Delta-rho}. Recall that $\lambda_1(P_{2k}) = 2\cos\pi/(2k+1)$ (see, e.g. \cite[p.9]{BrouwerHaemers2011}). 
It follows from \eqref{eq:sum-z} that
\begin{align*}
\Delta - \lambda_1\big(G_{\Delta,k-1}^p\odot F\big)  
& > \frac{p(\Delta - p) \pi^2}{(\Delta + 1)^2 (2k+1)^2} - O(n^{-3}) \\
& = \frac{p(\Delta - p) \pi^2}{4n^2} - O(n^{-3}),
\end{align*}
where the last equality use the fact $k=(n-\alpha)/(\Delta+1)$. Finally, we obtain the desired 
results by letting $p=\Delta-1$ for odd $\Delta$ and $p=\Delta-2$ for even $\Delta$, respectively.
\end{proof}

We immediately obtain the asymptotic behavior of $\Delta-\lambda_1(n,\Delta)$ for $\Delta\in\{3,4\}$.

\begin{theorem}\label{thm:asymptotic-behavior-Delta-3-4}
If $\Delta = 3$, then
\[
\lim_{n\to\infty} n^2 (3 - \lambda_1(n,3)) = \frac{\pi^2}{2}.
\]
If $\Delta = 4$, then 
\[
\lim_{n\to\infty} n^2 (4 - \lambda_1(n,4)) = \pi^2.
\]
\end{theorem}

\section{Concluding remarks and open problems}
\label{sec:concluding-remarks}

In this paper, we study the structural properties of extremal graphs in $\mathcal{G}(n,\Delta)$, which disprove 
\autoref{conj:limit} for $\Delta\geq 3$ and confirm \autoref{conj:degree-sequence} when $\Delta\in\{3,4\}$.
Although \autoref{conj:limit} is not true as the inconsistent asymptotic behavior of $\Delta-\lambda_1(n,\Delta)$ 
for $\Delta$ that have different parity, we can still ask what is the exact leading term of $\Delta-\lambda_1(n,\Delta)$. 
\autoref{thm:asymptotic-behavior-Delta-3-4} answers this problem for $\Delta=3$ and $\Delta=4$. However, 
for general $\Delta$, it seems to be difficult to solve it. Based on some numerical experiments and heuristic 
arguments, we present the following conjecture.

\begin{conjecture}\label{conj:concluding-conj-1}
Let $G$ be a graph attaining the maximum spectral radius among all connected 
nonregular graph with $n$ vertices and maximum degree $\Delta$. Then
the limit of $n^2 (\Delta-\lambda_1(G))$ always exists. Furthermore,

\begin{enumerate}
\item[$(1)$] if $\Delta$ is odd, then 
\[
\lim_{n\to\infty} \frac{n^2(\Delta-\lambda_1(G))}{\Delta-1} = \frac{\pi^2}{4}.
\]

\item[$(2)$] if $\Delta$ $(\Delta>2)$ is even, then 
\[
\lim_{n\to\infty} \frac{n^2 (\Delta-\lambda_1(G))}{\Delta-2} = \frac{\pi^2}{2}.
\]  
\end{enumerate}
\end{conjecture}

By analyzing the structural properties of the extremal graphs, we confirm \autoref{conj:degree-sequence}
for small $\Delta$. However, we cannot expect an affirmative answer to \autoref{conj:degree-sequence} for
general $\Delta$. In fact, we have some evidence for the following speculation.

\begin{conjecture}
Let $G$ be a graph attaining the maximum spectral radius among all connected 
nonregular graph with $n$ vertices and maximum degree $\Delta$ $(\Delta\geq 3)$.
For each fixed $\Delta$ and sufficiently large $n$, $G$ has degree sequence $(\Delta,\ldots,\Delta,\delta)$, 
where
\[
\delta =
\begin{cases}
\Delta-1, & \Delta\ \text{is odd},\ n\ \text{is odd}, \\
1, & \Delta\ \text{is odd},\ n\ \text{is even}, \\
\Delta-2, & \Delta\ \text{is even}.
\end{cases}
\]
\end{conjecture}

% \begin{conjecture}
% Let $\Delta$ be a fixed integer and $\alpha$ be defined as before. Suppose that $G$ is a graph attaining 
% the maximum spectral radius among all connected nonregular graph with $n$ vertices and maximum degree 
% $\Delta$ $(\Delta\geq 3)$.

% \begin{enumerate}
% \item[$(1)$] if $\Delta$ is even, then the induced subgraph of the first $\Delta+1+\alpha$ vertices in $G$ has
% the maximum spectral radius among all connected graphs on $\Delta+1+\alpha$ vertices with degree sequence $\Delta,\ldots,\Delta,\Delta-2$.

% \item[$(2)$] if $\Delta$ is odd, then the induced subgraph of the first $\Delta+1+\alpha$ $($or $\Delta+\alpha$, depending 
% on the parity of $n$$)$ vertices in $G$ has the maximum spectral radius among all connected graphs with degree sequence $\Delta,\ldots,\Delta,\Delta-1$.
% \end{enumerate}
% \end{conjecture}

Finally, we remark that for a connected nonregular graph $G$ of order $n$ with diameter $D$, maximum 
degree $\Delta$ and minimum degree $\delta$, Cioab\u a, Gregory and Nikiforov \cite{CioabaGregoryNikiforov2007}
present a lower bound of $\Delta-\lambda_1(G)$ in terms of the diameter $D$, which implies
$\Delta-\lambda_1(G)>1/(n(D+1))$. They also suggest that 
\[
\Delta - \lambda_1(G) > \frac{\sqrt{\Delta - \delta}}{nD}.
\]
The following example shows that the above inequality is not true in general.

\begin{example}
Let $\Delta$ be odd and $n=k(\Delta+1)+\alpha$ be even. Let $F\in\mathcal{T}_{\Delta+\alpha}^{(1)}$ and
$G=G_{\Delta, k-1}^{\Delta-1}\odot F$ be a graph in $\mathcal{H}_n^{(1)}$ (see Fig.\,\ref{fig:label-G-odot-F-in-example}).
By \autoref{thm:order-G-odot-F}, 
\[
\Delta - \lambda_1(G) = (1+o(1)) \frac{(\Delta - 1)\pi^2}{4n^2}.
\] 
Obviously, the diameter $D$ of $G$ is at most
\[
3(k - 1) + (\Delta + \alpha) \leq 3k + 2\Delta - 2
= \frac{3n}{\Delta + 1} + O(1).
\]
One can check that $\Delta - \lambda_1(G) < \sqrt{\Delta - \delta}/(nD)$ when $\Delta\geq 53$ and sufficiently large $n$.
\end{example}

\begin{figure}[htbp]
\centering
\begin{tikzpicture}[scale=0.9]
\coordinate (u1) at (0,0);
\coordinate (u2) at (0.9,0);
\coordinate (u3) at (2.3,0);
\coordinate (u4) at (3.2,0);
\coordinate (u5) at (4.6,0);
\coordinate (u6) at (6.2,0);
\coordinate (u7) at (7.6,0);
\coordinate (u8) at (8.5,0);

\coordinate (v1) at (1.6,0.5);
\coordinate (w1) at (1.6,-0.5);
\coordinate (v2) at (3.9,0.5);
\coordinate (w2) at (3.9,-0.5);
\coordinate (v3) at (6.9,0.5);
\coordinate (w3) at (6.9,-0.5);

\foreach \i in {1.6,3.9,6.9}
{
\filldraw[fill=gray!40, draw=gray!40] (\i,0) ellipse [x radius=0.45,y radius=0.7];
\node at (\i,0.15) {$\vdots$};
}

\fill[gray!40] (8.5,0) -- (9,0.54) -- (9,-0.54) -- cycle;
\fill[gray!40] (9.45,0) circle (0.7);

\node at (5.4,0) {$\cdots\cdot$};
\node at (9.45,0) {$F$};

\draw (u1) -- (u2) -- (v1) -- (u3) -- (w1) -- (u2);
\draw (u3) -- (u4) -- (v2) -- (u5) -- (w2) -- (u4);
\draw (u6) -- (v3) -- (u7) -- (w3) -- (u6);
\draw (u5) -- (4.8,0);
\draw (u6) -- (6,0);
\draw (u7) -- (u8);

\foreach \i in {u1,u2,u3,u4,u5,u6,u7,u8,v1,v2,v3,v3,w1,w2,w3}
\filldraw (\i) circle (0.07);
\end{tikzpicture}
\caption{A graph $G_{\Delta,k-1}^{\Delta-1}\odot F$ in $\mathcal{H}_n^{(1)}$. Each shadow except $F$ 
represents the complete graph $K_{\Delta-1}$}
\label{fig:label-G-odot-F-in-example}
\end{figure}
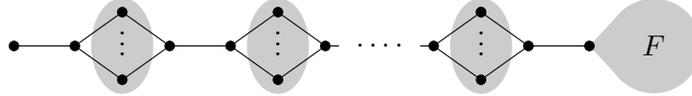

It is natural to ask what is the best constant $c$ such that $\Delta-\lambda_1(G) > c/(nD)$ for all 
connected nonregular graphs with maximum degree $\Delta$. Cioab\u a's bound \cite{Cioaba2007} implies 
that $c\geq 1$. In \cite{CioabaGregoryNikiforov2007}, the authors find a small graph, which forces 
$c < 1.355$ for small $n$. Cioab\u a \cite{Cioaba2007} also comments that there are infinite families 
of nonregular graphs with maximum degree $\Delta$ such that $\Delta - \lambda_1(G) \leq c/(nD)$. 
In \cite{CioabaGregoryNikiforov2007} the authors describe such a family with $c=4\pi^2$. Liu, Shen and 
Wang \cite{LiuShenWang2007} find an infinite family with $c=3\pi^2$. In this paper, we find an infinite 
family with $c=3\pi^2/4$.

\section*{Acknowledgements}

The author is grateful to Haiying Shan for providing help for the code.
The author is supported by the National Natural Science Foundation of China (No. 12001370).

\appendix

\section{Appendix: Proofs of \autoref{lem:forbidden-M2}, \autoref{lem:forbidden-M3} 
and \autoref{lem:forbidden-M5} -- \autoref{lem:forbidden-M7}}
\label{appendix}

\noindent {\bfseries Proof of \autoref{lem:forbidden-M2}:}
We prove this lemma by contradiction. Suppose on the contrary that $G$ contains $M_2$ with 
the components of $\bm{x}$ as depicted in Fig.\,\ref{fig:M1-M2}\,\subref{subfig:M2}. Set 
$\lambda:=\lambda_1(G)$. By eigenvalue equations we have
\[
b = \frac{\lambda-2}{2} a,~~
c = \frac{\lambda^2-2\lambda-6}{2} a,~~
d = \frac{\lambda^3-3\lambda^2-5\lambda+8}{4} a.
\]

Let $\widetilde{G}$ be the graph obtained by replacing $M_2$ with $\widetilde{M}_2$. We define 
a vector $\bm{y}$ on $V(\widetilde{G})$ such that its components on $\widetilde{M}_2$ are as given in 
Fig.\,\ref{fig:M1-M2}\,\subref{subfig:tilde-M2}, and on the rest of vertices agree with $\bm{x}$. 
Here, $x$, $y$, $z$ are given by 
\begin{align*}
x & = \frac{a}{4} \big((1-\sqrt{2})\lambda^3 + (4\sqrt{2}-3)\lambda^2 + (4\sqrt{2}-\sqrt{6}-5)\lambda -16\sqrt{2}+4\sqrt{6}+8\big), \\
y & = \frac{a}{4} \big((1-\sqrt{2})\lambda^3 + (4\sqrt{2}-3)\lambda^2 + (4\sqrt{2}-5)\lambda - 16\sqrt{2}+8\big), \\
z & = \frac{a}{4} \big((1-\sqrt{2})\lambda^3 + (5\sqrt{2}-3)\lambda^2 + (\sqrt{2}-5)\lambda - 20\sqrt{2}+8\big).
\end{align*}
A routine computation gives rise to
\begin{equation}\label{eq:square-x-y-2}
\sum_{uv\in E(G)} (x_u-x_v)^2 = \sum_{uv\in E(\widetilde{G})} (y_u-y_v)^2.
\end{equation}
In view of \eqref{eq:Delta-rho} and \eqref{eq:square-x-y-2} we see
\[
4-\lambda_1(\widetilde{G}) \leq \frac{4 -\lambda}{\|\bm{y}\|_2^2}.
\]
On the other hand, we have 
\begin{equation}\label{eq:appendix-1}
\|\bm{y}\|_2^2 = 1+ 4x^2+2y^2+2z^2 - (3a^2+2b^2+2c^2+d^2) := 1 + g(\lambda) a^2,
\end{equation}
where $g(x)$ is given by 
\begin{align*}
g(x) & = \frac{23-16\sqrt{2}}{16} x^6 + \frac{58\sqrt{2}-89}{8} x^5 + 
\Big(\sqrt{3}-\frac{9\sqrt{2}}{2}-\frac{\sqrt{6}}{2}+\frac{173}{16}\Big) x^4 \\
& ~~~ + \Big(\frac{7\sqrt{6}}{2}-8\sqrt{3}-56\sqrt{2}+\frac{661}{8}\Big) x^3 
+ \Big(\frac{275\sqrt{2}}{4}+12\sqrt{3}-\frac{7\sqrt{6}}{2}-\frac{2221}{16}\Big) x^2 \\
& ~~~ + \Big(111\sqrt{2}+32\sqrt{3}-14\sqrt{6}-163\Big) x - 136\sqrt{2}-64\sqrt{3}+16\sqrt{6}+321.
\end{align*}
It can be checked that $g(0)>0$, $g(3)<0$, $g(3.7)>0$, $g(5)<0$ and $g(32)>0$. Hence $g(x)=0$ 
has at least four positive roots. On the other hand, by Descartes's Rule of Signs, we obtain 
$g(x) =0$ has at most four positive roots. Hence, $g(x)$ has the same sign in interval $(3.7,4)$. 
Since $g(3.7)>0$ and $g(4)=0$, we see $g(\lambda)>0$, and therefore $\|\bm{y}\|_2>1$. This 
yields that $\lambda_1(\widetilde{G}) > \lambda$, a contradiction completing the proof 
of \autoref{lem:forbidden-M2}.
\hfill\ensuremath{\Box}
\par\vspace{3mm}

\noindent {\bfseries Proof of \autoref{lem:forbidden-M3}:}
For a contradiction suppose that $G$ contains $M_3$ with the components of $\bm{x}$ as depicted in 
Fig.\,\ref{fig:M3-M4}\,\subref{subfig:M3}. Set $\lambda:=\lambda_1(G)$. By eigenvalue equations we have
\[
b = (\lambda-3) a,~~
c = \frac{\lambda^2-3\lambda-2}{2} a.
\]

Now we replace $M_3$ by $\widetilde{M}_3$ to obtain $\widetilde{G}$. We define a vector $\bm{y}$
on $V(\widetilde{G})$ such that its components on $\widetilde{M}_3$ are as given in 
Fig.\,\ref{fig:M3-M4}\,\subref{subfig:tilde-M3}, and on the rest of vertices agree with $\bm{x}$. 
Here, $x$, $y$, $z$ are given by 
\begin{align*}
x & = \frac{\sqrt{2}-1}{2} (-\lambda^2+5\lambda+2\sqrt{2}-2) a, \\
y & = \frac{\sqrt{2}-1}{2} (-\lambda^2+(2\sqrt{2}+7)\lambda-10-6\sqrt{2}) a.
\end{align*}
By simple algebra we have
\[
(x-y)^2 = (a-b)^2,~~
(y-c)^2 = 2(b-c)^2,
\]
and therefore $\bm{x}^{\mathrm{T}} L(G) \bm{x} = \bm{y}^{\mathrm{T}} L(\widetilde{G}) \bm{y}$.
In view of \eqref{eq:Delta-rho} we see
\[
4-\lambda_1(\widetilde{G}) \leq \frac{4 -\lambda}{\|\bm{y}\|_2^2}
\]
On the other hand, we have 
\[
\|\bm{y}\|_2^2 = 1+ 4x^2+3y^2- (4a^2+2b^2+c^2) := 1 + h(\lambda) a^2,
\]
where $h(x)$ is given by 
\[
h(x) = \frac{10-7\sqrt{2}}{2} \lambda^4 + (32\sqrt{2}-48) \lambda^3 + 
\frac{306-191\sqrt{2}}{2} \lambda^2 + (103\sqrt{2}-182) \lambda +72-36\sqrt{2}.
\]
One can check that $h(0)>0$, $h(2)<0$, $h(3)>0$, $h(5)<0$ and $h(50)>0$. Noting that $h(4)=0$, 
we have $h(\lambda)>0$, and therefore $\|\bm{y}\|_2>1$. This yields that $\lambda_1(\widetilde{G}) > \lambda$, 
a contradiction completing the proof of \autoref{lem:forbidden-M3}.
\hfill\ensuremath{\Box}
\par\vspace{3mm}

\noindent {\bfseries Proof of \autoref{lem:forbidden-M5}:}
For a contradiction suppose that $G$ contains $M_5$ with the components of $\bm{x}$ as depicted 
in Fig.\,\ref{fig:M5}\,\subref{subfig:M5}. Set $\lambda:=\lambda_1(G)$ for short. By eigenvalue 
equations we see
\begin{align*}
b & = \frac{2(\lambda-3)(\lambda+1) a + 4h}{\lambda^3 - 2\lambda^2 - 5\lambda + 2}, \\[1mm]
c & = \frac{2\big((\lambda-1)a + \lambda h\big)}{\lambda^3 - 2\lambda^2 - 5\lambda + 2}, \\[1mm]
d & = \frac{4a + (\lambda - 2)(\lambda + 1)h}{\lambda^3 - 2\lambda^2 - 5\lambda +2}.
\end{align*}
Now we replace $M_5$ by $\widetilde{M}_5$ to obtain $\widetilde{G}$. We define a vector on 
$V(\widetilde{G})$ such that its components on $\widetilde{M}_5$ are as given in Fig.\,\ref{fig:M5}\,\subref{subfig:tilde-M5},
and on the remaining vertices agree with $\bm{x}$. Here, $x$, $y$, $z$ are given by 
\begin{align*}
x & = \frac{(\lambda - 2)(\lambda + 1)a + 4h}{\lambda^3 - 2\lambda^2 - 5\lambda +2}, \\[1mm]
y & = \frac{2\big(\lambda a + (\lambda-1)h\big)}{\lambda^3 - 2\lambda^2 - 5\lambda + 2}, \\[1mm]
z & = \frac{4a + 2(\lambda-3)(\lambda+1) h}{\lambda^3 - 2\lambda^2 - 5\lambda + 2}.
\end{align*}
By \eqref{eq:Delta-rho} and tedious calculation, we obtain that $4-\lambda_1(\widetilde{G}) < 4 - \lambda$, 
a contradiction completing the proof of \autoref{lem:forbidden-M5}.
\hfill\ensuremath{\Box}
\par\vspace{3mm}

\noindent {\bfseries Proof of \autoref{lem:forbidden-M6}:}
For a contradiction suppose that $G$ contains $M_6$ as an induced subgraph.
We can assume that for the two leftmost vertices, all their neighbors on their
left have the same component, denoted by $x$. Set $\lambda:=\lambda_1(G)$ for short. 
By eigenvalue equations, we have
\begin{align*}
a & = \frac{2(\lambda^3 - \lambda^2 - 4\lambda) x + 2(\lambda + 2) y}{\lambda^4 - 2\lambda^3 - 5\lambda^2 + 6\lambda + 4}, \\[1mm]
b & = \frac{2(2\lambda^2 - 2\lambda - 4) x + 2(\lambda^2 + \lambda -2) y}{\lambda^4 - 2\lambda^3 - 5\lambda^2 + 6\lambda + 4}, \\[1mm]
c & = \frac{4\lambda x + (\lambda^3 + \lambda^2 - 4\lambda - 4) y}{\lambda^4 - 2\lambda^3 - 5\lambda^2 + 6\lambda + 4}, \\[1mm]
d & = \frac{8x + 2(\lambda^3 - \lambda^2 - 4\lambda + 2) y}{\lambda^4 - 2\lambda^3 - 5\lambda^2 + 6\lambda + 4}.
\end{align*}
Now we replace $M_6$ by $\widetilde{M}_6$ to obtain $\widetilde{G}$. We define a vector on 
$V(\widetilde{G})$ such that its components on $\widetilde{M}_6$ are as given in Fig.\,\ref{fig:M6-M7}\,\subref{subfig:tilde-M6},
and on the rest of vertices agree with $\bm{x}$. Here, $\alpha_1$, $\alpha_2$, $\alpha_3$ and $\alpha_4$ are given by 
\begin{align*}
\alpha_1 & = \frac{2(\lambda^3 - \lambda^2 - 4\lambda + 2) x + 8y}{\lambda^4 - 2\lambda^3 - 5\lambda^2 + 6\lambda + 4}, \\[1mm]
\alpha_2 & = \frac{(\lambda^3 + \lambda^2 - 4\lambda - 4) x + 4\lambda y}{\lambda^4 - 2\lambda^3 - 5\lambda^2 + 6\lambda + 4}, \\[1mm]
\alpha_3 & = \frac{2(\lambda^2 + \lambda -2) x + 2(2\lambda^2 - 2\lambda - 4) y}{\lambda^4 - 2\lambda^3 - 5\lambda^2 + 6\lambda + 4}, \\[1mm]
\alpha_4 & = \frac{2(\lambda + 2) x + 2(\lambda^3 - \lambda^2 - 4\lambda) y}{\lambda^4 - 2\lambda^3 - 5\lambda^2 + 6\lambda + 4}.
\end{align*}
In view of \eqref{eq:Delta-rho}, we obtain that $4-\lambda_1(\widetilde{G}) < 4 - \lambda$ after tedious calculation, 
a contradiction completing the proof of \autoref{lem:forbidden-M6}.
\hfill\ensuremath{\Box}
\par\vspace{3mm}

\noindent {\bfseries Proof of \autoref{lem:forbidden-M7}:}
For a contradiction suppose that $G$ contains $M_7$ with the components of $\bm{x}$ as depicted 
in Fig.\,\ref{fig:M6-M7}\,\subref{subfig:M7}. Set $\lambda:=\lambda_1(G)$ for short. By eigenvalue 
equations we see
\begin{align*}
b & = \frac{2(\lambda^2 - \lambda - 2) a + 2(\lambda + 2) h}{\lambda (\lambda^2 - \lambda - 4)}, \\[1mm]
c & = \frac{2a + (\lambda + 2) h}{\lambda^2 - \lambda - 4}, \\[1mm]
d & = \frac{4a + 2(\lambda^2 - 2) h}{\lambda (\lambda^2 - \lambda - 4)}. 
\end{align*}
Now we replace $M_7$ by $\widetilde{M}_7$ to obtain $\widetilde{G}$, and define a vector on 
$V(\widetilde{G})$ such that its components on $\widetilde{M}_7$ are as given in Fig.\,\ref{fig:M6-M7}\,\subref{subfig:tilde-M7}
and on the rest of vertices agree with $\bm{x}$. Here, $x$, $y$ and $z$ are given by 
\begin{align*}
x & = \frac{2(\lambda^2 - 2) a + 4h}{\lambda (\lambda^2 - \lambda - 4)}, \\[1mm]
y & = \frac{(\lambda + 2) a + 2h}{\lambda^2 - \lambda - 4}, \\[1mm]
z & = \frac{2(\lambda + 2) a + 2(\lambda^2 - \lambda - 2) h}{\lambda (\lambda^2 - \lambda - 4)}. 
\end{align*}
By \eqref{eq:Delta-rho} and tedious calculation, we obtain that $4-\lambda_1(\widetilde{G}) < 4 - \lambda$, 
a contradiction completing the proof of \autoref{lem:forbidden-M7}.
\hfill\ensuremath{\Box}

\end{document}